\documentclass[reqno, oneside, 12pt]{amsart}
\usepackage[letterpaper]{geometry}
\usepackage{amsmath,amssymb}
\usepackage{amsthm}
\usepackage{thmtools}
\usepackage{mathtools}
\usepackage{dsfont}
\usepackage{color}
\usepackage{enumitem}
\usepackage{cases}
\usepackage{multirow}
\usepackage{mathrsfs}
\usepackage{tikz}

\usepackage[hypertexnames=false]{hyperref}

\newcommand{\ee}{{\mathbf{e}}}

\newcommand{\Z}{{\mathbb{Z}}}
\newcommand{\Q}{{\mathbb{Q}}}
\newcommand{\R}{{\mathbb{R}}}
\newcommand{\C}{{\mathbb{C}}}

\newcommand{\HH}{{\mathbb{H}}}

\newcommand{\Mform}{{\mathcal{M}}}

\newcommand{\Lform}{{\mathcal{L}}}
\newcommand{\LEigenform}{{\widetilde{\mathcal{L}}}}

\renewcommand{\phi}{\varphi}

\DeclareMathOperator{\re}{Re}
\DeclareMathOperator{\im}{Im}
\DeclareMathOperator*{\Res}{Res}
\DeclareMathOperator{\SL}{SL}

\DeclareMathOperator{\rad}{rad}

\newcommand{\ep}{\varepsilon}

\newcommand{\defeq}{\vcentcolon=}
\newcommand{\floor}[1]{\left\lfloor #1 \right\rfloor}

\newcommand{\ceil}[1]{\left\lceil #1 \right\rceil}

\renewcommand{\(}{\left(}
\renewcommand{\)}{\right)}

\newcommand{\Mod}[1]{\ (\mathrm{mod}\ #1)}
\DeclareMathOperator{\sgn}{sgn}

\newcommand{\mB}{{\mathcal{B}}}

\newcommand{\mS}{{\mathcal{S}}}

\newtheorem{theorem}{Theorem}[section]

\newtheorem{lemma}[theorem]{Lemma}
\newtheorem{corollary}[theorem]{Corollary}
\newtheorem{proposition}[theorem]{Proposition}
\newtheorem{conjecture}[theorem]{Conjecture}
\newtheorem{definition}[theorem]{Definition}

\theoremstyle{remark}
\newtheorem*{remark}{Remark}

\numberwithin{equation}{section}





\author{Danylo Radchenko}
\address{University of Lille, CNRS, Laboratoire Paul Painlev\'e, F-59655 Villeneuve d'Ascq, France}
\email{danradchenko@gmail.com}
\author{Qihang Sun}
\address{University of Lille, CNRS, Laboratoire Paul Painlev\'e, F-59655 Villeneuve d'Ascq, France}
\email{qihang.sun@univ-lille.fr}

\title[Fourier interpolation in dimensions 3 and 4]{Fourier interpolation in dimensions 3 and 4\\ and real-variable Kloosterman sums}

\date{\today}

\begin{document}

\begin{abstract}
    We give a construction of radial Fourier interpolation formulas in dimensions 3 and 4 using Maass--Poincar\'e type series. As a corollary we obtain explicit formulas for the basis functions of these interpolation formulas in terms of what we call real-variable Kloosterman sums, which were previously introduced by Stoller. We also improve the bounds on the corresponding basis functions $a_{n,d}(x)$, $d=3,4$, for fixed $x$, in terms of the index $n$.
\end{abstract}

\maketitle

\setcounter{tocdepth}{1}
\tableofcontents

\section{Introduction}
In~\cite[Theorem~1]{RadchenkoViazovska2019} it was proved that any even Schwartz function $f\colon\R\to\C$ is uniquely determined by the values $f(\sqrt{n})$, $\widehat{f}(\sqrt{n})$, $n\ge0$, where 
    \[\widehat{f}(\xi) = \int_{\R}f(x)e^{-2\pi i \xi x}dx\]
is the Fourier transform of $f$. More precisely, there is a linear interpolation formula
    \[
    f(x) = \sum_{n\ge0} a_n(x)f(\sqrt{n}) + \sum_{n\ge0}\widehat{a_n}(x)\widehat{f}(\sqrt{n}),
    \]
where $a_n(x)$, $\widehat{a_n}(x)$ are certain Schwartz functions, explicitly defined as contour integrals of weakly holomorphic modular forms of weight $3/2$. Similar interpolation formulas exist also for radial Schwartz functions in higher dimensions: from~\cite[Theorem~3.1]{BondarenkoRadchenkoSeip2022} it follows that there exist radial Schwartz functions $a_{d,n}$, $\widetilde{a}_{d,n}$ such that
    \begin{equation} \label{eq:radialmain}
    f(x) = \sum_{n\ge 0} a_{d,n}(x)f(\sqrt{n}) + \sum_{n\ge 0}\widetilde{a}_{d,n}(x)\widehat{f}(\sqrt{n}),
    \end{equation}
holds for all $f\in\mathcal{S}_{\rad}(\R^d)$, and we abuse notation, denoting $g(r)=g(r,0,\dots,0)$ for any radial function $g$ on $\R^d$ and $r\in \R$. Interpolation formulas of similar kind, with $\sqrt{n}$ replaced by $\sqrt{2n}$ and including the values of derivatives of $f$ and $\widehat{f}$, have played an important role in sphere packing and energy minimization problems~\cite{Viazovska}, \cite{CKMRV17}, \cite{CKMRV22}, where they were used in conjunction with linear programming techniques to prove optimality of the $E_8$~lattice and the Leech lattice. For more context and recent developments on the analytic aspects of similar interpolation formulas see \cite{Ramos-Sousa, Ramos-Sousa-2, Ramos-Stoller, Adve, Kulikov_Nazarov_Sodin_2025}.

The coefficients $a_{d,n}(x)$ and $\widetilde{a}_{d,n}(x)$\footnote{The choice of coefficients for which~\eqref{eq:radialmain} holds is not unique. For a fixed $x$, the space of all interpolation formulas of the type~\eqref{eq:radialmain} forms a coset for the $\C$-vector space of all linear relations between $f(\sqrt{n})$ and $\widehat{f}(\sqrt{n})$, $n\ge0$, which is finite-dimensional and isomorphic to $M_{d/2}(\Gamma(2),\nu_{\Theta}^d)$. In~\cite{BondarenkoRadchenkoSeip2022} a particular choice of $a_{d,n}(x)$, $\widetilde{a}_{d,n}(x)$ is given, uniquely specified by certain vanishing conditions. Henceforth, the notation $a_{d,n}(x)$, $\widetilde{a}_{d,n}(x)$ will refer to this particular choice (for a more precise statement, see Theorem~\ref{theoremBonRadSeip-2022}).} from~\eqref{eq:radialmain} can be represented as contour integrals involving weakly holomorphic modular forms of weight $2-d/2$. For instance, in dimension~1 we have
    \[a_{1,0}(x)=\frac{1}{4}\int_{-1}^{1}\Theta^3(z)e^{\pi i zx^2}dz,\]
where the integral is taken over a semicircle in the upper half plane, and $\Theta(z)=\sum_{n\in\Z}e^{\pi i n^2z}$ is the usual theta function. Although this construction is explicit, obtaining precise estimates for $a_{d,n}(x)$ as a function of $n$ is not easy. Having such estimates is of interest, for example, since they allow one to extend the class of functions $f$ for which~\eqref{eq:radialmain} holds. In~\cite{RadchenkoViazovska2019} it was proved that $a_{1,n}(x), \widetilde{a}_{1,n}(x)$ are bounded by $O(n^2)$, uniformly in $x\in\R$, which was used to show that~\eqref{eq:radialmain} holds if $f(x)$ and $\widehat{f}(x)$ are both $O(|x|^{-13})$ at infinity. In~\cite{BondarenkoRadchenkoSeip2022} the estimates were improved to 
    \begin{equation} \label{eq:heckebound}
    a_{d,n}(x),\widetilde{a}_{d,n}(x) = 
        \begin{cases}
        O_x(n^{d/4+o(1)}),\qquad d \le 4,\\
        O_x(n^{d/2-1}),\;\;\;\qquad d > 4,
        \end{cases}
    \end{equation}
which roughly correspond to the classical Hecke bound for the coefficients of cusp forms and estimates for the coefficients of Eisenstein series, respectively. 

Getting anything more precise than the bounds in~\eqref{eq:heckebound} appears to be challenging. This motivates the search for different formulas of the basis functions $a_{d,n}(x)$, $\widetilde{a}_{d,n}(x)$. In \cite{StollerInterpolation} Stoller proposed an approach based on generalized Poincar\'e series. His construction works for $d\ge 5$ and produces coefficients $a_{d,n}^{\mathrm{Sto}}(x)$ and $\widetilde{a}_{d,n}^{\mathrm{Sto}}(x)$ that are easier to analyze for $n\to\infty$, but they end up not being Schwartz functions of~$x$, and in general they also differ from the basis functions in~\eqref{eq:radialmain}. 

The main goal of this paper is to give a construction of interpolation formulas of type~\eqref{eq:radialmain} in dimensions $d=3$ and $d=4$ based on (generalized) Maass--Poincar\'e series, and to explicitly relate these to the formulas from~\cite{BondarenkoRadchenkoSeip2022}. As one of the applications we obtain an explicit infinite sum representation for the functions $a_{d,n}$ and $\widetilde{a}_{d,n}$, in dimensions $d=3,4$, involving what we call \emph{real-variable Kloosterman sums}, as well as Hurwitz class numbers (for $d=3$). As another application, we are able to go beyond the bounds~\eqref{eq:heckebound} for $d=3,4$.

As is explained in~\cite[\S6]{RadchenkoViazovska2019}, the interpolation formula~\eqref{eq:radialmain} is equivalent to a functional equation for the generating functions of the sequences $a_{d,n}(x)$, $\widetilde{a}_{d,n}(x)$, $n\ge0$.
Our starting point is the following precise version of this claim.

\begin{theorem}[{Corollary of \cite[Theorem~3.1]{BondarenkoRadchenkoSeip2022}}]
	\label{theoremBonRadSeip-2022}
	Let $\HH$ be the complex upper-half plane. 

    \begin{itemize}
        \item[(1)]For any $r\ge0$ there exists a unique pair of $2$-periodic analytic functions 
        \[\mathcal{F}_2(\,\cdot\,;r), \widetilde{\mathcal{F}_2}(\,\cdot\,;r)\colon\HH\rightarrow \C\]
        of moderate growth with Fourier expansions of the form 
	\[\mathcal{F}_{2}(\tau;r)=\sum_{n\geq 1} a_{4,n}(r)e^{\pi i n \tau},\quad \widetilde{\mathcal{F}}_{2}(\tau;r)=\sum_{n\geq 1} \widetilde{a}_{4,n}(r)e^{\pi i n \tau},\quad \tau\in \HH,\]
	such that
	\[\mathcal{F}_2(\tau;r)+(\tau/i)^{-2} \mathcal{F}_2(-1/\tau;r)=
    e^{\pi i r^2 \tau}. \]
    \item[(2)]For any $r\ge0$ there exists a unique pair of $2$-periodic analytic function \[\mathcal{F}_{3/2}(\,\cdot\,;r), \widetilde{\mathcal{F}}_{3/2}(\,\cdot\,;r):\HH\rightarrow \C\]
    of moderate growth with Fourier expansions of the form 
	\[\mathcal{F}_{3/2}(\tau;r)=\sum_{n\geq 0} a_{3,n}(r)e^{\pi i n \tau},\quad \widetilde{\mathcal{F}}_{3/2}(\tau;r)=\sum_{n\geq 0} \widetilde{a}_{3,n}(r)e^{\pi i n \tau},\quad \tau\in \HH,\]
	and $a_{3,0}(r) = \widetilde{a}_{3,0}(r)$, such that
	\[\mathcal{F}_{3/2}(\tau;r)
    +(\tau/i)^{-3/2} \widetilde{\mathcal{F}}_{3/2}(-1/\tau;r)
    =e^{\pi i r^2 \tau}.\]
    \end{itemize}
\end{theorem}

Here the condition of moderate growth simply means polynomial growth (in $n$) of the coefficients $a_{4,n}(r), \widetilde{a}_{4,n}(r)$ and $a_{3,n}(r), \widetilde{a}_{3,n}(r)$. The functions appearing in~\eqref{eq:radialmain} for $d=3,4$ are precisely the coefficients $a_{d,n}(r), \widetilde{a}_{d,n}(r)$ that are uniquely determined by the conditions of Theorem~\ref{theoremBonRadSeip-2022} (by convention we set $a_{4,0}=\widetilde{a}_{4,0}=0$ and $a_{3,0}=\widetilde{a}_{3,0}$).

In this paper we will construct the analytic functions $\mathcal{F}_{2}$ and $\mathcal{F}_{3/2}$ from Theorem~\ref{theoremBonRadSeip-2022} via Maass--Poincar\'e-type series. Specifically, we will construct a Maass--Poincar\'e-type series with a spectral parameter $s\in \C$, weight $k\in\{\frac 32,2\}$ and a real parameter $r\in \R$, and prove that its Fourier expansion converges for $s=\frac k2$. Moreover, we will show that these series can be adjusted to satisfy the conditions of Theorem~\ref{theoremBonRadSeip-2022}. This will give us an explicit identity for $a_{d,n}$, $d=3,4$, in terms of sums of real-variable Kloosterman sums. We now formulate these results in more detail.

In what follows we denote by $J_\alpha(x)$ the $J$-Bessel function \cite[(10.2.2)]{dlmf}. In dimension $d=4$, we have the following result.

\begin{theorem}\label{theoremMainDim4}
	Let two sequences of entire functions $b_{4,n}, \widetilde b_{4,n}$, $n\ge1$, be given by
	\begin{align*}
		b_{4,n}(r)&=B_{4,n}(r)+ \frac{8\sin(\pi r^2)}{\pi r^2}\(\sigma_1(\tfrac n2)-(-1)^n \sigma_1(n)\),\\
		\widetilde{b}_{4,n}(r)&=\widetilde{B}_{4,n}(r)+ \frac{8\sin(\pi r^2)}{\pi r^2}\(\sigma_1(\tfrac n2)-(-1)^n \sigma_1(n)\),
	\end{align*}
	where $\sigma_1(n)=\sum_{d|n} d$ is the sum of divisors of $n$, $\sigma_1(x)=0$ if $x\notin \Z$. Here $B_{4,n}(r)$ and $\widetilde{B}_{4,n}(r)$ are defined as
	\begin{align*}
		B_{4,n}(r)=\pi\(\frac{n}{r^2}\)^{\frac 12} \sum_{2|c>0}\frac{K(r^2,n,c,\nu_{\Theta}^4)}{c} J_1\(\frac{2\pi |r|\sqrt n}c\),\\
		\widetilde{B}_{4,n}(r)=-\pi\(\frac{n}{r^2}\)^{\frac 12} \sum_{2\nmid \widetilde{c}>0}\frac{\widetilde{K}(r^2,n,\widetilde{c},\nu_{\Theta}^4)}{\widetilde{c}} J_1\(\frac{2\pi |r|\sqrt n}{\widetilde{c}} \),
	\end{align*}
where the real-variable Kloosterman sums $K$ and $\widetilde{K}$ are defined at \eqref{eqReal-Val-Kloosterman-Sums} and \eqref{eqReal-Val-Kloosterman-Sums-tilde}, respectively. Then we have the following. 
	\begin{itemize}
		\item [(1)]
    For every $f\in \mS_{\rad}(\R^4)$ and every $x\in \R^4$, we have
	\[f(x)=\sum_{n=1}^\infty f(\sqrt n) b_{4,n}(|x|)+\sum_{n=1}^\infty \widehat f(\sqrt n) \widetilde b_{4,n}(|x|). \]
	
	\item[(2)] We have the following estimates for $B_{4,n}(r)$ and $\widetilde{B}_{4,n}(r)$: 
	\[|B_{4,n}(r)|, |\widetilde{B}_{4,n}(r)| \ll_\ep \left\{\begin{array}{ll}
	     n^{1+\ep} & \text{if }r^2\leq 1/n, \\
	     (|r|^{-1}+|r|^{-\frac 12+\ep}) n^{\frac34+\ep}& \text{if }r^2\geq 1/n.
	\end{array}\right.   \]

    \item[(3)] Let $\nu_2(n)$ be the 2-adic valuation of~$n$. The values $B_{4,n}(0)$ and $\widetilde{B}_{4,n}(0)$ are given by
    \begin{equation*}
        B_{4,n}(0)=8\sigma_1(n)\cdot \left\{\begin{array}{ll}
            \displaystyle\frac{2^{\nu_2(n)}-3}{2^{\nu_2(n)+1}-1}, & 2|n,  \vspace{2px}\\
            0,& 2\nmid n, 
        \end{array}\right. \quad 
        \widetilde{B}_{4,n}(0)=8\sigma_1(n)\cdot \frac{2^{\nu_2(n)}}{2^{\nu_2(n)+1}-1}.  
    \end{equation*}
	\end{itemize}
\end{theorem}
\begin{remark}
    Note that the coefficients $A_4(n)\defeq \sigma_1(\frac n2)-(-1)^n \sigma_1(n)$ satisfy
    \begin{equation*}
    2\sum_{n=1}^\infty A_4(n)q^n = 
    \sum_{n\in\Z} n^2 q^{n^2}\Big/\sum_{n\in\Z}q^{n^2}. 
    \end{equation*}
\end{remark}

In dimension $d=3$, we have the following interpolation function. 
\begin{theorem}\label{theoremMainDim3}
	Let two sequences of entire functions $b_{3,n}, \widetilde b_{3,n}$, $n\ge0$, be given by
	\begin{align*}
        b_{3,0}(r)&=\widetilde{b}_{3,0}(r)=\frac{\sin(\pi r^2)}{2r\sinh(\pi r)},\\
		b_{3,n}(r)&=B_{3,n}(r)-\frac{\sin(\pi r^2)}{r\sinh(\pi r)}\(8H(n)+\frac{r_3(n)}{6}\),\quad n\geq 1,\\
		\widetilde{b}_{3,n}(r)&=\widetilde{B}_{3,n}(r)-\frac{\sin(\pi r^2)}{r\sinh(\pi r)}\(8H(n)+\frac{r_3(n)}{6}\),\quad n\geq 1,
	\end{align*}
    where $H(n)$ is the Hurwitz class number and $r_3(n)$ is the number of ways to write $n$ as a sum of three squares (also the coefficient of $e^{\pi i n \tau}$ in $\Theta(\tau)^3$ \eqref{eqThetafunction}). Here $B_{3,n}(r)$ and $\widetilde{B}_{3,n}(r)$ are defined as
	\begin{align*}
		B_{3,n}(r)&=\frac{\ee(\tfrac 18)}{|r|}\sum_{2|c>0}\frac{K(r^2,n,c,\nu_{\Theta}^3)}{\sqrt{c}} \sin\(\frac{2\pi |r|\sqrt n}c\),\\
		\widetilde{B}_{3,n}(r)&=\frac{\ee(-\tfrac 38)}{|r|}\sum_{2\nmid {\widetilde{c}}>0}\frac{\widetilde{K}(r^2,n,{\widetilde{c}},\nu_{\Theta}^3)}{\sqrt{{\widetilde{c}}}} \sin\(\frac{2\pi |r|\sqrt n}{\widetilde{c}}\),
	\end{align*}
	where the real-variable Kloosterman sums $K$ and $\widetilde{K}$ are defined at \eqref{eqReal-Val-Kloosterman-Sums} and \eqref{eqReal-Val-Kloosterman-Sums-tilde}, respectively. Then
	\begin{itemize}
	\item [(1)]	For every $f\in \mS_{\rad}(\R^3)$ and every $x\in \R^3$ we have
	\[f(x)=\sum_{n=0}^\infty f(\sqrt n) b_{3,n}(|x|)+\sum_{n=0}^\infty \widehat f(\sqrt n) \widetilde b_{3,n}(|x|). \]

	\item[(2)] We have the following estimates for $B_{3,n}(r)$ and $\widetilde{B}_{3,n}(r)$: 
	\[|B_{3,n}(r)|, |\widetilde{B}_{3,n}(r)| \ll_\ep \left\{\begin{array}{ll}
	       \max\big(n^{\frac {437}{588}}, n^{\frac {11+\kappa}{16}}\big)n^{\ep}& \text{if }r^2\leq 1/n, \vspace{2px}\\
	    (|r|^{-\frac 12}+|r|^{\ep}) n^{\frac 12+\ep} & \text{if }r^2\geq 1/n. 
	\end{array} \right.   \]
	Here $\kappa\in [0,1]$ is chosen so that 
    $2^{\nu_2(n)}\ll n^\kappa$.  Note that $\frac{437}{588}\approx 0.743\ldots<0.75$.

    \item[(3)] For $n\geq 1$, we have
    \[\begin{array}{rll}
    &B_{3,n}(0)=8H(n)-r_3(n)/3,\ & \widetilde{B}_{3,n}(0)=8H(n)+2r_3(n)/3, \\  
    &\:\,b_{3,n}(0)=-{r_3(n)}/2,\  &\;\widetilde{b}_{3,n}(0)={r_3(n)}/2.
    \end{array}\]
\end{itemize}
\end{theorem}

\begin{remark}
    (1) Note that $B_{3,n}(0),\widetilde{B}_{3,n}(0)\ll_\ep n^{\frac 12+\ep}$, while the exponent $\frac{437}{588}$, appearing in the estimate for general $r>0$, is equal to $\frac 34-\frac 1{147}$. 

    (2) In \S\ref{sectionProof-of-theorem-FourierExpansion-weight3/2-s3/4} we prove\footnote{It is likely that~\eqref{eq:zagiertheta} is well-known, but we could not find it in the literature} the following identity between Zagier's weight $3/2$ non-holomorphic modular form $\mathcal{H}(\tau)$ (see \eqref{eq:Zagier'sMockModularForm} for definition) and $\Theta(\tau)$:
    \begin{equation} \label{eq:zagiertheta}
        \mathcal{H}(\tau)+(2\tau/i)^{-3/2}\mathcal{H}(-1/4\tau)=-\Theta(2\tau)^3/24.
    \end{equation}
    The quantity $A_3(n)\defeq 8H(n)+\frac{1}{6}r_3(n)$ comes from the Fourier coefficients of the function
    \[\widetilde{\mathcal{H}}(\tau)\defeq 8\mathcal{H}\(\frac \tau2\)+\frac{\Theta(\tau)^3}{6}=\sum_{n=1}^\infty A_3(n) e^{\pi i n \tau}+\text{non-holomorphic terms}  \]
    that satisfies $\widetilde{\mathcal{H}}(\tau)+(\tau/i)^{-3/2}\widetilde{\mathcal{H}}(-1/\tau)=0$. 
\end{remark}

\subsection{Relation with Theorem~\ref{theoremBonRadSeip-2022}}

For $d\in \{3,4\}$, the formulas for the functions $b_{d,n}(r)$ and $\widetilde{b}_{d,n}(r)$ will follow from a construction of the generating series
\begin{equation} \label{eqDefMathcalG}
	G_{\frac{d}2}(\tau;r)=\sum_{n=0}^\infty b_{d,n}(r)e^{\pi i n \tau} \quad \text{and}\quad  \widetilde{G}_{\frac{d}2}(\tau;r)=\sum_{n=0}^\infty \widetilde{b}_{d,n}(r)e^{\pi i n \tau}. 
\end{equation}
In the proofs of Theorem~\ref{theoremMainDim4} and Theorem~\ref{theoremMainDim3}, we will show that these series satisfy the functional equation
\begin{equation}
	G_{\frac{d}2}(\tau;r)+(-i\tau)^{-d/2}\widetilde{G}_{\frac{d}2}(-\tfrac 1\tau;r)=e^{\pi i r^2 \tau},
\end{equation}
which implies (and by \cite[\S6]{RadchenkoViazovska2019} is equivalent to) the corresponding Fourier interpolation formula. For $d=4$ we have $b_{4,0}=\widetilde{b}_{4,0}=0$, and for $d=3$ we have $b_{3,0}=\widetilde{b}_{3,0}$. Therefore, by the uniqueness part of Theorem~\ref{theoremBonRadSeip-2022}, we get the following result. 
\begin{corollary}\label{corollaryGconstructedD/2=FBRSD/2}
	Let $d\in \{3,4\}$. Then for $\mathcal{F}_{\frac{d}2}(\tau;r)$ and $\widetilde{\mathcal{F}}_{\frac{d}2}(\tau;r)$ defined in Theorem~\ref{theoremBonRadSeip-2022} and $\mathcal{G}_{\frac{d}2}^{\ep}(\tau;r)$ defined at \eqref{eqDefMathcalG}, we have
	\[\mathcal{G}_{\frac{d}2}(\tau;r)=\mathcal{F}_{\frac{d}2}(\tau;r),
    \qquad
     \widetilde{\mathcal{G}}_{\frac{d}2}(\tau;r)=\widetilde{\mathcal{F}}_{\frac{d}2}(\tau;r).\]
    In particular, we have $a_{3,n}=b_{3,n}$, $\widetilde{a}_{3,n}=\widetilde{b}_{3,n}$, $a_{4,n}=b_{4,n}$, and $\widetilde{a}_{4,n}=\widetilde{b}_{4,n}$ for all $n\ge0$.
\end{corollary}

In \cite[\S7]{RadchenkoViazovska2019}, the first author together with Viazovska constructed an interpolation basis for one-dimensional odd Schwartz functions.
\begin{theorem}[{Corollary of \cite[Prop.~3, Thm.~7]{RadchenkoViazovska2019}}] \label{theoremRV2019}
There exists a collection of odd Schwartz functions $d_{n}^{\ep}$, $n\ge0$, $\ep\in \{\pm\}$, satisfying
    \[\widehat{d_n^\ep}(x)=\ep(-i)d_n^\ep(x)\quad \text{and}\quad d_n^\ep(\sqrt m)=\delta_{n,m}\sqrt m,\quad m\geq 0, \]
together with ${d_{n}^{+}}'(0)=\delta_{n,0}$ and ${d_{n}^{-}}'(0)=0$, $n\ge0$. The Schwartz functions $d_{n}^{\ep}$ are uniquely determined by the above properties.
\end{theorem}

From Theorem~\ref{theoremMainDim3}, Corollary~\ref{corollaryGconstructedD/2=FBRSD/2}, and~\cite[Eq. (39)]{RadchenkoViazovska2019} we get that
\begin{align}\label{eqconstruct-cn+-as-rBHr3}
\begin{split}
    d_n^+(x)&=xB_{3,n}(|x|)+x\widetilde{B}_{3,n}(|x|) -\frac{2\sin(\pi x^2)}{\sinh(\pi x)}\(8H(n)+ \frac{r_3(n)}{6}\),\\
    d_n^-(x)&=xB_{3,n}(|x|)-x\widetilde{B}_{3,n}(|x|).
\end{split}
\end{align}

\begin{remark}
The above equation for $n=0$ implies (using that $H(0)=-\tfrac1{12}$) that $d_{0}^{+}(x)=\frac{\sin(\pi x^2)}{\sinh(\pi x)}$, as noted in~\cite[\S8]{RadchenkoViazovska2019}.    
\end{remark}

When rewritten in terms of real-variable Kloosterman sums, we get the following corollary.
\begin{corollary}
    For $n\geq 1$, let us define 
    \begin{align}
        K_n(r)&\defeq \sum_{2|c>0}\frac{K(r^2,n,c,\nu_{\Theta}^3)}{\sqrt{c}} \sin\(\frac{2\pi r\sqrt n}{c}\)=\ee(-\tfrac 18)rB_{3,n}(|r|),\\
        \widetilde{K}_n(r) & \defeq \sum_{2\nmid {\widetilde{c}}>0}\frac{\widetilde{K}(r^2,n,{\widetilde{c}},\nu_{\Theta}^3)}{\sqrt{{\widetilde{c}}}} \sin\(\frac{2\pi r\sqrt n}{{\widetilde{c}}}\)=\ee(\tfrac 38)r\widetilde{B}_{3,n}(|r|).
    \end{align}
    Then
    \begin{itemize}
        \item[(1)] $K_n(r),\widetilde{K}_n(r)$ are odd Schwartz functions,
        \item[(2)] $\widehat{K_n}(r)=i\,\widetilde{K}_n(r)$, $\widehat{\widetilde{K}_n}(r)=i\,K_n(r),$ and
        \item[(3)] for $m\geq 1$, \begin{equation}\label{eq_Knsqrtm_kroneckerdelta}
        K_n(\sqrt m)=\ee(-\tfrac 18)\delta_{m,n}\,\sqrt m,\quad \widetilde{K}_n(\sqrt m)=0.
    \end{equation} 
    \end{itemize}
\end{corollary}
\begin{remark}
    The property \eqref{eq_Knsqrtm_kroneckerdelta} provides another proof of the following identity on half-integral weight Kloosterman sums on $\Gamma_0(4)$ (for definitions see \eqref{eqKloosterman-Sums-General-Def} and \eqref{eqKloosterman-Sums-Relation-Even-Theta-to-theta}): 
    \begin{equation}
        2\pi\ee(-\tfrac 18)\(\frac n m\)^{\frac 14} \sum_{4|c>0} \frac{S(-m,-n,c,\nu_{\theta})}c J_{\frac 12}\(\frac{4\pi\sqrt{mn}}c\)=\delta_{m,n}, \quad \text{for }m,n\geq 1. 
    \end{equation}
\end{remark}
We also significantly improve the growth estimates for $d_n^\ep(r)$ in \cite[Theorem~6]{RadchenkoViazovska2019}, and improve the bound \eqref{eq:heckebound} from \cite{BondarenkoRadchenkoSeip2022}. 
\begin{corollary}
    For the function $d_n^\ep(x)$ in Theorem~\ref{theoremRV2019} where $\ep\in\{\pm\}$, we have 
    \[
        d_n^\ep(x)=-(1+\ep)\frac{\sin(\pi x^2)}{\sinh(\pi x)}\(8H(n)+\frac{r_3(n)}{6}\)+B_n^\ep(x),
    \]
    where $B_n^{\ep}(x)=xB_{3,n}(|x|)+\ep x\widetilde{B}_{3,n}(|x|)$ is an odd Schwartz function of $x$, given by a sum of real-variable Kloosterman sums, and can be estimated by
    \[B_n^{\ep}(x)\ll_\delta \left\{\begin{array}{ll}
	       |x|\max\big(n^{\frac 34-\frac 1{147}}, n^{\frac {11+\kappa}{16}}\big)n^{\delta}& \text{ if }x^2\leq 1/n, \vspace{2px}\\
	       (|x|^{\frac 12}+|x|^{1+\delta})n^{\frac 12+\delta} & \text{ if }x^2\geq 1/n, 
	\end{array} \right. \]
    for any $\delta>0$. As Theorem~\ref{theoremMainDim3}, $\kappa\in[0,1]$ is chosen so that $2^{\nu_2(n)}\ll n^\kappa$. 
\end{corollary}

The paper is organized as follows. In Section~\ref{sectionPrelimandNotation} we setup notation and recall some preliminary results on Kloosterman sums and Maass forms. In Section~\ref{SectionMaassPoincareSeries} we construct (following the ideas of~\cite{StollerInterpolation}) modular integrals via Maass--Poincar\'e series by a special choice of coset representatives modulo parabolic subgroup for $\Gamma(2)$ and compute their Fourier expansions. Section~\ref{SectionKLweight2} reviews properties of weight $2$ Kloosterman sums, which we then use to prove properties of the corresponding weight $2$ \textit{real-variable} Kloosterman sums in Section~\ref{SectionRVKLweight2}. We prove Theorem~\ref{theoremMainDim4} in Section~\ref{sectionProof-of-theorem-FourierExpansion-weight2-s1}. Sections~\ref{SectionKLweight3/2}, \ref{SectionRVKLweight3/2} and \ref{sectionProof-of-theorem-FourierExpansion-weight3/2-s3/4} treat, respectively, the properties of weight $3/2$ Kloosterman sums, the properties of weight $3/2$ \textit{real-variable} Kloosterman sums, and the proof of Theorem~\ref{theoremMainDim3}. We conclude with a brief discussion of open questions in Section~\ref{SectionDiscussion}.

\subsection*{Acknowledgement}
The authors acknowledge funding by the European Union (ERC, FourIntExP, 101078782).
Views and opinions expressed are those of the author(s) only and do not necessarily reflect those of the European Union or European Research Council (ERC). Neither the European Union nor ERC can be held responsible for them.
\subsection*{Conflict of interest statement}
On behalf of all authors, the corresponding author states that there is no conflict of interest. 
\subsection*{Data availability statement}
No data were generated or analyzed in this study.

\section{Preliminaries and notation}
\label{sectionPrelimandNotation}
We use $\Z_+$ and $\R_+$ to denote the sets of positive integers and positive real numbers, respectively. Let  $I=\begin{psmallmatrix}
	1&0\\0&1
\end{psmallmatrix}$ denote the $2\times 2$ identity matrix, $T=\begin{psmallmatrix}
1&1\\0&1
\end{psmallmatrix}$ and $S=\begin{psmallmatrix}
0&-1\\1&0
\end{psmallmatrix}$. 
For $n\in \Z_+$, let $\phi(n)$ denote Euler's totient function and define the divisor sum function by
    \[\sigma_s(n)\defeq \sum_{d|n} d^s. \]
As a special case, $\sigma_0(n)$ is the number of divisors of $n$. Let $\sgn n\in \{-1,0,1\}$ denote the sign function for $n\in \Z$. For any prime $p$, let $\nu_p(n)$ denote the $p$-adic valuation of $n$, i.e., $p^{\nu_p(n)}\| n$, where $p^t\|b$ means $p^t|b$ and $p^{t+1}\nmid b$. We also denote
    \[\ee(\tau) \defeq e^{2\pi i \tau}\]
    and let $\Gamma(a,\tau)$ be the incomplete Gamma function \cite[(8.2.2)]{dlmf}. 

Let $\zeta(s)$ be the Riemann zeta function. If $\chi$ is a Dirichlet character, let $L(s,\chi)$ be the Dirichlet $L$-function for $s\in \C$ where
\[L(s,\chi)=\sum_{n=1}^\infty \frac{\chi(n)}{n^s}\qquad  \text{for }\re(s)>1. \]
For $m,n\in \Z$, let $(\frac mn)$ be the Kronecker symbol where $(\frac m{-1})= 1$ if $m\geq 0$ and $(\frac m{-1})= -1$ if $m<0$. In particular, for any $m\in \Z$, $(\frac{-4m}\cdot)$ is a Dirichlet character modulo $|4m|$. 

We recall the following standard congruence subgroups of $\SL_2(\Z)$. For $N\in \Z_+$, define 
\[\Gamma_0(N)=\{\begin{psmallmatrix}
	*&*\\c&*
\end{psmallmatrix}\in \SL_2(\Z):c\equiv 0\Mod N\},\quad \Gamma(N)=\{\gamma\in \SL_2(\Z):\gamma\equiv I\Mod N\}.\]
Whether $\Gamma(\cdot)$ refers the congruence subgroup (we only use $\Gamma(2)$) or the Gamma function (like $\Gamma(s\pm \frac k2)$) shall be clear among the context.  

Let $\Gamma_{\Theta}$ denote another congruence subgroup of $\SL_2(\Z)$:
\begin{equation}\label{eqGammaThetaGroup}
	\Gamma_{\Theta}=\{\gamma\in \SL_2(\Z):\gamma\equiv I \text{ or }S\Mod 2\}. 
\end{equation} 
It is known that $\Gamma_{\Theta}$ is freely generated by $T^2$ and $S$ subject to $S^2=-1$.

For $\tau \in \C^{\times}$, we define the argument $\arg (\tau)$ to lie in $(-\pi,\pi]$. Let $\HH\defeq \{\tau\in \C: \im \tau>0\}$ be the complex upper half-plane. For any $\gamma=\begin{psmallmatrix}
	*&*\\c&d
\end{psmallmatrix}\in \SL_2(\R)$ and $\tau\in\HH$, we define the automorphic factor
    \[j(\gamma,\tau)\defeq c\tau+d. \]
For $k\in \frac 12\Z$ and $\tau\in \C$, we define
    \[\tau^k=|\tau|^k \exp(ik\arg(\tau))\]
and the weight $k$ slash operator 
    \begin{equation}\label{eq:def:slashoperator}
        (f|_k\gamma)(\tau)\defeq j(\gamma,\tau)^{-k}f(\gamma \tau).
    \end{equation}
    
\begin{definition}\label{defMultiplierSystem}
	We say that $\nu:\Gamma\to \C^\times$ is a multiplier system of weight $k\in \frac 12\Z$ if
	\begin{enumerate}[label=(\roman*)]
		\item $|\nu|=1$,
		\item $\nu(-I)=e^{-\pi i k}$, and
		\item $\nu(\gamma_1 \gamma_2) =w_k(\gamma_1,\gamma_2)\nu(\gamma_1)\nu(\gamma_2)$ for all $\gamma_1,\gamma_2\in \Gamma$, where
		\[w_k(\gamma_1,\gamma_2)\defeq j(\gamma_2,\tau)^k j(\gamma_1,\gamma_2\tau)^k j(\gamma_1\gamma_2,\tau)^{-k}.\]
	\end{enumerate}
\end{definition}

If $\nu$ is a multiplier system of weight $k$, then it is also a multiplier system of weight $k'$ for any $k'\equiv k\pmod 2$, and its conjugate $\overline\nu$ is a multiplier system of weight $-k$. 
One can also easily check that
    \begin{equation}\label{MultiplierSystemBasicProprety}
    \nu(\gamma)\nu(\gamma^{-1})=1\quad\text{and} \quad  \nu(\gamma \begin{psmallmatrix}
    	1&bt\\0&1
    \end{psmallmatrix})=\nu(\gamma)\nu(\begin{psmallmatrix}
    	1&b\\0&1
    \end{psmallmatrix})^t\quad \text{for }b,t\in \Z. 
    \end{equation}
From the definition it follows that if $\nu$ is a multiplier system of weight $\frac12$, then
    \begin{equation}\label{eqMultiplierOn-Minus-gamma}
	\nu(-\gamma)= i\,\nu(\gamma) \quad \text{if }\gamma=\begin{psmallmatrix}
		*&*\\c&*
	\end{psmallmatrix} \text{ and }c>0. 
    \end{equation}

For any congruence subgroup $\Gamma$ and any cusp $\mathfrak{a}\in \mathbb{P}^1(\Q)/\Gamma$, let $\Gamma_{\mathfrak{a}}$ be the stabilizer of $\mathfrak{a}$ in $\Gamma$. For example, $\Gamma_\infty=\{\pm\begin{psmallmatrix}
	1&b\\0&1
\end{psmallmatrix}:b\in\Z\}\cap \Gamma$. Let $\sigma_{\mathfrak{a}}\in\SL_2(\R)$ denote a scaling matrix satisfying $\sigma_{\mathfrak{a}}\infty=\mathfrak{a}$ and $\sigma_{\mathfrak{a}}^{-1} \Gamma_{\mathfrak{a}}\sigma_{\mathfrak{a}}=\Gamma_\infty$.

Let $\overline{\Gamma}\defeq \Gamma/\{\pm I\}$ for any congruence subgroup $\Gamma$. For an element $\gamma \in \SL_2(\Z)$, we let $[\gamma]$ denote its class modulo $\{\pm I\}$. In this subsection we focus on $\Gamma(2)=\{\gamma\in \SL_2(\Z): \gamma\equiv I\Mod 2\}$ as a congruence subgroup of $\Gamma_{\Theta}$ \eqref{eqGammaThetaGroup}. We know that $\overline{\Gamma(2)}$ is freely generated by $A=[T^2]$ and $B=[ST^2S]$, and  $\Gamma_{\Theta}=\Gamma(2)\sqcup \Gamma(2)S$, where $\sqcup$ means the disjoint union.

\subsection{Kloosterman sums}
\label{subsectionKloosterman-Sums-Defs}
For any congruence subgroup $\Gamma$, we write the stabilizer of the cusp at $\infty$ as $\Gamma_{\infty}=\{\pm \begin{psmallmatrix}
	1&b\\0&1
\end{psmallmatrix}:b\in w\Z\}$ for some $w\in \Z_+$ which denotes the width of the cusp. For $c>0$, we define the Kloosterman sums with multiplier system $\nu$ on $\Gamma$ as
\begin{equation}\label{eqKloosterman-Sums-General-Def}
	S(m,n,c,\nu)=\sum_{\gamma=\begin{psmallmatrix}
			a&b\\c&d
	\end{psmallmatrix}\in \Gamma_{\infty}\setminus \Gamma/\Gamma_{\infty}} \nu(\gamma)^{-1} \ee\(\frac{ma+nd}{wc}\). 
\end{equation}
Let $S(m,n,c)\defeq S(m,n,c,{\text{id}}_{\SL_2(\Z)})$ denote the standard Kloosterman sums as \cite[(1)]{SarnakTsimerman09}. 

We define the theta function as
\begin{equation}\label{eqThetafunction}
	\Theta(\tau)=\sum_{n\in \Z} e^{\pi i n^2 \tau}.
\end{equation}
It satisfies
\begin{equation}
    \Theta(\gamma \tau)=\nu_{\Theta}(\gamma) (c\tau+d)^{\frac 12}\Theta(\tau),\quad \gamma\in \Gamma_\Theta, 
\end{equation}
where $\nu_{\Theta}$ is a weight $\frac 12$ multiplier system on $\Gamma_{\Theta}$. For $c>0$, we have
\begin{equation}\label{ThetaMultiplier}
	\nu_{\Theta}
	\begin{psmallmatrix}
		a&b\\c&d
	\end{psmallmatrix}=\left\{
	\begin{array}{ll}
		\ep_d^{-1}\(\frac{2c}d \),& c\equiv 0\Mod 2,\\
		\ee(-\frac 18)\ep_c\(\frac{2d}c \),&c\equiv 1\Mod 2,
	\end{array}
	\right.\quad \text{where } \ep_d=\left\{\begin{array}{ll}
		1,&d\equiv 1\Mod 4,\\
		i,&d\equiv 3\Mod 4,\\
		0,&\text{otherwise.}
	\end{array}
	\right. 
\end{equation}
Equivalently, by \cite[Theorem~7.1]{Tata1}, for $c>0$ we have
\begin{equation}
	\nu_{\Theta}
	\begin{psmallmatrix}
		a&b\\c&d
	\end{psmallmatrix}=\left\{
	\begin{array}{ll}
		i^{\frac {d-1}2}(\frac{c}{|d|} ),& c\equiv 0\Mod 2,\\
		\ee(-\frac c8)(\frac{d}c ),&c\equiv 1\Mod 2. 
	\end{array}
	\right.
\end{equation}
Specifically, we have $\nu_{\Theta}\begin{psmallmatrix}
	*&*\\0&*
\end{psmallmatrix}=1$,  $\nu_{\Theta}(S)=\ee(-\frac 18)$, and
\begin{equation}\label{eqNuTheta-gammaS-property}
    \nu_{\Theta}(\gamma S)=(\sgn d)\cdot \ee(-\tfrac 18)\,\nu_{\Theta}(\gamma) \quad \text{for }c>0.
\end{equation}
For $c<0$ we recall \eqref{eqMultiplierOn-Minus-gamma}. 

Another normalization for Jacobi's theta function is
\begin{equation}\label{eqthetafunction}
	\theta(\tau)\defeq \sum_{n\in \Z} \ee(n^2 \tau)=\Theta(2\tau). 
\end{equation}
It is a modular form of weight $\frac12$ on $\Gamma_0(4)$ with a multiplier system $\nu_\theta$ that is given by
\begin{equation}
	\theta(\gamma \tau)=\nu_\theta(\gamma) (c\tau+d)^{\frac 12}\theta(\tau), \quad  \nu_\theta\begin{psmallmatrix}
		a&b\\c&d
	\end{psmallmatrix}=\ep_d^{-1}(\tfrac cd), \quad \text{for } \gamma=\begin{psmallmatrix}
	a&b\\c&d
	\end{psmallmatrix}\in \Gamma_0(4). 
\end{equation}

Since the stabilizer in $\Gamma_{\Theta}$ of the cusp at $\infty$ is $(\Gamma_{\Theta})_{\infty}=\{\pm T^{2n}:n\in \Z\}$, the width of the cusp at $\infty$ is 2. For $k\in \frac 12\Z$ and $c\in \Z_+$, we can explicitly write down the weight $k$ Kloosterman sums defined on $\Gamma_{\Theta}$ with multiplier system $\nu_{\Theta}^{2k}$, where we shorten $a\Mod c$ to $a(c)$ for simplicity: 
\begin{equation}\label{eqKloosterman-Sums-on-Theta}
	S(m,n,c,\nu_{\Theta}^{2k})=\left\{\begin{array}{ll}
		\displaystyle \sum_{\substack{d(2c)\\ad\equiv 1(2c)}} \ep_d^{2k} \(\frac{2c}d\)^{2k} \ee\(\frac{ma+nd}{2c}\),&\text{ if }c\text{ is even,}\\
		\displaystyle \sum_{\substack{d(2c)\\2|a,\  2|d\\(a,c)=(d,c)=1\\ad\equiv 1(c)}}\ee(\tfrac k4)\ep_c^{-2k}\(\frac{2d}c\)^{2k} \ee\(\frac{ma+nd}{2c}\), &\text{ if } c\text{ is odd.}
	\end{array}
	\right.
\end{equation}
For $m,n\in \Z$ and $c\in 2\Z_+$, we have $4|2c$ and the following relations are clear by definition: 
\begin{equation}\label{eqKloosterman-Sums-Relation-Even-Theta-to-theta}
S(m,n,c,\nu_{\Theta}^{2k})=S(m,n,2c,\nu_\theta^{2k}),\quad \text{in particular, \ } S(m,n,c,\nu_{\Theta}^4)=S(m,n,2c).  
\end{equation}

We also need the conjugation property of Kloosterman sums. For any multiplier system~$\nu$ on a congruence subgroup $\Gamma$, if $\nu(\gamma)=1$ for all $\gamma\in \Gamma_{\infty}$, then 
\begin{equation}
	\overline{S(m,n,c,\nu)}=S(-m,-n,c,\overline{\nu}). 
\end{equation}
Specifically, we have
\begin{equation}\label{eqKloosterman-Sums-Conjugate-ThetasMulti}
	\overline{S(m,n,c,\nu_{\Theta}^{2k})}=S(-m,-n,c,\nu_{\Theta}^{-2k})\quad \text{and}\quad  \overline{S(m,n,c,\nu_{\theta}^{2k})}=S(-m,-n,c,\nu_{\theta}^{-2k}).
\end{equation}
We also have
\begin{equation}\label{eqKloosterman-Sums-alter-mn-inTheta}
S(m,n,c,\nu_{\Theta}^{2k})=S(n,m,c,\nu_{\Theta}^{2k}) \quad \text{for }2|c,
\end{equation}
because $ad\equiv 1\Mod {2c}$, for even $c$, implies $\ep_a=\ep_d$ and $(\frac{2c}{a})=(\frac{2c}d)$.

\subsection{Other relations between Kloosterman sums} \label{subsectionRelation-KL-sums}
In this subsection we prove the following proposition.
\begin{proposition}\label{propAppendix-KLsums-evenodd-relation}
	For $m,n\in \Z$ and $c=2\widetilde{c}$ where $\widetilde{c}\in \Z_+$ is odd, we have: 
	\begin{equation}\label{eqrelation-KLsum-evenodd-k1/2}
		S(m,4n,c,\nu_{\Theta})=\left\{\begin{array}{ll}
			\sqrt{2}\,S(m,n,\widetilde{c},\nu_{\Theta}),& \text{if }m\equiv 0,1\Mod 4;\\
			-\sqrt{2}\, S(m,n,\widetilde{c},\nu_{\Theta}),& \text{if }m\equiv 2,3\Mod 4;\\
		\end{array}\right.
	\end{equation}
	\begin{equation}\label{eqrelation-KLsum-evenodd-k3/2}
		S(m,4n,c,\nu_{\Theta}^3)=\left\{\begin{array}{ll}
			-\sqrt{2}\,S(m,n,\widetilde{c},\nu_{\Theta}^3),& \text{if }m\equiv 0,3\Mod 4;\\
			\sqrt{2}\,S(m,n,\widetilde{c},\nu_{\Theta}^3),& \text{if }m\equiv 1,2\Mod 4;\\
		\end{array}\right.
	\end{equation}
	\begin{equation}\label{eqrelation-KLsum-evenodd-k2}
		S(2m,2n,c,\nu_{\Theta}^4)=2(-1)^{m+n+1}S(m,n,\widetilde{c},\nu_{\Theta}^4). 
	\end{equation}
\end{proposition}
These relations will help us in estimating sums of $S(m,n,\widetilde{c},\nu_{\Theta}^{2k})$ for $k=\frac 32$ and $2$.

The proof goes along similar lines as Bir\'o's work in \cite[Appendix \S A.3]{Biro2000cycleintegral}, which is about Kloosterman sums $S(m,n,c,\nu_{\theta})$ on $\Gamma_0(4)$ with cusp pairs $(\infty,\infty)$, $(\infty,0)$ and $(\infty,\frac 12)$. One can show that the Kloosterman sum $S(m,n,\widetilde{c},\nu_{\Theta}^{2k})$ for odd~$\widetilde{c}$ is essentially the Kloosterman sum with cusp pair $(\infty,0)$ on $\Gamma(2)$. The relation between $\Gamma_0(4)$ and $\Gamma(2)$ allows us to translate Bir\'o's computations to our setting.

The following lemma is direct and we omit the proof.

\begin{lemma}\label{lemmaAppendix-two-ranges-modulo-2c}
	For odd $\widetilde{c}\in \Z_+$, the following integer sets are the same modulo $2\widetilde{c}$: 
	\begin{itemize}
		\item[(1)] $\{d:\ \  1\leq d\leq 2\widetilde{c},\ 2|d,\ (d,\widetilde{c})=1\}$;
		\item[(2)] $\{2d:\ 1\leq d\leq 4\widetilde{c},\ (d,2\widetilde{c})=1,\ d\equiv 1\Mod 4\}$; 
		\item[(3)] $\{\beta d:\ 1\leq d\leq 2\widetilde{c},\ (d,2\widetilde{c})=1\}$, for any even $\beta$ with $(\beta,\widetilde{c})=1$. 
	\end{itemize}
\end{lemma}

Now we suppose $\widetilde{c}\in \Z_+$ is odd and $c=2\widetilde{c}$. Let $\overline{x_{n}}$ denote the inverse of $x$ modulo~$n$. For every $d$ modulo $2c$ such that $(d,2c)=1$, we use the pairing $d\leftrightarrow d+c$, the fact that $d\equiv (-1)^n\Mod 4$ if and only if $d+c\equiv (-1)^{n+1}\Mod 4$, and the fact that $\overline{(d+c)_{2c}}\equiv \overline{d_{2c}}+c$ modulo $2c$. Then we have
\begin{align}\label{eqAppendix-KL-m4n2c-compute}
	\begin{split}
		S(m,4n,c,\nu_{\Theta}^{2k})&=\sum_{\substack{d\Mod{2c}\\ad\equiv 1(2c)}} \ep_d^{2k} \(\frac{2c}d\)^{2k} \ee\(\frac{ma+4nd}{2c}\)\\
		&=\(1+i^{2k}(-1)^{\frac{\widetilde{c}-1}2\cdot 2k+m}\)\sum_{\substack{d\Mod{2c}\\ad\equiv 1(2c)\\d\equiv 1\Mod 4}}  \(\frac{d}{c/2}\)^{2k} \ee\(\frac{ma+4nd}{2c}\).
	\end{split}
\end{align}
Since $4\cdot\overline{4_{\widetilde{c}}}=1+\beta \widetilde{c}$ implies $\beta+\widetilde{c}\equiv 0\Mod 4$, for $a\equiv 1\Mod 4$ we have
\begin{equation}\label{eqAppendix-ma2c-to-m2a4inverse}
	\ee\(\frac{ma}{2c}\)=\ee\(\frac{ma\cdot \overline{4_{\widetilde{c}}}}{\widetilde{c}}\)\ee\(\frac{m\widetilde{c}}4\). 
\end{equation}
We obtain that
\begin{equation}\label{eqAppendix-quadratic-tildec-relation-alotof-2k134}
\(1+i^{2k}(-1)^{\frac{\widetilde{c}-1}2\cdot 2k+m}\)\ee\(\frac{m\widetilde{c}}4\) =\left\{\begin{array}{ll}
	(1+i) \ep_{\widetilde{c}}^{-1}, &\text{if }2k=1,\ m\equiv 0,1\Mod 4;\\
	-(1+i) \ep_{\widetilde{c}}^{-1}, &\text{if }2k=1,\ m\equiv 2,3\Mod 4;\\
	(1-i)\ep_{\widetilde{c}}, &\text{if }2k=3,\ m\equiv 0,3\Mod 4;\\
	-(1-i)\ep_{\widetilde{c}}, &\text{if }2k=3,\ m\equiv 1,2\Mod 4.\\
\end{array}
\right.
\end{equation}

Taking the following facts into account: 
\[\(\frac{d}{c/2}\)=\(\frac{2\cdot 2d}{\widetilde{c}}\),\quad (2a\cdot \overline{4_{\widetilde{c}}})\cdot(2d)\equiv 1\Mod {\widetilde{c}},\]
the equations \eqref{eqrelation-KLsum-evenodd-k1/2} and \eqref{eqrelation-KLsum-evenodd-k3/2} in Proposition~\ref{propAppendix-KLsums-evenodd-relation} are then proved by combining \eqref{eqKloosterman-Sums-on-Theta}, \eqref{eqAppendix-KL-m4n2c-compute}, \eqref{eqAppendix-ma2c-to-m2a4inverse},  \eqref{eqAppendix-quadratic-tildec-relation-alotof-2k134}, and Lemma~\ref{lemmaAppendix-two-ranges-modulo-2c}. 

Note that the left hand side of \eqref{eqAppendix-quadratic-tildec-relation-alotof-2k134} is always $0$ when $k=2$ and $m\equiv 1,3\Mod 4$. Hence from there we cannot conclude the relation between $S(m,4n,c,\nu_{\Theta}^4)$ and $S(m,n,\widetilde{c},\nu_{\Theta}^4)$. In fact, for $k=2$ we need to modify the coefficients in \eqref{eqrelation-KLsum-evenodd-k2}. By \eqref{eqAppendix-ma2c-to-m2a4inverse} we have
\begin{align*}
	S(2m,2n,c,\nu_{\Theta}^4)&=\sum_{\substack{d\Mod{2c}\\ad\equiv 1(2c)}} \ee\(\frac{2ma}{2c}\)\ee\(\frac{2nd}{2c}\)\\
	&=2\sum_{\substack{d\Mod{c}\\ad\equiv 1(2c)}} \ee\(\frac{2ma}{4\widetilde{c}}\)\ee\(\frac{2nd}{4\widetilde{c}}\)\\
	&=2\sum_{\substack{d\Mod{2\widetilde{c}}\\ad\equiv 1(\widetilde{c})}} \ee\(\frac{m(4\cdot \overline{4_{\widetilde{c}}}) a}{2\widetilde{c}}\)\ee\(\frac{n(4\cdot \overline{4_{\widetilde{c}}}) d}{2\widetilde{c}}\)\ee\(\frac{2(m+n)\widetilde{c}}4\)
\end{align*}
We finish the proof of \eqref{eqrelation-KLsum-evenodd-k2} by Lemma~\ref{lemmaAppendix-two-ranges-modulo-2c} and $\ee(\frac k4)=-1$.

\subsection{Maass forms}
\label{subsectionMaass-forms}
Details in this subsection can be find in various references \cite{Proskurin2005,DFI2002,DFI12,AAimrn,ahlgrendunn}; we are mainly following \cite[\S3]{QihangFirstAsympt}.

We call a function $f:\HH\rightarrow \C$ automorphic of weight $k$ with a multiplier $\nu$ on a congruence subgroup $\Gamma$ if
\[\(\frac{cz+d}{|cz+d|}\)^{-k} f(\gamma \tau)=\nu(\gamma) f(\tau)\quad \text{for all }\gamma=\begin{psmallmatrix}
	a&b\\c&d
\end{psmallmatrix}\in \Gamma. \]
Let $\mathcal{A}_k(\Gamma,\nu)$ denote the linear space consisting of all such functions, and let $\Lform_k(\Gamma,\nu)$ denote the space of square-integrable functions on $\Gamma\setminus \HH$ with respect to the hyperbolic measure 
\[d\mu(\tau)=\frac{dxdy}{y^2}\,.\]
The space $\Lform_k(\Gamma,\nu)$ is equipped with the Petersson inner product
\begin{equation}
	\langle f,g\rangle_{\Gamma}\defeq \int_{\Gamma\setminus \HH} f(\tau) \overline{g(\tau)}\frac{dxdy}{y^2} \quad\text{for }f,g\in \Lform_k(\Gamma,\nu). 
\end{equation}
The weight $k$ hyperbolic Laplacian is defined as
\[\Delta_k\defeq y^2\(\frac{\partial^2}{\partial x^2}+\frac{\partial^2}{\partial y^2}\)-iky \frac{\partial}{\partial x}. \]
It is known that $-\Delta_k$ is a self-adjoint operator on the Hilbert space $\Lform_k(\Gamma,\nu)$. The spectrum of $\Delta_k$ contains two parts: the continuous spectrum $\lambda\in [\frac 14,\infty)$ and the discrete spectrum of finite multiplicity $\lambda_0=\frac{|k|}2(1-\frac {|k|}2)<\lambda_1<\cdots \rightarrow \infty$. 

Let $M_k(\Gamma,\nu)$ denote the space of weight $k$ holomorphic modular forms on $(\Gamma,\nu)$. We recall the Serre-Stark basis theorem. 
\begin{theorem}[{\cite[Theorem~A]{SerreStark}}]
	The basis of $M_{\frac 12}(\Gamma_0(N),\nu_{\theta})$ consists of the theta series
	\[\theta_{\psi,t}(\tau)=\sum_{n\in  \Z} \psi(n) \ee(tn^2 z)\]
	where $t\in \Z_+$ and $\psi$ is an even primitive character with conductor $r(\psi)$, satisfying both $4r(\psi )^2 t|N$ and $\psi=(\frac D\cdot)$ where $D$ is the discriminant of $\Q(\sqrt t)/\Q$.
\end{theorem}
We have the following simple corollary. 
\begin{corollary}
	The two spaces are identical: $M_{\frac 12}(\Gamma_0(4),\nu_{\theta})=M_{\frac 12}(\Gamma_0(8),\nu_{\theta})$. Both of them are one-dimensional and generated by $\theta$, defined in \eqref{eqthetafunction}. 
\end{corollary}

Let $\LEigenform_k(\Gamma,\nu,\lambda)\subset \Lform_k(\Gamma,\nu)$ be the subspace spanned by eigenfunctions of $\Delta_k$ on $(\Gamma,\nu)$ with eigenvalue $\lambda$. There is a one-to-one correspondence between all $f\in \LEigenform_k(\Gamma,\nu,\lambda_0)$ and $F\in M_k(\Gamma,\nu)$ by
\begin{equation}\label{eqCorrespondence-Maass-and-holomorphic}
	f(\tau)=\left\{\begin{array}{ll}
		y^{\frac k2}F(\tau),&\quad k\geq 0,\ F\in M_k(\Gamma,\nu);\\
		y^{-\frac k2}\overline{F(\tau)}, &\quad k<0,\ F\in M_{-k}(\Gamma,\overline{\nu}). 
	\end{array}
	\right. 
\end{equation}
For example, every $f\in \Lform_{\frac 12}(\Gamma_0(4),\nu_{\theta},\frac 3{16})$ and every $g\in \Lform_{-\frac 12}(\Gamma_0(4),\overline{\nu_{\theta}},\frac 3{16})$ are of the form
\begin{equation}\label{eqExpansion-theta-Weight-pm1/2}
	f(\tau)=C_f y^{\frac 14}\theta(\tau), \quad g(\tau)=C_g y^{\frac 14}\overline{\theta(\tau)}\quad \text{for some }C_f,C_g\in \C,
\end{equation} 
 Moreover, such $f$ and $g$ are eigenfunctions of corresponding hyperbolic Laplacian with eigenvalue $\lambda_0=\frac 3{16}$. By \cite[(2.18)]{QihangSecondAsympt} (or also \cite[(6.2)]{AAAlgbraic16}) and \eqref{eqExpansion-theta-Weight-pm1/2} above, their Fourier expansions satisfy:
\begin{align}\label{eqMaass-form-coeffs-Whittaker}
	\begin{split}
		f(\tau)&= C_f y^{\frac 14}+\sum_{n=1}^\infty \rho_0^{(f)}(n) W_{\frac 14,\frac 14}(4\pi n y)\ee(nx)= C_f y^{\frac 14}+\sum_{m=1}^\infty 2C_f y^{\frac 14} \ee(m^2 \tau),\\
	g(\tau)&= C_g y^{\frac 14}+\sum_{n=-1}^{-\infty} \rho_0^{(g)}(n) W_{\frac 14,\frac 14}(4\pi |n| y) \ee(nx)=C_g y^{\frac 14}+\sum_{m=1}^\infty 2C_g y^{\frac 14} \ee(-m^2 \overline{\tau}).
	\end{split}
\end{align}

We need to know $\rho_0^{(f)}(n)$ and $\rho_0^{(g)}(n)$ when $f$ and $g$ are normalized eigenforms, i.e. when $\langle f,f\rangle_{\Gamma_0(4)}=\langle g,g\rangle_{\Gamma_0(4)}=1$, in order to apply theorems like \cite[Theorem~2]{gs} to estimate sums of Kloosterman sums. If $\langle f,f\rangle_{\Gamma_0(4)}=\langle g,g\rangle_{\Gamma_0(4)}=1$, then
\[C_f=C_g=(2\pi)^{-\frac 12}\]
due to the regularized Petersson inner product on the theta function: 
	\[\int_{\Gamma_0(4)\setminus \HH} y^{\frac 12}|\theta(\tau)|^2\frac{dxdy}{y^2}=2\pi.\]
This result can be found in \cite[Theorem~2.2]{Chiera2007PeterssonProduct}, noting that $[\SL_2(\Z):\Gamma_0(4)]=6$. Similarly, identities
\[\int_{\Gamma_0(8)\setminus \HH} y^{\frac 12}|\theta(\tau)|^2\frac{dxdy}{y^2}=4\pi \quad \text{and}\quad [\SL_2(\Z):\Gamma_0(8)]=12\]
help us deal with the case on $\Gamma_0(8)$.

By \cite[(13.18.2)]{dlmf}, $W_{\frac 14,\frac 14}(4\pi |n|y)=e^{-2\pi |n|y}(4\pi |n|y)^{\frac 14}$. Therefore, we get $\rho_0^{(f)}(n)$ and $\rho_0^{(g)}(n)$ by comparing coefficients in  \eqref{eqMaass-form-coeffs-Whittaker}. 
\begin{lemma}\label{lemmaMaass-form-coeffs-thetapm}
	Suppose $f\in \LEigenform_{\frac 12}(\Gamma_0(4),\nu_{\theta},\frac 3{16})$, $g\in \LEigenform_{-\frac 12}(\Gamma_0(4),\overline{\nu_{\theta}},\frac 3{16})$, $\langle f,f\rangle_{\Gamma_0(4)}=1$, and $\langle g,g\rangle_{\Gamma_0(4)}=1$. Then $f$ and $g$ have the Fourier expansion as in~\eqref{eqMaass-form-coeffs-Whittaker} with coefficients
	\begin{equation}\label{eqMaass-form-coeffs-thetapm}
		\rho_0^{(f)}(n)=\left\{ \begin{array}{ll}
			(\pi^3 n)^{-\frac 14},&\ n=m^2>0;\\
			0,&\ \text{other }n\neq 0,
		\end{array}
		\right.\quad \rho_0^{(g)}(n)=\left\{ \begin{array}{ll}
			|\pi^{3}n|^{-\frac 14},&\ n=-m^2<0;\\
			0,&\ \text{other }n\neq 0.
		\end{array}
		\right.
	\end{equation}
	Similarly, if $f\in \LEigenform_{\frac 12}(\Gamma_0(8),\nu_{\theta},\frac 3{16})$, $g\in \LEigenform_{-\frac 12}(\Gamma_0(8),\overline{\nu_{\theta}},\frac 3{16})$, and $\langle f,f\rangle_{\Gamma_0(8)}=\langle g,g\rangle_{\Gamma_0(8)}=1$, then
	\begin{equation}\label{eqMaass-form-coeffs-thetapm-Level8}
		\rho_0^{(f)}(n)=\left\{ \begin{array}{ll}
			(4\pi^{3}n)^{-\frac 14},&\ n=m^2>0;\\
			0,&\ \text{other }n\neq 0,
		\end{array}
		\right.\quad \rho_0^{(g)}(n)=\left\{ \begin{array}{ll}
			|4\pi^{3}n|^{-\frac 14},&\ n=-m^2<0;\\
			0,&\ \text{other }n\neq 0.
		\end{array}
		\right.
	\end{equation}
\end{lemma}

For the discrete spectrum of hyperbolic Laplacian, other than $\lambda_0$, Selberg conjectured \cite{selberg} that $\lambda_1\geq\frac 14$ for $\Delta_0$ on $(\Gamma,\text{id})$ for all congruence subgroups $\Gamma$ of $\SL_2(\Z)$ and showed that $\lambda_1\geq\frac 3{16}$. We call any $\lambda\in (\lambda_0,\frac 14)$ as an exceptional eigenvalue. The best progress is known today is $\lambda_1\geq \frac 14-\(\frac 7{64}\)^2$ for $(\Gamma_0(N),\text{id})$, for all $N\in \Z_+$ by \cite{KimSarnak764}. For small $N$, the following fact is also known. 
\begin{lemma}\cite{HuxleyKloostermania,BookerStrombergsson2007}
	\label{lemmaSelberg-eigenvalue}
	There is no exceptional eigenvalue on $\Gamma_0(N)$ for $N\leq 18$, i.e. Selberg's eigenvalue conjecture is known to be true for $N\leq 18$.   
\end{lemma}

\subsection{Some useful elementary inequalities}
For $r\in \R$, let $\ceil{r}$ denote the smallest integer larger than or equal to $r$, and $\floor{r}$ denote the largest integer smaller than or equal to $r$. Let $\|r\|=\min(r-\floor{r},\ceil{r}-r)$ denote the distance from $r$ to its closest integer. 

The following lemmas will be helpful in the proofs. 
\begin{lemma}\label{lemmaExponential-interpolation}
	For $r\in \R\setminus \Z$, $a,c\in \Z$, $-c\leq a < c$, we have
	\begin{equation}
		\ee\(\frac{ra}{2c}\)=\frac 1{2c}\sum_{k\Mod {2c}} \frac{2i\sin(\pi(r-k))}{\ee(\frac{r-k}{2c})-1} \ee\(\frac{ka}{2c}\). 
	\end{equation}
\end{lemma}
\begin{proof}
	On the right hand side, note that the summand (as a function of $k$) is periodic modulo~$2c$, so the sum is well-defined. We define and compute the discrete Fourier transform: 
	\[x_c(k)\defeq \sum_{a=-c}^{c-1} \ee\(\frac{ra}{2c}\)\ee\(-\frac{ka}{2c}\)=\frac{2i\sin(\pi(r-k))}{\ee(\frac{r-k}{2c})-1}. \]
	It is then straighforward to verify that for any $d\in \Z$ and $d\equiv a\Mod {2c}$, 
	\[\frac 1{2c}\sum_{k\Mod {2c}} x_c(k) \ee\(\frac{kd}{2c}\)=\ee\(\frac{ra}{2c}\).\qedhere \]
\end{proof}

\begin{remark}
	When applying the lemma above, we will usually require the range $|r-k|<c$ for $k \Mod {2c}$, because $z=\frac {r-k}{2c}$ will satisfy $|z|<\frac 12$ and we would be able to use the expansion from the lemma that follows. 
\end{remark}

\begin{lemma}\label{lemmaBernoulli-numbers-expansion}
	For $z\in \C$ and $0<|z|<1$, we have 
	\[\frac 1{\ee(z)-1}=\frac 1{2\pi i z}+\sum_{\ell=1}^\infty B_{\ell}\frac{(2\pi i z)^{\ell-1}}{\ell!},\]
	where $B_\ell$ is the $\ell$-th Bernoulli number, $B_0=1$, $B_1=-\frac 12$. 
\end{lemma}

We will also need the basic inequality
\begin{equation}\label{eqInequalityX+YpowAlpha}
	 (X+Y)^\alpha\leq X^\alpha+Y^\alpha \quad \text{for }X,Y>0\text{ and }\alpha\in [0,1]
\end{equation}
and the following lemma. 
\begin{lemma}\label{lemmaInequality-estimate-kAlpha-to-r-to-x}
	Let $r\in \R_+\setminus \Z$ and $x\geq 1$. Then for $\alpha\in (0,1]$ we have 
	\[\left. \begin{array}{r}
		\displaystyle \sum_{k=\ceil{r-x}}^{\floor{r}}\frac {k^\alpha}{r-k}\\
		\displaystyle \sum_{k=\ceil{r}}^{\floor{r+x}}\frac {k^\alpha}{k-r}
	\end{array}\right\}
	\ll_\ep \frac{(r^\alpha+1)x^\ep}{\|r\|}+\frac{x^\alpha}{\alpha}\quad \text{for any }\ep>0. \]
\end{lemma}
\begin{proof}
	These inequalities are helpful: $\sum_{u=1}^n u^{-1} \leq \log n+1$, and for $\alpha\in(0,1]$, 
	\[\sum_{u=1}^n u^{\alpha-1}\leq 1+\int_1^n x^{\alpha-1}dx\leq 1+\frac{n^\alpha}{\alpha}. \]
	First we consider $\sum_{k=\ceil{r-x}}^{\floor{r}}\frac {k^\alpha}{r-k}$. For $k=\floor{r}$, the summand is bounded by $r^\alpha/\|r\|$ (note that $0<r<1$ implies $\floor{r}=0$). For the other terms we have
	\[\sum_{k=\ceil{r-x}}^{\floor{r}-1}\frac {k^\alpha}{r-k}\leq \sum_{u=1}^{\floor{x}}\frac {r^\alpha+u^\alpha}{u}\leq r^\alpha(1+\log x)+\frac{x^\alpha}{\alpha}+1. \] 
	Next we prove the bound for $\sum_{k=\ceil{r}}^{\floor{r+x}}\frac {k^\alpha}{k-r}$. The first term $k=\ceil{r}$ gives
	\[\frac{\ceil{r}^\alpha}{\ceil{r}-r}\leq \frac{(r+1)^\alpha}{\|r\|}\leq \frac{r^\alpha+1}{\|r\|}\]
	by \eqref{eqInequalityX+YpowAlpha}. The remaining terms are bounded by
	\[\sum_{k=\ceil{r}+1}^{\floor{r+x}}\frac {k^\alpha}{k-r}\leq \sum_{u=1}^{\floor{x}}\frac {u^\alpha+r^\alpha+1}{u}\leq (r^\alpha+1)(\log x+1)+\frac{x^\alpha}{\alpha}+1.\]

    Combining the calculations above we get
    \begin{equation}
        \left. \begin{array}{r}
		\displaystyle \sum_{k=\ceil{r-x}}^{\floor{r}}\frac {k^\alpha}{r-k}\\
		\displaystyle \sum_{k=\ceil{r}}^{\floor{r+x}}\frac {k^\alpha}{k-r}
	\end{array}\right\}
	\leq \frac {r^\alpha+1}{\|r\|}+(r^\alpha+1)(\log x+1)+\frac{2x^\alpha}{\alpha}+1. 
    \end{equation}
    The concluded bound in the lemma is clear. 
\end{proof}

\section{Construction of modular integrals via Poincar\'e-type sums}
\label{SectionMaassPoincareSeries}

In this section we recall Stoller's construction of modular integrals for the group $\Gamma(2)$ via Poincar\'e-type series. First, we recall a few basic notions on cocycles and modular integrals and explain the connection to Theorem~\ref{theoremBonRadSeip-2022}.

Given a Fuchsian group $\Gamma\subset\operatorname{PSL}_2(\R)$
acting on functions on the upper half-plane via slash operators 
\[|\gamma=|_{k,\nu}\gamma:\quad (f|\gamma)(\tau)=\overline{\nu(\gamma)} j(\gamma,\tau)^{-k} f(\gamma\tau)\]
(for simplicity, we will omit the weight and the multiplier system from notation), we say that a collection of functions $\{f_{\gamma}\}_{\gamma\in\Gamma}$ is a cocycle if
    \[f_{\gamma_1\gamma_2} = f_{\gamma_1}|\gamma_2+f_{\gamma_2}\,,\qquad \gamma_1,\gamma_2\in\Gamma.\]
A cocycle $\{f_{\gamma}\}_{\gamma}$ is called trivial if there exists a function $F$ such that $f_{\gamma}=F-F|\gamma$ for all $\gamma\in\Gamma$. If this is the case, we say that $F$ is a modular integral for the cocycle $\{f_{\gamma}\}_{\gamma}$. We are usually interested in cocycles restricted to a certain class of functions on $\HH$, for example, with all $f_{\gamma}$ holomorphic and of moderate growth, in which case we require $F$ to belong to the same class. (The terminology modular integral comes from Eichler integrals, which are modular integrals for polynomial-valued cocycles). 
    
For the group $\Gamma(2)$, which is free on two generators $T^2$ and $ST^2S$, any cocycle is uniquely determined by an arbitrary choice of two functions $\phi_1=f_{T^2}$ and $\phi_2=f_{ST^2S}$. It follows from the results of Knopp~\cite{knopp1974some} (for weights $k>2$), and, e.g., from~\cite[Theorem~3.1]{BondarenkoRadchenkoSeip2022} (for all $k\ge0$), that any cocycle (with values in holomorphic functions of moderate growth) for the group $\Gamma(2)$ in non-negative weight is trivial. Hence for any holomorphic, moderately growing $\phi_i$, $i=1,2,$ there exists a holomorphic function of moderate growth $F$ such that
    \[
    \begin{cases}
    F-F|T^2 = \phi_1,\\
    F-F|ST^2S = \phi_2.
    \end{cases}
    \]
Choosing $\phi_1=0$ and $\phi_2(\tau)=\phi-\phi|ST^2S$, where $\phi(\tau)=e^{\pi i r^2 \tau}$ then leads to functions satisfying the conditions of Theorem~\ref{theoremBonRadSeip-2022}. In~\cite{RadchenkoViazovska2019} and~\cite{BondarenkoRadchenkoSeip2022} modular integrals are constructed using contour integrals against modular Green functions, which is convenient for proving analytic properties, but is hard to use to estimates the size of the Fourier coefficients of~$F$. In~\cite{StollerInterpolation} it is shown that $F$ can be constructed as $F=\sum_{\gamma\in\mathcal{B}}\phi|\gamma$, for a suitably chosen set $\mathcal{B}$ of coset representatives for $\Gamma(2)_\infty \setminus \Gamma(2)$ (since $\phi$ is not assumed to be periodic, the choice of representatives matters). We call these Poincar\'e-type series since they becomes the usual Poincar\'e series in the case when $\phi$ is a 2-periodic function.

Unfortunately, the above Poincar\'e-type series only converges when the weight $k$ is larger than $2$ and the resulting modular integral does not agree with the one obtained by contour integrals. We overcome these difficulties by considering $F_s=\sum_{\gamma\in\mathcal{B}}\phi_s|\gamma$ (a Maass--Poincar\'e-type series) with $\phi_s$ an eigenfunction of the weight $k$ hyperbolic Laplace operator. We show that the series converges for $\re(s)>1$, and via its Fourier expansion $F_s$ admits an analytic continuation (in the variable $s$) to the point $s=k/2$. For this value of $s$ the function $\phi_s$ specializes to $e^{\pi i r^2\tau}$ and $F_s$ becomes harmonic. By analyzing the specialization more carefully we are then able to precisely identify the holomorphic part of $F_{k/2}$ and relate it to the solutions from Theorem~\ref{theoremBonRadSeip-2022}.
    
\subsection{A special choice of coset representatives}
\label{subsectionCosetRepB}
Following \cite{StollerInterpolation}, we make a special choice of coset representatives for $\Gamma(2)_\infty \setminus \Gamma(2) $ and $(\Gamma(2)_\infty \setminus \Gamma(2))S $. Recall the notation from \S\ref{sectionPrelimandNotation}:  $\overline{\Gamma}=\Gamma/\{\pm I\}$, $A=[T^2]$ and $B=[ST^2S]$ are the classes of $T^2$ and $ST^2S$, respectively. Also recall that $\Gamma_{\Theta}=\Gamma(2)\sqcup \Gamma(2)S$.  
\begin{definition}[{\cite[Definition~5.1]{StollerInterpolation}}]\label{defMB-and-tildeMB}
	The subset $\mB\subset \overline{\Gamma(2)}$ is defined as the set of all nonempty finite reduced words in $A$ and $B$ that start with a nonzero power of $B$. More formally, an element $\gamma \in \overline{\Gamma(2)}$ belongs to $\mB$, if and only if there are integers $m\geq 1$, and $e_1,\cdots, e_m,f_1,\cdots,f_m$, all non-zero, except possibly $e_m$, such that $\gamma=B^{f_1}A^{e_1}\cdots B^{f_m} A^{e_m}$. We also define
	\[\widetilde{\mB}\defeq \mB[S]\sqcup\{[S]\}=\{\gamma [S]:\gamma\in \mB\}\sqcup\{[S]\}\subset \overline{\Gamma(2)}. \]
\end{definition}

\begin{proposition}[{\cite[Lemma~5.1, Lemma~5.2]{StollerInterpolation}}]
	\label{propMB-and-tildeMB-properties}
	We have the following properties. 
	\begin{itemize}
		\item[(1)] For each $\gamma\in \mB$, each $\widetilde \gamma \in\widetilde{\mB}$, and each $\ell\in \Z$,  one has $\gamma A^\ell\in \mB $ and $\widetilde \gamma A^\ell \in \widetilde{\mB}$. 
		\item[(2)]
	Define
	\begin{align*}
		\mathcal{P}&=\{(c,d)\in \Z^2:\ \gcd(c,d)=1,\ c\equiv 0,\ d\equiv 1\Mod 2,\ c\neq 0\},\\
		\mathcal{P}_I&=\mathcal{P}\sqcup\{(0,1),\ (0,-1)\},\\
		\widetilde {\mathcal{P}}&=\{(c,d)\in \Z^2:\ \gcd(c,d)=1,\ c\equiv 1,\ d\equiv 0\Mod 2\}.
	\end{align*} 
	Then 
	\[\left[\begin{pmatrix}
		a&b\\c&d
	\end{pmatrix}\right] \rightarrow [(c,d)]\]
	defines $\Z$-equivariant bijections $\mB\cong \mathcal{P}/\{\pm 1\}$, $\mB\sqcup\{[ I]\}\cong \mathcal{P}_I/\{\pm 1\}$ and $\widetilde{\mB}\cong \widetilde {\mathcal{P}}/\{\pm 1\}$. 
	\item[(3)] For every $\gamma=\left[\begin{psmallmatrix}
		a&b\\c&d
	\end{psmallmatrix}\right]\in \mB$, we have $|a|<|c|$ and $|b|<|d|$; For every $\widetilde{\gamma}=\left[\begin{psmallmatrix}
	\widetilde{a}&\widetilde{b}\\ \widetilde{c}&\widetilde{d}
	\end{psmallmatrix}\right]\in \widetilde{\mB}$, we have $|\widetilde{a}|<|\widetilde{c}|$. 
	\end{itemize}
\end{proposition}

\begin{remark}
    If we denote by $\mathcal{D}=\{z\in\HH\colon |\re(z)|<1, |z\pm1/2|>1/2\}$ the standard fundamental domain for $\Gamma(2)$, then from the figure below we can see that this choice of coset representatives ensures that $\{\gamma\mathcal{D}:\gamma\in \mathcal{B}\}$ fill the region $\mathcal{D}_{\infty} = \{z\in \HH: |\re(z)| < 1\}$
    up to a measure-zero set. In other words, when doing the ``unfolding trick" to compute the Fourier expansion of Maass--Poincar\'e series in \S\ref{subsectionComputeFourierExpofMaassPoincare} below, we end up with an integral over $\mathcal{D}_{\infty}$, same as in the classical case.
\end{remark}
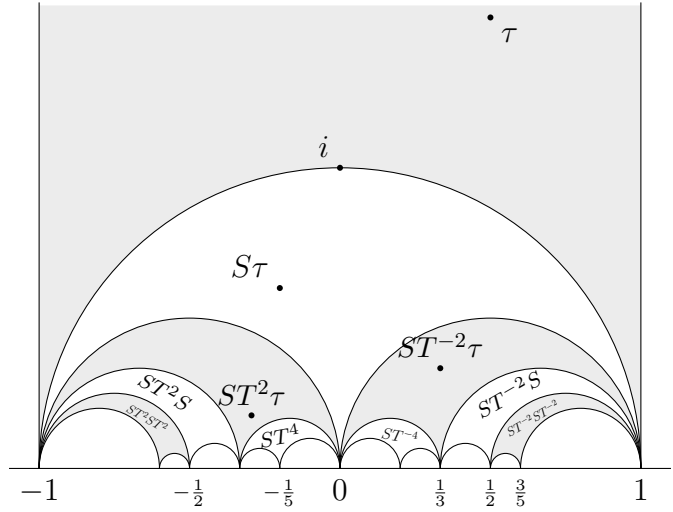
\begin{center}
\begin{tikzpicture}[scale=4] 
    \fill[gray!15]
    (-1,1.54) -- (1,1.54)  
    -- (1,0) arc[start angle=0, end angle=180, radius=1] -- cycle;
    
    \fill[gray!15, even odd rule]
    (1,0) arc[start angle=0, end angle=180, radius=0.5]
    -- cycle
    (1,0) arc[start angle=0, end angle=180, radius=0.333334]
    -- cycle
    (0.333334,0) arc[start angle=0, end angle=180, radius=0.166667]
    -- cycle;

    \fill[gray!15, even odd rule]
    (0,0) arc[start angle=0, end angle=180, radius=0.5]
    -- cycle
    (-0.333334,0) arc[start angle=0, end angle=180, radius=0.333334]
    -- cycle
    (0,0) arc[start angle=0, end angle=180, radius=0.166667]
    -- cycle;

    \fill[gray!15, even odd rule]
    (-0.5,0) arc[start angle=0, end angle=180, radius=0.25]
    -- cycle
    (-0.6,0) arc[start angle=0, end angle=180, radius=0.2]
    -- cycle
    (-0.5,0) arc[start angle=0, end angle=180, radius=0.05]
    -- cycle;

    \fill[gray!15, even odd rule]
    (1,0) arc[start angle=0, end angle=180, radius=0.25]
    -- cycle
    (0.6,0) arc[start angle=0, end angle=180, radius=0.05]
    -- cycle
    (1,0) arc[start angle=0, end angle=180, radius=0.2]
    -- cycle;
    
  \draw (-1.1,0) -- (1.1,0);
  \draw (1,0) -- (1,1.55) ;    
  \draw (-1,0) -- (-1,1.55) ;   
  \fill (0.5, 1.5) circle[radius=0.01];
  \node[below right] at (0.5, 1.5) {$\tau$};

  \draw (1,0) arc[start angle=0, end angle=180, radius=1]; 
  \fill (-0.2, 0.6) circle[radius=0.01];
  \node[above left] at (-0.2, 0.6) {$S\tau$};
  \draw (0,0) arc[start angle=0, end angle=180, radius=0.5];
  \draw (1,0) arc[start angle=0, end angle=180, radius=0.5];
  \fill (-0.294118, 0.176471) circle[radius=0.01];
  \node[above] at (-0.294118, 0.176471) {\small $ST^2\tau$};
  \fill (0.333333, 0.333333) circle[radius=0.01];
  \node[above] at (0.333333, 0.333333) {\small $ST^{-2}\tau$};
  \draw (-0.333333,0) arc[start angle=0, end angle=180, radius=0.333333];
  \draw (0,0) arc[start angle=0, end angle=180, radius=0.166667];
  \draw (0.333333,0) arc[start angle=0, end angle=180, radius=0.166667];
  \draw (1,0) arc[start angle=0, end angle=180, radius=0.333333];

  \node[rotate=-30] at (-0.58, 0.26) {\tiny $ST^{2}S$};
  \draw (-0.5,0) arc[start angle=0, end angle=180, radius=0.25];
  \draw (-0.333333,0) arc[start angle=0, end angle=180, radius=0.083333];
  
  \node[rotate=10] at (-0.2,0.11) {\tiny $ST^{4}$};
  \draw (-0.2,0) arc[start angle=0, end angle=180, radius=0.066667];
  \draw (0,0) arc[start angle=0, end angle=180, radius=0.1];

  \node[rotate=-10] at (0.2,0.11) {\scalebox{0.5}{$ST^{-4}$}};
  \draw (0.2,0) arc[start angle=0, end angle=180, radius=0.1];
  \draw (0.333333,0) arc[start angle=0, end angle=180, radius=0.066667];
  
  \node[rotate=30] at (0.56, 0.25) {\tiny $ST^{-2}S$};
  \draw (1,0) arc[start angle=0, end angle=180, radius=0.25];
  \draw (0.5,0) arc[start angle=0, end angle=180, radius=0.083333];

  \node[rotate=-30] at (-0.64, 0.168) {\scalebox{0.4}{$ST^{2}ST^2$}};
  \draw (-0.6,0) arc[start angle=0, end angle=180, radius=0.2];
  \draw (-0.5,0) arc[start angle=0, end angle=180, radius=0.05];

  \node[rotate=30] at (0.64, 0.168) {\scalebox{0.4}{$ST^{-2}ST^{-2}$}};
  \draw (0.6,0) arc[start angle=0, end angle=180, radius=0.05];
  \draw (1,0) arc[start angle=0, end angle=180, radius=0.2];

  \node[below] at (-1,0) {$-1$};
  \node[below] at (1,0) {$1$};
  \node[below] at (0,0) {$0$};
  \fill (0, 1) circle [radius=0.01];
  \node[above left] at (0,1) {$i$};
  \node[below] at (-0.5,0) {\tiny $-\frac 12$};
  \node[below] at (0.5,0) {\tiny $\frac 12$};
  \node[below] at (0.6,0) {\tiny{$\frac 35$}};
  \node[below] at (0.333333,0) {\tiny{$\frac 13$}};
  \node[below] at (-0.2,0) {\tiny{$-\frac 15$}};
\end{tikzpicture}
\end{center}

\subsection{Real-variable Kloosterman sums}

In the proof we need to define real-variable Kloosterman sums analogous to \eqref{eqKloosterman-Sums-General-Def}. Let $r\in \R$, $n\in \Z$, and $k\in \frac 12\Z$. For $2|c>0$, we define
\begin{align}\label{eqReal-Val-Kloosterman-Sums}
	\begin{split}
	K(r,n,c,\nu_{\Theta}^{2k})&\defeq \sum_{\substack{-c<a<c\\(a,2c)=1\\ad\equiv 1\Mod {2c}}} \nu_{\Theta}\begin{psmallmatrix}
		a&*\\c&d
	\end{psmallmatrix}^{-2k} \ee\(\frac{ra+nd}{2c}\)\\
	&= \sum_{\substack{-c<a<c\\(a,2c)=1\\ad\equiv 1\Mod {2c}}} \ep_d^{2k}\(\frac{2c}d\)^{-2k} \ee\(\frac{ra+nd}{2c}\);
	\end{split}
\end{align}
for $2\nmid d>0$, we define
\begin{align}\label{eqReal-Val-Kloosterman-Sums-tilde}
\begin{split}
		\widetilde K(r,n,d,\nu_{\Theta}^{2k})&\defeq \sum_{\substack{-d<b<d\\2|b,\ (b,d)=1\\2|c,\ bc\equiv -1\Mod d}}\nu_{\Theta}\begin{psmallmatrix}
		b&*\\d&-c
	\end{psmallmatrix}^{-2k} \ee\(\frac{rb-nc}{2d}\)\\
	&= \sum_{\substack{-d<b<d\\2|b,\ (b,d)=1\\2|c,\ bc\equiv -1\Mod d}} \ee(\tfrac k4) \ep_d^{-2k} \(\frac{-2c}d\)^{-2k} \ee\(\frac{rb-nc}{2d}\).
\end{split}
\end{align}
As a special case, we set $\widetilde{K}(r,n,1,\nu_{\Theta}^{2k})=\ee(\frac k4)$.

\begin{remark}
	Here are a few remarks about real-variable Kloosterman sums. 
	\begin{itemize}
		\item[(1)] Note that we specify $-c<a<c$ in \eqref{eqReal-Val-Kloosterman-Sums} and $-d<b<d$ in \eqref{eqReal-Val-Kloosterman-Sums-tilde}. This restriction is important because $r$ is a real number and changing $a\rightarrow a+2c$ (or $b\rightarrow b+2d$) will change the value of the sum. This aligns with part (3) of Proposition~\ref{propMB-and-tildeMB-properties}. 
		\item[(2)] The function $K(r,n,c,\nu_{\Theta}^{2k})$ is the real-variable analogue of $S(m,n,c,\nu_{\Theta}^{2k})$ for $2|c$, in the sense that setting $r=m$ specializes $K$ to $S$, and the function $\widetilde{K}(r,n,d,\nu_{\Theta}^{2k})$ is the real-variable analogue of $S(m,n,d,\nu_{\Theta}^{2k})$ for $2\nmid d$. 
        \item[(3)] This definition of real-variable Kloosterman sums originated from \cite[(8.4)]{StollerInterpolation}. Stoller's estimate focused on the properties of Poincar\'e series. In this paper, we prove properties of these Kloosterman sums directly with more precise results. 
        \item[(4)] To avoid ambiguity about the meaning of $d$, we will use $\widetilde{K}(m,n,\widetilde{c},\nu_{\Theta}^{2k})$ for $2\nmid \widetilde{c}>0$ in the rest of the paper. 
	\end{itemize}
\end{remark}

\subsection{Maass--Poincar\'e type series}
Let $M_{\beta,\mu}$ and $W_{\beta,\mu}$ denote the $M$- and $W$-Whittaker functions (for definition see \cite[(13.14.2-3)]{dlmf}). For $s\in \C$, $x,y\in \R$, $z=x+iy$ and $k\in \frac 12 \Z$, we define
\begin{equation}
	\Mform_s(y)\defeq |y|^{-\frac k2} M_{\frac k2\sgn y,\,s-\frac 12}(|y|) \quad \text{and}\quad \mathcal{W}_s(y)\defeq |y|^{-\frac k2} W_{\frac k2\sgn y,\,s-\frac 12}(|y|) . 
\end{equation}
We also define
\begin{equation}
	\varphi_{s,k}(x+iy)\defeq \Mform_s(4\pi y)\ee(x). 
\end{equation}
These functions have the following properties. For $y>0$, by \cite[(13.18.2)]{dlmf}, we have 
\begin{equation}\label{eqMWhittaker-special-form}
	\Mform_{\frac k2}(y)=y^{-\frac k2} M_{\frac k2,\,\frac k2-\frac 12}(y)=e^{-\frac y 2}; 
\end{equation}
by \cite[(13.14.31), (13.18.2)]{dlmf}, we have 
\begin{equation}\label{eqWWhittaker-special-form}
	W_{\kappa,\mu}(z)=W_{\kappa,-\mu}(z),\quad W_{-\frac k2,\,\frac k2-\frac 12}(y)=y^{\frac k2}e^{\frac y2}\Gamma(1-k,y)\  \text{and}\  W_{\frac k2,\,\frac k2-\frac 12}(y)=y^{\frac k2}e^{-\frac y2}; 
\end{equation}
moreover, by \cite[(13.14.14)]{dlmf}, we have
\begin{equation}\label{eqWhittakerVarphi-growthrate}
	\varphi_{s,k}(z)=O\(y^{\re s-\frac k2}\) \quad \text{for }z\rightarrow 0. 
\end{equation}

Recall Definition~\ref{defMB-and-tildeMB} for $\mB$ and $\widetilde{\mB}$. For $\tau=x+iy\in \HH$, $r\in \R\setminus \{0\}$, $k\in \frac 12 \Z$, and $\re(s)>1$, we define the series
\begin{align}\label{eqDefineFks}
	\begin{split}
	&F_k(\tau;r,s)\defeq -\sum_{\gamma =\begin{psmallmatrix}
			*&*\\c&d
		\end{psmallmatrix}\in \mB} \nu_{\Theta}(\gamma )^{-2k}(c\tau+d)^{-k} \varphi_{s,k}(\tfrac{r^2}2\gamma \tau), \\
	&\widetilde{F}_k(\tau;r,s)\defeq \sum_{\widetilde{\gamma} =\begin{psmallmatrix}
			*&*\\c&d
		\end{psmallmatrix}\in \widetilde {\mB}}  \nu_{\Theta}(\gamma)^{-2k}(c\tau +d)^{-k} \varphi_{s,k}(\tfrac{r^2}2 \widetilde{\gamma }\tau). 
	\end{split}
\end{align}
These two series converge uniformly and absolutely for $s$ in any compact subset of $\{s\in \C: \re (s)>1\}$ due to \eqref{eqWhittakerVarphi-growthrate}. 
Note that $\mB$ and $\widetilde{\mB}$ are defined up to the equivalence relation modulo $\{\pm I\}$. The series in \eqref{eqDefineFks} are well-defined because 
\[\nu_{\Theta}(-\gamma)^{-1}(-cz-d)^{-\frac 12}=-i\nu_{\Theta}(\gamma)^{-1}\cdot i(cz+d)^{-\frac 12}=\nu_{\Theta}(\gamma)^{-1}(cz+d)^{-\frac 12}\]
by \eqref{eqMultiplierOn-Minus-gamma} for $c>0$ and $\nu_{\Theta}\begin{psmallmatrix}
	*&*\\0&*
\end{psmallmatrix}=1$.

For all $r\in \R$, we have the following functional equation. 
\begin{proposition}\label{propFunctionalequation-ats}
	For $r\in \R$, $k\in \tfrac 12\Z$, $\tau=x+iy\in \HH$, and for $F_k(\tau,r;s)$, $\widetilde F_k(\tau,r;s)$ defined at \eqref{eqDefineFks}, we have
	\begin{equation}
		F_k(\tau;r,s)+(-i\tau)^{-k}\widetilde F_k(-\tfrac 1{\tau};r,s)=\varphi_{s,k}(\tfrac{r^2}2\tau)=(2\pi r^2 y)^{-\frac k2} M_{\frac k2,s-\frac 12}(2\pi r^2 y) e^{\pi i r^2 x}. 
	\end{equation}
\end{proposition}
\begin{proof}
	We start with $\widetilde{F}_k(-\frac 1\tau;r,s)$. By Proposition~\ref{propMB-and-tildeMB-properties}, we choose representatives of $\mathcal{P}/\Z^\times$ and $\widetilde{\mathcal{P}}/\Z^\times$ as
	\begin{align*}
		&\{(c,d)\in \Z\times \Z: c=2,4,6,\cdots, \ (c,d)=1\}, \quad \text{and}\\
		\{(1,0)\}\cup &\{(c,-d)\in \Z\times \Z: d=2,4,6,\cdots, \ (c,d)=1\},\quad  \text{respectively.}
	\end{align*}
	The corresponding representatives of $\mB$ and $\widetilde{\mB}$ are chosen by
	\begin{align}
		\label{eqCosetRepof-MB}
		&\mathfrak{B}\defeq \{\begin{psmallmatrix}
			a&b+2ta\\c&d+2tc
		\end{psmallmatrix}\in \Gamma_{\Theta}: c\in 2\Z_+, \ d\Mod{2c}^*,\ ad\equiv 1\Mod{2c},\ t\in \Z\},\\
		\label{eqCosetRepof-tildeMB}
		&\text{and} \quad \{\begin{psmallmatrix}
			0&-1\\1&0
		\end{psmallmatrix}\}\cup \mathfrak{B}S,\quad \text{respectively}. 
	\end{align}
	For $\tau\in \HH$ and $\gamma=\begin{psmallmatrix}
		*&*\\c&d
	\end{psmallmatrix}\in \Gamma_{\Theta}$ with $c>0$, by \eqref{eqNuTheta-gammaS-property} we get
	\begin{align*}
		\nu_{\Theta}(\gamma S) \(d(-\tfrac 1\tau)-c\)^{\frac 12}&=(\sgn d)\ee(-\tfrac 18)\nu_{\Theta}(\gamma)\cdot (\sgn d)(-\tfrac 1\tau)^{\frac 12}(c\tau+d)^{\frac 12}\\
		&=\ee(\tfrac 18)\tau^{-\frac 12}\nu_{\Theta}(\gamma)(c\tau+d)^{\frac 12}
	\end{align*}
	Hence we have
	\begin{align*}
		&\widetilde{F}_k(-\tfrac 1\tau;r,s)=\nu_{\Theta}(S)^{-2k}(-\tau^{-1})^{-k} \varphi_{s,k}(\tfrac{r^2}2 \tau)\\
		&+\sum_{2|c>0} \sum_{t\in \Z} \sum_{\substack{d\Mod{2c}^*\\ad\equiv 1(2c)}} \nu_{\Theta}\(\begin{psmallmatrix}
			a&b+2ta\\c&d+2tc
		\end{psmallmatrix}S\)^{-2k}\frac{\varphi_{s,k}\(\frac{r^2}2\begin{psmallmatrix}
			a&b+2ta\\c&d+2tc
		\end{psmallmatrix}S(-\frac 1\tau)\)}{\((d+2tc)(-\frac 1\tau)-c\)^{k}}\\
	&=\ee(\tfrac k4)e^{-\pi i k}\tau^k\varphi_{s,k}(\tfrac{r^2}2 \tau)+\sum_{\gamma\in \mB}\ee(-\tfrac k4)\tau^k \nu_{\Theta}(\gamma)^{-2k}\frac{\varphi_{s,k}(\tfrac{r^2}2 \gamma \tau)}{(c\tau +d)^{k}} \\
	&=\ee(-\tfrac k4)\tau^k\(\varphi_{s,k}(\tfrac{r^2}2 \tau)-F_k(\tau;r,s)\). 
	\end{align*}
	We finish the proof by noticing $(-i\tau)^k = \ee(-\frac k4)\tau^k$ for $\tau\in \HH$. 
\end{proof}

In Proposition~\ref{propFunctionalequation-ats}, if we take $s=\frac k2$, by \eqref{eqMWhittaker-special-form} we have
\begin{equation}\label{eqFunctionalequation-atsk2}
	F_k(\tau; r,\tfrac k2)+(-i\tau)^{-k}\widetilde F_k(-\tfrac 1{\tau}; r,\tfrac k2)= e^{\pi i r^2 \tau}. 
\end{equation}
This is clear if the weight $k>2$ because $s=\frac k2>1$. Since we are focusing on the case $k=2$ and $k=\frac 32$ in this paper, by computing their Fourier expansions in Lemma~\ref{lemmaFourierExpansionAts}, we are able to prove the following two theorems, Theorem~\ref{theoremFourierExpansionForWeight2Ats=1} and Theorem~\ref{theoremFourierExpansionForWeight3/2Ats=3/4}, which show that $F_2(\tau;r,s)$ and $\widetilde F_2(\tau;r,s)$ can be analytically continued to $s=1$ and $F_{\frac 32}(\tau;r,s)$ and $\widetilde F_{\frac 32}(\tau;r,s)$ can be analytically continued to $s=\frac 34$, respectively. After proving these theorems, we conclude that \eqref{eqFunctionalequation-atsk2} still holds for $k=2$ and $\frac 32$ by analytic continuation.

\subsection{Computation of Fourier expansions}
\label{subsectionComputeFourierExpofMaassPoincare}
\begin{lemma}\label{lemmaFourierExpansionAts}
	Let $\tau=x+iy\in \HH$. For $n\in \Z$, we define
    \begin{align*}
        &B_{2k,n}(r;s,y)\defeq -\Gamma(2s)\pi 
    \ee(-\tfrac k4)\cdot \\
        &\left\{\begin{array}{ll}
              \displaystyle \frac{W_{\frac k2,s-\frac 12}(2\pi n y)}{\Gamma(s+\frac k2)(2\pi r^2y)^{\frac k2} } \left|\frac{r^2}n\right|^{\frac {1}2} \sum_{2|c>0}\frac{K(r^2,n,c,\nu_{\Theta}^{2k})}c J_{2s-1}\(\frac{2\pi |r|\sqrt n}{c}\), & n>0, \\
		   \displaystyle \frac{(\pi r^2)^{s-\frac k2}2^{-s-\frac k2}y^{1-s-\frac k2}}{\Gamma(s+\frac k2)\Gamma(s-\frac k2)(s-\frac 12)}\sum_{2|c>0}\frac{K(r^2,0,c,\nu_{\Theta}^{2k})}{c^{2s}}, & n=0, \\
		  \displaystyle \frac{W_{-\frac k2,s-\frac 12}(2\pi |n| y)}{\Gamma(s-\frac k2)(2\pi r^2y)^{\frac k2}}   \left|\frac{r^2}n\right|^{\frac {1}2} \sum_{2|c>0}\frac{K(r^2,n,c,\nu_{\Theta}^{2k})}c I_{2s-1}\(\frac{2\pi |r^2n|^{\frac 12}}{c}\), & n<0, 
        \end{array}\right. 
    \end{align*}
    and 
    \begin{align*}
        &\widetilde{B}_{2k,n}(r;s,y)\defeq \Gamma(2s)\pi \ee(-\tfrac k4)\cdot\\
        &\left\{\begin{array}{ll}
             \displaystyle \frac{W_{\frac k2,s-\frac 12}(2\pi n y)}{\Gamma(s+\frac k2)(2\pi r^2y)^{\frac k2} } \left|\frac{r^2}n\right|^{\frac {1}2} \sum_{2\nmid d>0}\frac{\widetilde{K}(r^2,n,d,\nu_{\Theta}^{2k})}d J_{2s-1}\(\frac{2\pi |r|\sqrt n}{d}\), & n>0, \\
		      \displaystyle \frac{(\pi r^2)^{s-\frac k2}2^{-s-\frac k2}y^{1-s-\frac k2}}{\Gamma(s+\frac k2)\Gamma(s-\frac k2)(s-\frac 12)}\sum_{2\nmid d>0}\frac{\widetilde{K}(r^2,0,d,\nu_{\Theta}^{2k})}{d^{2s}}, & n=0,\\
		    \displaystyle \frac{W_{-\frac k2,s-\frac 12}(2\pi |n| y)}{\Gamma(s-\frac k2)(2\pi r^2y)^{\frac k2}}   \left|\frac{r^2}n\right|^{\frac {1}2} \sum_{2\nmid d>0}\frac{\widetilde{K}(r^2,n,d,\nu_{\Theta}^{2k})}d I_{2s-1}\(\frac{2\pi |r^2n|^{\frac 12}}{d}\), & n<0. 
        \end{array}
        \right. 
    \end{align*}
    Then the Fourier expansions of $F_k(\tau,r;s)$ and $\widetilde{F}_k(\tau,r;s)$ at the cusp at $\infty$ are
	\begin{align}
    \label{eqFourierExpansionFAts-on-cusp-Infinity}
		F_k(\tau;r,s)&= \sum_{n\in \Z} B_{2k,n}(r;s,y)\; e^{\pi i n x},\\
    \label{eqFourierExpansion-tildeFAts-on-cusp-infty}
        \widetilde{F}_k(\tau;r,s)&= \sum_{n\in \Z}\widetilde{B}_{2k,n}(r;s,y)\;e^{\pi i n x} . 
	\end{align}
\end{lemma}

\begin{proof}
	When $r^2$ is an integer, this is a standard computation in the theory of harmonic Maass forms, see \cite[Proof of Theorem~1.9]{BruinierBookBorcherds}, \cite[Proof of Theorem~3.2]{BrmOno2006ivt}, \cite[Proof of Theorem~3.2]{BringmannOno2012}, \cite[Theorem~3.2]{JeonKangKim2012} and \cite[\S5.2]{QihangExactFormula} for examples. However, since we require $r$ in $\varphi_{s,k}(\frac {r^2}2 \tau)$ to be real and use real-variable Kloosterman sums \eqref{eqReal-Val-Kloosterman-Sums}, \eqref{eqReal-Val-Kloosterman-Sums-tilde} here, we provide a detailed proof.
	
	We first define an auxiliary function
	\[f(\tau;c)\defeq \sum_{t\in \Z} \frac{|\tau+2t|^k}{(\tau+2t)^k} M_{\frac k2,\,s-\frac 12}\(\frac{2\pi r^2 y}{c^2|\tau+2t|^2}\)\ee\(\frac{r^2/2}{-c^2}\re \(\frac 1{\tau+2t}\)\). \]
	This function is invariant under $\tau\rightarrow \tau+2$ and has the Fourier expansion
	\[f(\tau;c)=:\sum_{n\in \Z} f_{n}(y) e^{\pi i n x} \quad \text{for }\tau=x+iy,\]
	where the Fourier coefficient $f_n(y)$ can be computed by
	\begin{align*}
		f_n(y)&=\frac 12\int_{-1}^1 f(\tau;c)e^{-\pi i n x} dx\\
		&=\frac 12\int_{\R} \frac{|x+iy|^k}{(x+iy)^k} M_{\frac k2,\,s-\frac 12}\(\frac{2\pi r^2 y}{c^2(x^2+y^2)}\)\ee\(\frac{r^2/2}{-c^2}\re \(\frac 1{x+iy}\)-\frac{nx}2\)dx. 
	\end{align*}
	By letting $x=-yu$ where $y>0$, we have
	\begin{equation}
		\frac{|x+iy|^k}{(x+iy)^k}=\ee(-\tfrac k4)\(\frac{1-iu}{1+iu}\)^{\frac k2}. 
	\end{equation}
	Then we continue and apply \cite[Proposition~3.6]{JeonKangKim2012} to get
	\begin{align}\label{eqFourier-expansion-auxillaryfunction-ftauc}
		\begin{split}
		f_n(y)&=\frac{y\ee(-\frac k4)}2\int_{\R}\(\frac{1-iu}{1+iu}\)^{\frac k2}M_{\frac k2,\,s-\frac 12}\(\frac{2\pi r^2 }{c^2y(u^2+1)}\)\ee\(\frac{r^2u}{2c^2y(u^2+1)}+\frac{nyu}2\)du\\
		&=\left\{\begin{array}{ll}
			\displaystyle \frac{\pi \ee(-\frac k4) \Gamma(2s)}{c\cdot \Gamma(s+\frac k2)} \left|\frac{r^2}n\right|^{\frac 12}W_{\frac k2,s-\frac 12}(2\pi ny) J_{2s-1}\(\frac{2\pi |r|\sqrt n}c\),&\ n>0;\vspace{5px}\\
			\displaystyle \frac{\pi \ee(-\frac k4)\Gamma(2s) (\pi r^2)^s}{2^s c^{2s}(s-\frac 12)\Gamma(s+\frac k2)\Gamma(s-\frac k2)} y^{1-s} ,&\ n=0;\vspace{5px}\\
			\displaystyle \frac{\pi \ee(-\frac k4) \Gamma(2s)}{c\cdot \Gamma(s-\frac k2)} \left|\frac{r^2}n\right|^{\frac 12}W_{-\frac k2,s-\frac 12}(2\pi |n|y) I_{2s-1}\(\frac{2\pi |r^2 n|^{\frac 12}}c\) ,&\ n<0;
		\end{array}
		\right.
		\end{split}
	\end{align}
	
	Now we compute the Fourier expansions with the help of $f(\tau;c)$. By \eqref{eqCosetRepof-MB} and \eqref{MultiplierSystemBasicProprety}
	\begin{align*}
		F_k(\tau;r,s)&=-\sum_{2|c>0}\sum_{\substack{d\Mod{2c}^*\\-c<a<c\\ad\equiv 1\Mod{2c}}} \frac{\nu_{\Theta}\begin{psmallmatrix}
			*&*\\c&d
		\end{psmallmatrix}^{-2k}}{|2\pi r^2 y|^{\frac k2}} \\
	&\cdot \sum_{t\in \Z} \frac{|c\tau+d+2t|^k}{(c\tau+d+2t)^k}M_{\frac k2,s-\frac 12}\(\frac{2\pi r^2 y}{|c\tau+d+2t|^2}\) \ee\(\frac {r^2}2\(\frac ac-\frac{1}{c(c\tau+d+2t)}\)\)\\
	&=-\sum_{2|c>0} \sum_{\substack{d\Mod{2c}^*\\-c<a<c\\ad\equiv 1\Mod{2c}}} \frac{\nu_{\Theta}\begin{psmallmatrix}
			*&*\\c&d
		\end{psmallmatrix}^{-2k}}{|2\pi r^2 y|^{\frac k2}} \ee\(\frac{r^2a}{2c}\)f\(\tau+\frac dc;c\). 
	\end{align*}
	The choice of $a$ in $\ee(\frac{r^2a}{2c})$ is important and due to Proposition~\ref{propMB-and-tildeMB-properties}. The real-variable Kloosterman sums defined in~\eqref{eqReal-Val-Kloosterman-Sums} are therefore involved. We have
	\[f\(\tau+\frac dc;c\)=\sum_{n\in \Z} f_n(y) \ee\(\frac{nd}{2c}\) e^{\pi i n x}\]
	and \eqref{eqFourierExpansionFAts-on-cusp-Infinity} is proved by \eqref{eqFourier-expansion-auxillaryfunction-ftauc}. 
	
    The proof of \eqref{eqFourierExpansion-tildeFAts-on-cusp-infty} follows in a similar way as the proof of \eqref{eqFourierExpansionFAts-on-cusp-Infinity}, by recalling Proposition~\ref{propMB-and-tildeMB-properties}, noticing that every $\widetilde{\gamma}\in \widetilde{\mB}$ has the form
    \[\widetilde{\gamma}=\left[\begin{psmallmatrix}
		b&-a\\d&-c
	\end{psmallmatrix}\right]\quad \text{for }\left[\begin{psmallmatrix}
	a&b\\c&d
	\end{psmallmatrix}\right]\in \mB\text{ where we choose }d>0\]
	and $-d<b<d$, as well as the Fourier expansion of $f(\tau-\frac cd;d)$ with \eqref{eqFourier-expansion-auxillaryfunction-ftauc}.
\end{proof}

The following two theorems play a key role in the proofs of Theorem~\ref{theoremMainDim4} and Theorem~\ref{theoremMainDim3}. We prove them in Section~\ref{sectionProof-of-theorem-FourierExpansion-weight2-s1} and Section~\ref{sectionProof-of-theorem-FourierExpansion-weight3/2-s3/4}, respectively.

\begin{theorem}\label{theoremFourierExpansionForWeight2Ats=1}
Let $\tau=x+iy\in \HH$ and $r\in \R\setminus\{0\}$. For $s\in[1,1.001]$, both $B_{4,n}(r;s,y)$ and $\widetilde{B}_{4,n}(r;s,y)$ are convergent and bounded by
\begin{align*}
\begin{array}{r}
     |B_{4,n}(r;s,y)| \vspace{5px}\\
    |\widetilde{B}_{4,n}(r;s,y)| 
\end{array}
\ll_\ep 
\left\{\begin{array}{ll}
     n^{1+\ep} e^{-\pi y n}& \text{if }n\geq 1,\ r^2n\leq 1,\\
     (|r|^{-1}+|r|^{-\frac 12+\ep}) n^{\frac 34+\ep} e^{-\pi y n}& \text{if }n\geq 1,\ r^2n\geq 1,\\
      (|r|^{-1}+|r|^{1+\ep}) |n|^2 e^{-\pi y |n|+2\pi |r||n|^{1/2}}/\Gamma(s-1) &\text{if }n\leq -1. 
\end{array}\right.
\end{align*}
The limits $\lim_{s\rightarrow 1^+} B_{4,0}(r;s,y)$ and $\lim_{s\rightarrow 1^+} \widetilde{B}_{4,0}(r;s,y)$ both exists. 
Therefore, both expansions of $F_2(\tau;r,s)$ \eqref{eqFourierExpansionFAts-on-cusp-Infinity} and $\widetilde F_2(\tau;r,s)$ \eqref{eqFourierExpansion-tildeFAts-on-cusp-infty} are absolutely convergent for $s\in [1,1.001]$. 

Moreover, $F_{2}(\tau;r,1)$ and $\widetilde{F}_{2}(\tau;r,1)$ are explicitly given by 
\begin{equation}\label{eqFourierExpansionForWeight2Ats=1-evenc}
			F_2(\tau;r,1)=\frac{\sin(\pi r^2) }{y\pi^2 r^2}+\pi\sum_{n=1}^\infty e^{\pi i n \tau} \(\frac n{r^2}\)^{\frac 12} \sum_{2|c>0}\frac{K(r^2,n,c,\nu_{\Theta}^{4})}c J_{1}\(\frac{2\pi |r|\sqrt n}{c}\)
		\end{equation}
		and
		\begin{equation}\label{eqFourierExpansionForWeight2Ats=1-oddd}
			\widetilde{F}_2(\tau;r,1)=\frac{\sin(\pi r^2)}{y\pi^2 r^2}-\pi \sum_{n=1}^\infty e^{\pi i n \tau} \(\frac n{r^2}\)^{\frac 12} \sum_{2\nmid d>0}\frac{\widetilde K(r^2,n,{\widetilde{c}},\nu_{\Theta}^{4})}{\widetilde{c}} J_{1}\(\frac{2\pi |r|\sqrt n}{{\widetilde{c}}}\). 
		\end{equation}
		They satisfy
		\begin{equation}\label{eqFourierExpansionForWeight2Ats=1-FunctionalEquation}
			F_2(\tau;r,1)+(-i\tau)^{-2}\widetilde F_2(-1/\tau;r,1)=e^{\pi i r^2 \tau}. 
		\end{equation}
\end{theorem}

\begin{theorem}\label{theoremFourierExpansionForWeight3/2Ats=3/4}
Let $\tau=x+iy\in \HH$, and $r\in \R\setminus \{0\}$. For $s\in [\frac 34,1.001]$, both $|B_{3,n}(r;s,y)|$ and $|\widetilde{B}_{3,n}(r;s,y)|$ are convergent and bounded by
\begin{align*}
    \ll_{\ep}\left\{\begin{array}{ll}
    n^{\max(\frac{437}{588},\frac{11+\kappa}{16})+\ep} e^{-\pi y n},& \text{if } n \geq 1, \ r^2n\leq 1,\\
     (|r|^{-\frac 12}+|r|^\ep )n^{\frac 12+\ep} e^{-\pi y n},& \text{if }n\geq 1,\ r^2n\geq 1, \\
     \frac{(|r|^{-1/2}+|r|^3) |n|^5}{\Gamma(s-\frac 34) } e^{-\pi y |n|+2\pi |r||n|^{1/2}},&\text{if }n\leq -1\text{ is not a negative square},\\
      |n/r|^{\frac 12}e^{-\pi y | n|}+\frac{(|r|^{-1/2}+|r|^3) |n|^5}{\Gamma(s-\frac 34) } e^{-\pi y |n|+2\pi |r||n|^{1/2}} &\text{if } n=-m^2 \text{ for some }m\in \Z_+. 
\end{array}\right.
\end{align*}
Then both $F_{\frac 32}(\tau;r,s)$ \eqref{eqFourierExpansionFAts-on-cusp-Infinity} and $\widetilde F_{\frac 32}(\tau;r,s)$ \eqref{eqFourierExpansion-tildeFAts-on-cusp-infty} are absolutely convergent for $s\in [\frac 34,1.001]$. Explicitly, we have
		\begin{align}
		\begin{split}
				F_{\frac 32}(\tau;r,\tfrac 34)&=\pi \ee(\tfrac 18)\sum_{n=1}^\infty e^{\pi i n \tau} \(\frac n{r^2}\)^{\frac 14} \sum_{2|c>0} \frac{K(r^2,n,c,\nu_{\Theta}^3)}{c} J_{\frac 12}\(\frac{2\pi |r| \sqrt{n}}{c}\)\\
			& + \frac{\sin(\pi r^2)}{\pi r\sinh(\pi r)}\(\sqrt{\frac 2y} + 2 \pi^{\frac 12} \sum_{m=1}^\infty m\Gamma(-\tfrac 12,2\pi m^2 y)e^{-\pi i m^2 \tau}\)
		\end{split}
		\end{align}
		and
		\begin{align}
			\begin{split}
				\widetilde F_{\frac 32}(\tau;r,\tfrac 34)&=\pi \ee(-\tfrac 38)\sum_{n=1}^\infty  e^{\pi i n \tau} \(\frac n{r^2}\)^{\frac 14} \sum_{2\nmid {\widetilde{c}}>0} \frac{\widetilde K(r^2,n,{\widetilde{c}},\nu_{\Theta}^3)}{{\widetilde{c}}} J_{\frac 12}\(\frac{2\pi |r| \sqrt n}{\widetilde{c}}\)\\
			&+\frac{\sin(\pi r^2)}{\pi r\sinh(\pi r)}\(\sqrt{\frac 2y} +  2 \pi^{\frac 12}\sum_{m=1}^\infty m\Gamma(-\tfrac 12,2\pi m^2 y)e^{-\pi i m^2 \tau}\). 
			\end{split}
		\end{align}
		They satisfy
		\begin{equation}
			F_{\frac 32}(\tau;r,\tfrac 34)+(-i\tau )^{-\frac 32} \widetilde F_{\frac 32}(-1/\tau;r,\tfrac 34) =e^{\pi i r^2 \tau}. 
		\end{equation}
\end{theorem}

\section{Kloosterman sums in weight 2}
\label{SectionKLweight2}
 In this section we assume $m,n,c\in \Z$ and $c>0$.  Recall the notation from \S\ref{subsectionKloosterman-Sums-Defs}.

For the standard Kloosterman sum $S(m,n,c)$, by \cite[(2)]{SarnakTsimerman09} we have the Weil bound
\begin{equation}
	|S(m,n,c)|\leq \sigma_0(c) \sqrt{\gcd(m,n,c)}\sqrt c. 
\end{equation}
By Proposition~\ref{propAppendix-KLsums-evenodd-relation}, for positive integers $2|c$ and $2\nmid \widetilde{c}$ we have
\begin{align}\label{eqWeilBound-Theta4}
    \begin{split}
        |S(m,n,c,\nu_\Theta^4)|&\leq \sigma_0(2c)\sqrt{\gcd(m,n,2c)}\sqrt{2c}, \\
        |S(m,n,\widetilde{c},\nu_\Theta^4)|&\leq 2\sqrt 2\sigma_0(4\widetilde{c})\sqrt{\gcd(m,n,2\widetilde{c})}\sqrt{\widetilde{c}}. 
    \end{split}
\end{align}

When $mn=0$, the Kloosterman sums are Ramanujan sums. For $n\neq 0$, 
\[S(n,0,c)=S(0,n,c)=\sum_{d\Mod c^*}\ee\(\frac{nd}c\)=\sum_{d|(n,c)}d\mu\(\frac{c}d\). \]
Then we have
\begin{equation}\label{eqSum-of-KL-sum-evenandoddc-Theta4-0n}
	\sum_{c\leq x} \frac{S(0,n,c)}c=\sum_{\substack{d|n\\d\leq x}}\sum_{t\leq \frac xd} \frac{\mu(t)}t\ll \min\(x,\sigma_0(n)\)\ll_\ep n^\ep. 
\end{equation}
The same bound holds for $\sum_{2|c} S(m,n,c,\nu_{\Theta}^4)/c$ and $\sum_{2\nmid \widetilde{c}} S(m,n,\widetilde{c},\nu_{\Theta}^4)/\widetilde{c}$. Moreover, we have $S(0,0,c)=\phi(c)$ and 
\begin{equation}
	\left. \begin{array}{rcl}
		\displaystyle \sum_{1\leq c\leq x}\frac{S(0,0,c)}c&=&\displaystyle\sum_{c\leq x}\frac{\phi(c)}c\\
	\displaystyle 	\sum_{2|c\leq x}\frac{S(0,0,c,\nu_{\Theta}^4)}c&=&\displaystyle \sum_{2|c\leq x}\frac{\phi(2c)}c\\
		\displaystyle \sum_{2\nmid \widetilde{c}\leq x}\frac{S(0,0,\widetilde{c},\nu_{\Theta}^4)}{\widetilde c}&=&\displaystyle \sum_{2\nmid \widetilde{c}\leq x}\frac{-\phi(\widetilde{c})}{\widetilde{c}}
	\end{array}
	\right\}=\left\{\begin{array}{l}
		\displaystyle \frac 6{\pi^2} x \vspace{10px}\\
		\displaystyle \frac 4{\pi^2} x\vspace{10px}\\
		\displaystyle -\frac 4{\pi^2} x
	\end{array}\right. 
	+O(\log x). 
\end{equation}

\section{Real-variable Kloosterman sums in weight 2}
\label{SectionRVKLweight2}
In this section we suppose $k=2$, $n\in \Z$ and $r\in \R_+$. Recall the real-variable Kloosterman sums $K(r,n,c,\nu_{\Theta}^{4})$ and $\widetilde{K}(r,n,c,\nu_{\Theta}^{4})$ from \eqref{eqReal-Val-Kloosterman-Sums} and \eqref{eqReal-Val-Kloosterman-Sums-tilde}. Our goal is to estimate
\[\sum_{2|c\leq x}\frac{K(r^2,n,c,\nu_{\Theta}^4)}c\quad \text{and}\quad \sum_{2\nmid \widetilde{c}\leq x}\frac{\widetilde{K}(r^2,n,\widetilde{c},\nu_{\Theta}^4)}{\widetilde{c}}\]
and conclude the properties of the expansion in Lemma~\ref{lemmaFourierExpansionAts}.

\subsection{Case \texorpdfstring{$n\neq 0$}{n not 0}}

We first prove the Weil bound for individual real-variable Kloosterman sums. 
\begin{lemma}\label{lemmaWeiltypebound-real-val-k2}
	For $r\in \R\setminus \Z$ and $n\in \Z\setminus \{0\}$, we have
	\[K(r,n,c,\nu_{\Theta}^4)\ll_\ep (n,c)^{\frac 12+\ep} c^{\frac 12+\ep} \quad\text{and}\quad  \widetilde{K}(r,n,\widetilde{c},\nu_{\Theta}^4)\ll_\ep (n,\widetilde{c})^{\frac 12+\ep} \widetilde{c}^{\frac 12+\ep}\]
	for $2|c>0$ and $2\nmid \widetilde{c}>0$.
\end{lemma}

\begin{proof}
	For $k\in \Z$, recall the Weil-type bound of Kloosterman sums \eqref{eqWeilBound-Theta4} and $\sigma_0(n)\ll_\ep n^\ep$. 
	Then by Lemma~\ref{lemmaExponential-interpolation}, 
	\begin{align*}
		K(r,n,c,\nu_{\Theta}^4)&=\frac ic \sum_{\substack{k\in \Z\\|r-k|<c}} \frac{\sin(\pi(r-k))}{\ee(\frac{r-k}{2c})-1} S(k,n,c,\nu_{\Theta}^4). 
	\end{align*}
	Since
	\begin{equation}\label{eq:crude bound et-1 abs}
		|\ee(t)-1|=\sqrt{2-2\cos(2\pi t)} \geq \pi |t| \quad \text{for }t\in[-\tfrac 12,\tfrac 12], 
	\end{equation}
	we get
	\begin{align*}
		|K(r,n,c,\nu_{\Theta}^4)|&\leq 2\sum_{\substack{k\in \Z\\ r-c<k<r+c}}\left| \frac{\sin(\pi(r-k))}{\pi(r-k)} \right| \cdot |S(k,n,c,\nu_{\Theta}^4)|\\
		&\leq 2\sum_{\substack{k\in \Z\\ r-c<k\leq \floor{r}-1}} \frac {|S(k,n,2c)|}{\pi(r-k)} + 2\big|S(\floor{r},n,2c)\big|\\
		&+ 2\big|S(\ceil{r},n,2c)\big| + 2\sum_{\substack{k\in \Z\\ \ceil{r}+1\leq k<r+c}} \frac {|S(k,n,2c)|}{\pi(k-r)}. 
	\end{align*}
	By the fact that $\sigma_s(n)\ll_{s,\ep} n^{s+\ep}$ for $s\geq 0$, we reorder the summation and get the desired bound: the sum over $r-c<k\leq \floor{r}-1$ is bounded in the following way: 
	\begin{align*}
		\sum_{\substack{k\in \Z\\ r-c<k\leq \floor{r}-1}} \frac {|S(k,n,2c)|}{r-k}&\ll_\ep c^{\frac 12+\ep} \sum_{\substack{k\in \Z\\ r-c<k\leq \floor{r}-1}} \frac{\sqrt{(k,n,c)}}{r-k}\\
		&\ll_\ep c^{\frac 12+\ep} \sum_{\delta|(n,c)} \delta^{\frac 12} \sum_{\substack{t\in \Z\\\frac{r-c}{\delta}<t\leq \frac{\floor{r}-1}{\delta}}}\frac{1}{r-\delta t}\\
		&\ll_\ep  (n,c)^{\frac 12+\ep} c^{\frac 12+\ep}; 
	\end{align*}
	the bound over $\ceil{r}+1\leq k<r+c$ follows the same way. 
    
	The proof for $\widetilde{K}(r,n,\widetilde{c},\nu_{\Theta}^4)$ follows from \eqref{eqrelation-KLsum-evenodd-k2} and the same process. 
\end{proof}

\begin{proposition}\label{propEstimate-Sum-of-real-variable-Kloosterman-sums-weightk2-byWeilbound}
	For $x\geq 1,r\in \R_+$ and $n\neq 0$, we have the following estimates: 
\[\sum_{2|c\leq x}\frac{|K(r^2,n,c,\nu_{\Theta}^4)|}c \ll_\ep |n|^\ep x^{\frac 12+\ep},\quad \sum_{2\nmid \widetilde{c}\leq x}\frac{|K(r^2,n,\widetilde{c},\nu_{\Theta}^4)|}{\widetilde{c}}\ll_\ep |n|^\ep x^{\frac 12+\ep}.  \]
\end{proposition}
\begin{proof}
The case $r^2\in \Z$ follows by directly applying the Weil bound of Kloosterman sums, and the case $r^2\notin \Z$ follows by directly applying Lemma~\ref{lemmaWeiltypebound-real-val-k2}. 
\end{proof}

\begin{proposition}\label{propEstimate-Sum-of-real-variable-Kloosterman-sums-weightk2-with-IJBessel}
	Let $s\in[1,1.001]$ and $r\in \R_+$. If $n>0$, both the following sums are absolute convergent and have the same estimate depending on $r^2 n\leq1 $ or $r^2n\geq 1$: 
	\[\text{both}\quad \left. \begin{array}{r}
		\displaystyle\sum_{2|c>0} \frac{K(r^2,n,c,\nu_{\Theta}^4)}{c} J_{2s-1}\(\frac{2\pi |r|\sqrt n}c\)\\
		\displaystyle\sum_{2\nmid \widetilde{c}>0} \frac{\widetilde{K}(r^2,n,\widetilde{c},\nu_{\Theta}^4)}{\widetilde{c}} J_{2s-1}\(\frac{2\pi |r|\sqrt n}{\widetilde{c}}\)\\
	\end{array}\right\} \ll_\ep \left\{\begin{array}{ll}
	    |r|n^{\frac 12+\ep}, & \text{if }r^2\leq \frac 1n, \vspace{5px}\\
	     (1+|r|^{\frac 12+\ep})n^{\frac 14+\ep},& \text{if }r^2\geq \frac 1n.
	\end{array}\right.  \]
	If $n<0$, the following sums are absolute convergent and have estimates
	\[\left. \begin{array}{r}
		\displaystyle\sum_{2|c>2\pi |r^2 n|^{1/2}} \frac{K(r^2,n,c,\nu_{\Theta}^4)}{c} I_{2s-1}\(\frac{2\pi |r^2 n|^{\frac 12}}c\)\\
		\displaystyle\sum_{2\nmid \widetilde{c}>2\pi |r^2 n|^{1/2}} \frac{\widetilde{K}(r^2,n,\widetilde{c},\nu_{\Theta}^4)}{\widetilde{c}} I_{2s-1}\(\frac{2\pi |r^2 n|^{\frac 12}}{\widetilde{c}}\)\\
	\end{array}\right\} \ll_\ep (1+|r|^{1+\ep})|n|^{\frac 12+\ep}. 
	\]
    \[\left. \begin{array}{r}
		\displaystyle\sum_{2|c\leq 2\pi |r^2 n|^{1/2}} \frac{K(r^2,n,c,\nu_{\Theta}^4)}{c} I_{2s-1}\(\frac{2\pi |r^2 n|^{\frac 12}}c\)\\
		\displaystyle\sum_{2\nmid \widetilde{c}\leq 2\pi |r^2 n|^{1/2}} \frac{\widetilde{K}(r^2,n,\widetilde{c},\nu_{\Theta}^4)}{\widetilde{c}} I_{2s-1}\(\frac{2\pi |r^2 n|^{\frac 12}}{\widetilde{c}}\)\\
	\end{array}\right\} \ll_\ep (1+|r|^{2+\ep})|n|^{1+\ep}e^{2\pi |r||n|^{1/2}}. 
	\]
\end{proposition}

\begin{proof}
The proof requires the following bounds on $J$- and $I$-Bessel functions. By \cite[(10.14.1)]{dlmf}, $|J_\nu(x)|\leq 1$ for $\nu\geq 0$; by \cite[(10.7.3), (10.7.8)]{dlmf},
	\begin{equation}\label{eqBound-of-J--1/2tonu}
		J_\nu(x)\ll_\nu \min\(x^{-\frac 12},x^\nu\),\quad \text{hence } J_{2s-1}(x)\ll \min\(x^{-\frac 12},x\) \text{ for }s\in[1,1.001].  
	\end{equation}
	By Lemma~\ref{lemmaWeiltypebound-real-val-k2} and \eqref{eqBound-of-J--1/2tonu}, we have the following two cases: 

    (1) when $r^2\leq \frac 1n$, the following sum converges absolutely and has estimate
	\begin{align}\label{eqEstimate-SumofKLsum-withJ2-prog1-r2smallerthan1/n}
		\begin{split}
		\sum_{2|c>0} \frac{K(r^2,n,c,\nu_{\Theta}^4)}{c} J_{2s-1}\(\frac{2\pi |r|\sqrt n}c\)&\ll_\ep |r|n^{\frac 12} \sum_{2|c>0} (n,c)^{\frac 12+\ep} c^{-\frac 32+\ep}\\
		&\ll_{\ep} |r| n^{\frac 12+\ep};
		\end{split}
	\end{align}

    (2) when $r^2\geq \frac 1n$, we break the sum at $|r^2 n|^{1/2}$ and still get the absolute convergence
    \begin{align}\label{eqEstimate-SumofKLsum-withJ2-prog1}
        \begin{split}
		&\sum_{2|c>0}
        \frac{K(r^2,n,c,\nu_{\Theta}^4)}{c} J_{2s-1}\(\frac{2\pi |r|\sqrt n}c\)\\
        &\ll_\ep |r^2 n|^{-\frac 14} \sum_{2|c,\, c\leq |r^2n|^{1/2}} (n,c)^{\frac 12+\ep}c^\ep+|r^2 n|^{\frac 12} \sum_{2|c,\,c\geq |r^2n|^{1/2}} (n,c)^{\frac 12+\ep} c^{-\frac 32+\ep}\\
        &\ll_\ep |r^2 n|^{\frac 14+\ep}. 
		\end{split}
    \end{align}
	The case $2\nmid \widetilde{c}$ follows similarly.

	For $n<0$ we need following bounds for the $I$-Bessel function: by \cite[(10.30.1)]{dlmf}, for fixed $\nu\geq 0$, $I_\nu(x)\ll_\nu  x^\nu$ for $x\leq 1$, hence $I_{2s-1}(x)\ll x$ for $x\leq 1$; by \cite[(10.30.4)]{dlmf}, $I_\nu(x)\ll e^x$ for $x\geq 1$.
    
	Those series are absolute convergent following from Lemma~\ref{lemmaWeiltypebound-real-val-k2}, Proposition~\ref{propEstimate-Sum-of-real-variable-Kloosterman-sums-weightk2-byWeilbound} and the above inequalities. The proof for case $2\nmid \widetilde{c}$ follows by the same argument. We have finished the proof of the proposition. 
\end{proof}

\subsection{Case \texorpdfstring{$n=0$}{n=0}}
\label{subsectionKLRVk2n=0}
For $2|c>0$, we define and compute
\begin{align}\label{eqComputation-in-weight-2-part-phic}
	\begin{split}
		\phi_c(r)&\defeq \sum_{\substack{1\leq a<c\\(a,c)=1}} e^{\pi i r a/c}=\sum_{1\leq a\leq c}\sum_{\delta|(a,c)} \mu(\delta) e^{\pi i r a/c}=\sum_{\delta|c}\mu(\delta)\sum_{\substack{a=m\delta\\1\leq m\leq \frac c\delta}}e^{\pi i r m \delta /c}\\
		&=\sum_{\delta|c}\mu(\delta)
		\frac{e^{\pi i r}-1}{1-e^{-\pi i r \delta/c}}. 
	\end{split}
\end{align}
Recall the real-variable Kloosterman sum defined in \eqref{eqReal-Val-Kloosterman-Sums}. We have 
\[2\re \varphi_c(r^2)=K(r^2,0,c,\nu_\Theta^4).\]
Let $Z(s)$ denote the zeta function
\begin{equation}
	Z(s)\defeq \sum_{2|c>0}\frac{\phi_c(r^2)}{c^{2s}}, \quad \re s>1. 
\end{equation}
We want to understand the behavior of $Z(s)$ when $s\rightarrow 1^+$. By \eqref{eqComputation-in-weight-2-part-phic}, for $\re s>1$ we have
\begin{align*}
	Z(s)&=(e^{\pi i r^2}-1)\sum_{2|c>0} \sum_{\delta|c}\frac{\mu(\delta)}{1-e^{-\pi i r^2 \delta/c}} c^{-2s}\\
	(\delta\text{ odd, }c=2m\delta)&=(e^{\pi i r^2}-1)\sum_{2\nmid \delta >0} \frac{\mu(\delta)}{(2\delta)^{2s}} \sum_{m=1}^\infty \frac{m^{-2s}}{(1-\ee\(\frac{-r^2}{4m}\))}\\
	(\delta\text{ even, }c=m\delta)&+(e^{\pi i r^2}-1)\sum_{2|\delta>0} \frac{\mu(\delta)}{\delta^{2s}} \sum_{m=1}^\infty \frac{m^{-2s}}{(1-\ee\(\frac{-r^2}{2m}\))}\\
	&=:\Sigma_1+\Sigma_2. 
\end{align*}
To apply Lemma~\ref{lemmaBernoulli-numbers-expansion} we require $m\geq r^2$, which will only affect finitely many of $m$. Then for $\Sigma_1$ we have
\begin{align*}
	\sum_{m=1}^\infty \frac{m^{-2s}}{(1-\ee\(\frac{-r^2}{4m}\))}&=O_r(1)+\sum_{m>r^2}\frac{m^{-2s}}{(1-\ee\(\frac{-r^2}{4m}\))}\\
	&=O_r(1)+\sum_{m>r^2}\frac{m^{-2s+1}}{\pi i r^2/2}-\sum_{\ell=1}^\infty B_\ell \frac{(-2\pi i r^2)^{\ell-1} }{\ell! 4^{\ell-1}}\sum_{m>r^2}\frac 1{m^{2s+\ell-1}}\\
	&=\sum_{m>r^2}\frac{m^{-2s+1}}{\pi i r^2/2}+O_r(1),
\end{align*} 
where the last step is the result of 
\[\sum_{m>r^2} \frac 1{m^{2s+\ell-1}}=O((r^2)^{-\ell}) \quad \text{for }\re s\geq 1,\quad \text{when }\ell\geq 1. \]
We have similar estimates for $\Sigma_2$. Then the main contribution to $Z(s)$ is given by
\begin{align*}
	Z(s)	&=\frac{e^{\pi i r^2}-1}{\pi i r^2/2} \sum_{2\nmid \delta>0} \frac{\mu(\delta)}{(2\delta)^{2s}} \zeta(2s-1)\\
	&+\frac{e^{\pi i r^2}-1}{\pi i r^2} \sum_{2|\delta>0} \frac{\mu(\delta)}{\delta^{2s}} \zeta(2s-1)+O_r(1). 
\end{align*}
Since 
\[\lim_{s\rightarrow 1^+} \frac{\zeta(2s-1)}{\Gamma(s-1)}=\frac 12\quad \text{and}\quad \sum_{2\nmid \delta>0}\frac{\mu(\delta)}{\delta^{2}}=\frac{8}{\pi^{2}},\]
 we have
\begin{equation}    \label{eqweight-2-evenc-n0residue}
		\lim_{s\rightarrow 1^+} \frac{1}{\Gamma(s-1)}\sum_{2|c>0}\frac{K(r^2,0,c,\nu_{\Theta}^4)}{c^{2s}}=\lim_{s\rightarrow 1^+} \frac{2\re Z(s)}{\Gamma(s-1)}
		=\frac{2\sin(\pi r^2)}{\pi^3r^2}. 
\end{equation}

The similar process applies to $\widetilde K(r^2,0,d,\nu_{\Theta}^4)$. We have $\widetilde K(r^2,0,1,\nu_{\Theta}^{4})=-1$ and
\begin{equation*}
	\widetilde{K}(r^2,0,d,\nu_{\Theta}^{4})=-\sum_{\substack{-d< b<d\\ 2|b,\ (b,d)=1}}e^{\pi i r^2 b/d}=(e^{\pi i r^2}-e^{-\pi i r^2})\sum_{\delta|2d}\frac{\mu(\delta)}{\ee(-\frac{r\delta}{2d})-1}
\end{equation*}
for odd $d>1$. Finally, we find
\begin{equation}\label{eqweight-2-oddd-n0residue}
	\lim_{s\rightarrow 1^+}\frac{1}{\Gamma(s-1)}\sum_{2\nmid d>0}\frac{\widetilde K(r^2,0,d,\nu_{\Theta}^4)}{d^{2s}}
	=\frac {2\sin(\pi r^2)}{\pi^3 r^2}. 
\end{equation}

\section{Proof of Theorem~\ref{theoremFourierExpansionForWeight2Ats=1} and Theorem~\ref{theoremMainDim4}}
\label{sectionProof-of-theorem-FourierExpansion-weight2-s1}

We are now ready to prove Theorem~\ref{theoremFourierExpansionForWeight2Ats=1}. 

\begin{proof}[Proof of Theorem~\ref{theoremFourierExpansionForWeight2Ats=1}]
We first cite the following bounds of Whittaker function \cite[(13.14.21)]{dlmf}: for $n\neq 0$,
\begin{equation}\label{eq:WhittakerboundE-piny}
    \frac{|W_{\pm\frac k2,s-\frac 12}(2\pi |n| y)|}{|2\pi r^2 y|^{\frac k2}} \left|\frac{r^2 }n\right|^{\frac 12}\ll e^{-\pi |n| y} \left|\frac{n} {r^2}\right|^{\frac {k-1}2}
\end{equation}
	Combining Proposition~\ref{propEstimate-Sum-of-real-variable-Kloosterman-sums-weightk2-with-IJBessel} and \S\ref{subsectionKLRVk2n=0}, for $r\in \R_+$ and $n\in \Z$, we conclude that for $s\in[1,1.001]$, $B_{4,n}(r;y,s)$ and $\widetilde{B}_{4,n}(r;y,s)$ are absolute convergent and have the claimed bound. 
    
    By dominated convergence theorem, the factor $e^{-\pi y|n|}$ ensures that the Fourier expansions of $F_2(\tau,r;s)$ and $\widetilde{F}_2(\tau,r;s)$ for $s\in [1,1.001]$ in Lemma~\ref{lemmaFourierExpansionAts} are both absolutely convergent. By analytic continuation,  $F_2(\tau,r;1)$ and $\widetilde{F}_2(\tau,r;1)$ has the corresponding Fourier expansion given by $B_{4,n}(r;1,y)$ and $\widetilde{B}_{4,n}(r;1,y)$. The explicit formula in \eqref{eqFourierExpansionForWeight2Ats=1-evenc} and \eqref{eqFourierExpansionForWeight2Ats=1-oddd} follows from \eqref{eqWWhittaker-special-form}, \eqref{eqweight-2-evenc-n0residue} and \eqref{eqweight-2-oddd-n0residue}. Note that when $n\leq -1$, $B_{4,n}(r;y,s)$ and $\widetilde{B}_{4,n}(r;y,s)$ are absolute convergent and converge to $0$ when $s\rightarrow 1^+$ because of the Gamma factor. 
    
    The functional equation \eqref{eqFourierExpansionForWeight2Ats=1-FunctionalEquation} follows from Proposition~\ref{propFunctionalequation-ats}, analytic continuation, and \eqref{eqFunctionalequation-atsk2}. 
\end{proof}

To remove the $n=0$ term $\frac{\sin(\pi r^2)}{y\pi^2 r^2}$ in Theorem~\ref{theoremFourierExpansionForWeight2Ats=1} such that the resulting functions fulfill the conditions of Theorem~\ref{theoremBonRadSeip-2022}, we need to use the weight $2$ Eisenstein series. 
\begin{definition}[{\cite[(1.5)]{DiamondShurman}}]
	For $\tau=x+iy\in \HH$ and $x,y\in \R$, the weight $2$ Eisenstein series $E_2(\tau)$ is defined by
	\[E_2(\tau)=1-24\sum_{n=1}^\infty \sigma_1(n) e^{2\pi i n \tau},\quad \text{where }\sigma_1(n)=\sum_{d|n} d. \]
\end{definition}
The weight 2 Eisenstein series is not modular, but it satisfies the following identity
    \begin{equation} \label{eq:E2transform}
    E_2(\tau)-\tau^{-2} E_2(-1/\tau)=-\frac{12}{2\pi i \tau}. 
    \end{equation}

Let $\sigma_1(x)=0$ for $x\notin \Z$ and define 
\begin{equation}
    \mathcal{E}_2(\tau)\defeq -\frac 3{\pi y}+E_2\(\frac {\tau+1}2\)-E_2(\tau)=-\frac 3{\pi y}+24\sum_{n=1}^\infty \(\sigma_1\(\frac n2\)-(-1)^n\sigma_1(n)\)e^{\pi i n\tau}. 
\end{equation} 
It is straightforward by \eqref{eq:E2transform} that
\[\mathcal{E}_2(\tau+2)=\mathcal{E}_2(\tau)\quad \text{and}\quad \mathcal{E}_2(\tau)+(\tau/i)^{-2}\mathcal{E}_2(-1/\tau)=0. \]
Therefore, by defining the series 
\begin{equation}\label{eqDef-G2-tildeG2_1}
\begin{split}
    G_2(\tau;r)&\defeq F_2(\tau;r,1)+\frac{\sin(\pi r^2)}{6\pi r^2} \mathcal{E}_2(\tau)\\
    &=\sum_{n=1}^\infty e^{\pi i n \tau} \cdot \pi \( \frac n{r^2}\)^{\frac 12} \sum_{2|c>0}\frac{K(r^2,n,c,\nu_{\Theta}^{4})}c J_{1}\(\frac{2\pi |r|\sqrt n}{c}\)\\
    &+\sum_{n=1}^\infty e^{\pi i n \tau } \cdot \frac{8\sin(\pi r^2)}{\pi r^2}\(\sigma_1\(\frac n2\)-(-1)^n \sigma_1(n)\),
\end{split}    
\end{equation}
and
\begin{equation}\label{eqDef-G2-tildeG2_2}
	\begin{split}
    \widetilde{G}_2(\tau;r)&\defeq \widetilde {F}_2(\tau;r,1)+\frac{\sin(\pi r^2)}{6\pi r^2} \(E_2^*(\tfrac{\tau+1}2)-\Theta(\tau)^4\)\\
    &=\sum_{n=1}^\infty e^{\pi i n \tau} \cdot (-\pi) \( \frac n{r^2}\)^{\frac 12} \sum_{2\nmid {\widetilde{c}}>0}\frac{\widetilde{K}(r^2,n,{\widetilde{c}},\nu_{\Theta}^{4})}{\widetilde c} J_{1}\(\frac{2\pi |r|\sqrt n}{{\widetilde{c}}}\)\\
    &+\sum_{n=1}^\infty e^{\pi i n \tau } \cdot  \frac{8\sin(\pi r^2)}{\pi r^2}\(\sigma_1\(\frac n2\)-(-1)^n \sigma_1(n)\),
	\end{split}
\end{equation}
we get that they satisfy
\begin{equation}
	G_2(\tau;r)+(\tau/i)^{-2}\widetilde{G}_2(\tau;r)=e^{\pi i r^2 \tau}. 
\end{equation}

\begin{proof}[Proof of Theorem~\ref{theoremMainDim4}]\label{ProofofTheoremMainDim4}
	The coefficients $b_{4,n}(r)$ and $\widetilde{b}_{4,n}(r)$ can be read from \eqref{eqDef-G2-tildeG2_1} and~\eqref{eqDef-G2-tildeG2_2}. The bound for them in $n$ is clear by Proposition~\ref{propEstimate-Sum-of-real-variable-Kloosterman-sums-weightk2-with-IJBessel}. It remains to determine the values 
    \begin{equation}\label{eq:values_b4ntildeb4n}
          B_{4,n}(0)=\lim_{r\rightarrow 0^+} B_{4,n}(r)\quad\text{and}\quad \widetilde{B}_{4,n}(0)=\lim_{r\rightarrow 0^+} \widetilde{B}_{4,n}(r)\quad \text{for }n\geq 1.
    \end{equation}
    By \cite[(10.2.2)]{dlmf}:
    \begin{equation}\label{eqExpansionofJ1}
        \frac{\pi\sqrt{n}}{|r|}J_1\(\frac{2\pi |r|\sqrt n}{c}\)=\frac{\pi^2 n}{c}\(1+\sum_{j=1}^\infty \frac{(-1)^j \(\pi |r|\sqrt{n}/c\)^{2j}}{j!(j+1)!}\),
    \end{equation}
    we claim that the limits in \eqref{eq:values_b4ntildeb4n} are given by
    \begin{align}\label{eq:b4nANDtildeb4nAS_KLsumr=0}
    \begin{split}
        B_{4,n}(0)=\pi^2 n\sum_{2|c>0}\frac{S(0,n,c,\nu_\Theta^4)}{c^2}\quad \text{and}\quad \widetilde{B}_{4,n}(0)=\pi^2 n\sum_{2\nmid \widetilde{c}>0}\frac{S(0,n,\widetilde{c},\nu_\Theta^4)}{\widetilde{c}^2}. 
    \end{split}
    \end{align}
    This can be proved by absolute convergence with Proposition~\ref{propEstimate-Sum-of-real-variable-Kloosterman-sums-weightk2-byWeilbound}. 
    
    To compute $B_{4,n}(0)$ and $\widetilde{B}_{4,n}(0)$ explicitly, we define the Ramanujan sum 
    \[R_c(n)\defeq \sum_{d\Mod c} \ee\(\frac{nd}c\). \]
    It is well-known that $R_c(n)$ is a multiplicative function of $c$, 
    \begin{equation}
        \sum_{c=1}^\infty \frac{R_c(n)}{c^s}=\frac{\sigma_{s-1}(|n|)}{|n|^{s-1}\zeta(s)},\quad \text{and}\quad R_{p^j}(n)=\left\{\begin{array}{ll}
            0, & j\geq \nu_p(n)+2, \\
            -p^{j-1}, & j=\nu_p(n)+1, \\
            \varphi(p^j), & j\leq \nu_p(n). 
        \end{array}\right.
    \end{equation}
    By defining $h=\nu_2(n)$ and computing the $2$-factor of the $L$-function, we conclude that
    \begin{align*}
        B_{4,n}(0)&=\frac{\pi^2 n}{4}\sum_{c=1}^\infty \frac{R_{4c}(n)}{c^2}=8\sigma_1(n)\cdot \left\{\begin{array}{ll}
            \displaystyle\frac{2^h-3}{2^{h+1}-1}, & 2|n,  \vspace{2px}\\
            0,& 2\nmid n, 
        \end{array}\right. \\
        \widetilde{B}_{4,n}(0)&=\pi^2 n\sum_{2\nmid \widetilde{c}>0} \frac{R_{\widetilde{c}}(n)}{\widetilde{c}^{2}}=8\sigma_1(n)\cdot \frac{2^h}{2^{h+1}-1}. 
    \end{align*}
    This concludes the proof. 
\end{proof}

\section{Kloosterman sums in weight \texorpdfstring{$3/2$}{3/2}}
\label{SectionKLweight3/2}
Recall the notation from \S\ref{sectionPrelimandNotation}. In this section we prove the properties of $S(m,n,c,\nu_\Theta^3)$ for $2| c>0$ and of $S(m,n,\widetilde{c},\nu_\Theta^3)$ for $2\nmid \widetilde{c}>0$, where $m,n\in \Z$.   
\subsection{Case \texorpdfstring{$mn\neq 0$}{mn not 0}}

For half-integer weight $k\in \Z+\frac 12$, for Kloosterman sums with multiplier $\nu_\theta$ we have the Weil bound
\begin{equation}\label{eqWeilBound-theta3}
	|S(m,n,c,\nu_\theta^{2k})|\leq \sigma_0(c)\sqrt{\gcd(m,n,c)}\sqrt c
\end{equation}
by \cite[(2.15)]{BlomerSumofHeckeEvaluesOverQP}. According to the relation \eqref{eqKloosterman-Sums-Relation-Even-Theta-to-theta}, we get the Weil bound for $\nu_{\Theta}$ when $c\in 2\Z_+$:  
\begin{equation}\label{eqWeilBound-Theta3}
	|S(m,n,c,\nu_{\Theta}^{2k})|=|S(m,n,2c,\nu_\theta^{2k})|\leq 2\sigma_0(2c)\sqrt{\gcd(m,n,c)}\sqrt {c}. 
\end{equation}

\begin{proposition}\label{propSumofKloostermansums-weight32-mnnot0}
	For $mn\neq 0$, we have the following estimate: 
	\begin{align}\label{eqSumofKloostermansums-weight32-mnnot0}
		\begin{split}
			\sum_{2|c\leq x} \frac{S(m,n,c,\nu_{\Theta}^3)}{c}&-\left\{
			\begin{array}{ll}
				\ee(-\frac 18) \frac{16}{\pi^2} x^{\frac 12}& \text{if both }m\text{ and }n\text{ are negative squares}\\
				0& \text{other }m,n\text{ with } mn\neq 0
			\end{array}
			\right. \\
			&=O_\ep \( \big(x^{\frac 16}+A_u(m,n)\big)|xmn|^\ep\). 
		\end{split}
	\end{align}
	Here $A_u(m,n)$ is defined as in \cite[Theorem~1.4]{QihangFirstAsympt} and \cite[Theorem~1.3]{QihangSecondAsympt}: writing $m=t_mu_m^2w_m^2$ and $n=t_nu_n^2w_n^2$, where $t_m,t_n$ are square-free, $u_m,u_n$ are even, $w_m,w_n$ are odd, and $\delta= \frac 1{147}$, we define
	\begin{equation}\label{eqdefAumn}
		A_u(m,n)\defeq \(|m|^{\frac 12-8\delta}+u_m\)^{\frac 18}\(|n|^{\frac 12-8\delta }+u_n\)^{\frac 18}|mn|^{\frac 3{16}}. 
	\end{equation} 
	We have $A_u(m,n)\ll |mn|^{\frac 14}$, and if $2^{\nu_2(m)}$ and $2^{\nu_2(n)}$ are bounded, $A_u(m,n)\ll |mn|^{\frac 14-\delta}$. 
\end{proposition}

\begin{remark}
	In other words, $u_n=2^{\floor{\nu_2(n)/2}}$. 
\end{remark}

\begin{proof}
	By \eqref{eqKloosterman-Sums-Relation-Even-Theta-to-theta}, we have
	\begin{align}\label{eqEstimate-Theta-as-estimate-theta-2x}
		\begin{split}
			\sum_{2|c\leq x} \frac{S(m,n,c,\nu_{\Theta}^3)}{c}&=\sum_{4|2c\leq 2x} \frac{S(m,n,2c,\nu_{\theta}^3)}{c}\\
		\quad &=2\sum_{4|c\leq 2x} \frac{S(m,n,c,\nu_{\theta}^3)}{c}=2\overline{\sum_{4|c\leq 2x} \frac{S(-m,-n,c,\nu_{\theta})}{c}}.
		\end{split}
	\end{align} 
	So it suffices to prove the asymptotics for the sum of $S(m,n,c,\nu_{\theta}^3)$ and $S(m,n,c,\nu_{\theta})$. Note that $\nu_{\theta}^3=\nu_{\theta}^{-1}$ is a weight $-\frac 12$ multiplier system on $\Gamma_0(4)$. 
	
	Recall \S\ref{subsectionMaass-forms} for the theory of Maass forms. 
	We apply \cite[Theorem~1.4]{QihangFirstAsympt} (respectively \cite[Theorem~1.3]{QihangSecondAsympt}) to estimate when $m>0$ and $n>0$ (respectively $m>0$ and $n<0$). For $k=\pm \frac 12$, the multiplier $\nu_{\theta}$ is on $\Gamma_0(4)$ and satisfies \cite[Definition~1.1]{QihangFirstAsympt,QihangSecondAsympt} by \eqref{eqWeilBound-theta3} and taking $D=4$. Therefore, for $k=\pm \frac 12$, 
	\[\sum_{4|c\leq x} \frac{S(m,n,c,\nu_{\theta}^{2k})}{c}-\sum_{r_j\in i(0,\frac 14]} \tau_j^{\pm}(m,n)\frac{x^{2s_j-1}}{2s_j-1}\ll _\ep  \(A_u(m,n)+x^{\frac 16}\)|xmn|^\ep\]
	where $s_j=\im r_j+\frac 12$ runs over eigenvalues $\lambda_j=\frac 14+r_j^2<\frac 14$ of weight $k$ hyperbolic Laplacian on $\Gamma_0(4)$. Here we write $\tau_j^{\pm}$ to mention that $\tau_j(m,n)$ depends on the weight $k=\pm \frac 12$. By Lemma~\ref{lemmaSelberg-eigenvalue}, the Shimura correspondence of eigenvalues of $\Delta_0$ and $\Delta_{\frac 12}$ (\cite[Theorem~5.6]{QihangFirstAsympt}, \cite[p. 304]{sarnakAdditive}), and \eqref{eqCorrespondence-Maass-and-holomorphic}, we only need to consider $\tau_0^{\pm}(m,n)$, where $r_0=\frac i4$ and $s_0=\frac 34$. 
	
	By \cite{gs} (corrected in \cite[(6.3)]{AAAlgbraic16}), or by combining \cite[Theorem~1.4]{QihangFirstAsympt} and \cite[Theorem~1.3]{QihangSecondAsympt}, when $m>0$ and $k=\pm \frac 12$, the main contribution to the estimate above is
	\[2\tau_0^{\pm}(m,n) x^{\frac 12}= 4\sqrt 2\,\ee(\tfrac k4)\overline{\rho_0(m)}\rho_0(n)|mn|^{\frac 14} \frac{\Gamma(\frac 34+\sgn(n)\frac k2)}{\Gamma(\frac 34-\frac k2)} x^{\frac 12}, 
\]
where $\rho_0(n)$ for $n\in \Z\setminus \{0\}$ is given by \eqref{eqMaass-form-coeffs-thetapm}. We get
\begin{equation}\label{eqTau0+}
2\tau_0^{+}(m,n)=\left\{\begin{array}{ll}
	\ee(\tfrac 18) \frac{4\sqrt 2}{\pi^2},&\ \text{if }m,n\text{ are both positive squares;,}\\
	0,&\ \text{other }m>0,\ n\neq 0,
\end{array}
\right.
\end{equation}
and
\begin{equation}\label{eqTau0-}
	2\tau_0^{-}(m,n)=0\quad \text{if }m>0,\ n\neq 0.  
\end{equation}
Thus, the estimate of \eqref{eqSumofKloostermansums-weight32-mnnot0} is given by \eqref{eqEstimate-Theta-as-estimate-theta-2x} and the following three cases: (1) when $m>0$, \eqref{eqTau0-} suggests the coefficient for $x^{\frac 12}$ as $0$; (2) when $m<0$, \eqref{eqTau0+} suggests main term consisting of $x^{\frac 12}$ as 
\[2\ee(-\tfrac 18)\frac{4\sqrt 2}{\pi^{2}} (2x)^{\frac 12}=\ee(-\tfrac 18)\frac{16}{\pi^{2}} x^{\frac 12}\quad \text{when }m\text{ and }n\text{ are negative squares};\]
(3) in the other cases when $m<0$, \eqref{eqTau0+} and \eqref{eqTau0-} combined show that the coefficient of $x^{\frac 12}$ is $0$. We have finished the proof.
\end{proof}

A similar method applies to the sum over $4|c$ by Lemma~\ref{lemmaMaass-form-coeffs-thetapm}, hence we state the following proposition without proof. 
\begin{proposition}\label{propSumofKloostermansums-weight32-mnnot0-4|c}
	For $mn\neq 0$, we have
	\begin{align}\label{eqSumofKloostermansums-weight32-mnnot0-4|c}
		\begin{split}
			\sum_{4|c\leq x} \frac{S(m,n,c,\nu_{\Theta}^3)}{c}&-\left\{
			\begin{array}{ll}
				\ee(-\frac 18) \frac{8}{\pi^2} x^{\frac 12}& \text{both }m\text{ and }n\text{ are negative squares}\\
				0& \text{other } mn\neq 0
			\end{array}
			\right. \\
			&\ll_\ep \( x^{\frac 16}+A_u(m,n)\)	|xmn|^\ep,  
		\end{split}
	\end{align}
	where $A_u(m,n)$ is defined at \eqref{eqdefAumn}. 
\end{proposition}

We also have the following proposition for the odd case. 
\begin{proposition}\label{propSumofKloostermansums-weight32-mnnot0-oddd}
	For $mn\neq 0$, we have the following estimate: 
	\begin{align}\label{eqSumofKloostermansums-weight32-mnnot0-oddd}
		\begin{split}
			\sum_{2\nmid \widetilde{c}\leq x} \frac{S(m,n,\widetilde{c},\nu_{\Theta}^3)}{\widetilde{c}}&+\left\{
			\begin{array}{ll}
				\ee(-\frac 18) \frac{16}{\pi^2} x^{\frac 12}& \text{both }m\text{ and }n\text{ are negative squares}\\
				0& \text{other } mn\neq 0
			\end{array}
			\right. \\
			&\ll_\ep \( x^{\frac 16}+A_u(m,n)\)	|xmn|^\ep. 
			\end{split}
	\end{align}
	Where $A_u(m,n)$ is defined as in \eqref{eqdefAumn}. 
\end{proposition}

\begin{proof}
	Define $\xi\in \{\pm 1\}$ by $\xi=-1$ if $m\equiv 0,3\Mod 4$ and $\xi=+1$ if $m\equiv  1,2\Mod 4$. By \eqref{eqrelation-KLsum-evenodd-k3/2} we have
	\begin{align}
		\begin{split}
			\sum_{2\nmid \widetilde{c}\leq x}\frac{S(m,n,\widetilde{c},\nu_{\Theta}^3)}{\widetilde{c}}
			&=\frac{\xi}{ \sqrt 2}\sum_{2\| c\leq 2x}\frac{S(m,4n,c,\nu_{\Theta}^3)}{c/2}\\
			&=\xi\sqrt 2\(\sum_{2|c\leq 2x}-\sum_{4|c\leq 2x}\)\frac{S(m,4n,c,\nu_{\Theta}^3)}{c/2}. 
		\end{split}
	\end{align}
	The main term involving $x^{\frac 12}$ is then given by Proposition~\ref{propSumofKloostermansums-weight32-mnnot0} and Proposition~\ref{propSumofKloostermansums-weight32-mnnot0-4|c}. Note that when $m$ is a negative square, we always have $m\equiv 0,3\Mod 4$ and hence $\xi=-1$. 
\end{proof}

\subsection{Case \texorpdfstring{$mn=0$}{mn=0}}\label{subsectionWeight3/2mn=0alphaApns}
This case does not satisfy the conditions of \cite[Theorem~1.4]{QihangFirstAsympt} and \cite[Theorem~1.3]{QihangSecondAsympt}, so we deal with it separately in this subsection. 

Recall our notations in \S\ref{sectionPrelimandNotation}: $\ep_d$ at \eqref{ThetaMultiplier}; for any prime $p$ and $n\in \Z\setminus \{0\}$, let $p^\nu\|n$ denote that $p^\nu|n$ while $p^{\nu+1}\nmid n$, and we write $\nu_p(n)=\nu$ for the integer $\nu$. 
First we need the following property: for positive odd integers $a,b$,
\begin{equation}\label{eqQuadratic-reciprocity-involving-ep}
	\(\frac ab\)=\(\frac ba\) \ep_a\ep_b\ep_{ab}^{-1}=\(\frac ba\) \ep_a^{-1}\ep_b^{-1}\ep_{ab}.
\end{equation}

The following construction is from \cite[Chapter~1]{WangPeiMFHalfBook} and can be traced back to \cite{Maass1937theta,Bateman1951ThreeSquares}. We record the explicit construction and results here for our application. 
For $n\in \Z$, $k\in \Z+\frac 12$ and $p$ as a prime, we define
\begin{align}
	\alpha_{2k}(2^\nu,n)&\defeq \sum_{a=1}^{2^\nu} \(\frac{2^\nu}a\)\ep_a^{2k} \ee\(\frac{na}{2^\nu}\),\quad \nu\geq 2;\\
    \alpha_{2k}(p^\nu,n)&\defeq \ep_{p^\nu}^{-2k}\sum_{a=1}^{p^\nu} \(\frac{a}{p^\nu}\)\ee\(\frac{na}{p^\nu}\),\quad \nu \geq 0,\ p>2.
\end{align}
Then for $2|c>0$, the condition $ad\equiv 1\Mod {2c}$ implies $(\frac{2c}d)=(\frac{2c}a)$ and $\ep_d=\ep_a$, hence we have
\begin{equation}\label{eqKLsum-n0weight3-into-prods}
S(0,n,c,\nu_\Theta^{2k})=S(n,0,c,\nu_{\Theta}^{2k})=\sum_{a\Mod{2c}^*} \ep_a^{2k}\(\frac{2c}a\)\ee\(\frac{na}{2c}\)=\prod_{p^\nu\|2c} \alpha_{2k}(p^\nu,n).  
\end{equation}
For $2\nmid \widetilde{c}>0$, note that $2|a$, $2|d$ and the condition $ad\equiv 1\Mod {\widetilde{c}}$ implies $(\frac{2a}{\widetilde{c}})=(\frac{2d}{\widetilde{c}})$. With the help of \eqref{eqQuadratic-reciprocity-involving-ep} we have
\begin{align}\label{eqKLsum-n0weight3-into-prods-oddd}
	\begin{split}
		S(0,n,\widetilde{c},\nu_{\Theta}^{2k})&=S(n,0,\widetilde{c},\nu_{\Theta}^{2k})=\ee(\tfrac k4)\sum_{\substack{a\Mod{2\widetilde{c}}\\2|a,\,(a,\widetilde{c})=1}} \ep_{\widetilde{c}}^{-2k}\(\frac{2a}{\widetilde c}\)\ee\(\frac{na}{2\widetilde{c}}\)\\
		&=\ee(\tfrac k4)\ep_{\widetilde{c}}^{-2k}\sum_{a\Mod{\widetilde{c}}^*} \(\frac{a}{\widetilde c}\)\ee\(\frac{na}{\widetilde{c}}\)
		=\ee(\tfrac k4)\prod_{p^\nu\|\widetilde{c}} \alpha_{2k}(p^\nu,n). 
	\end{split}
\end{align}
For $p$ as a prime, we also define
\begin{align*}
	A_{2k}(2,n,s)\defeq \sum_{\nu=2}^\infty \frac{\alpha_{2k}(2^\nu,n)}{2^{2s\nu}} \quad \text{and}\quad 
	A_{2k}(p,n,s)\defeq \sum_{\nu=0}^\infty \frac{\alpha_{2k}(p^\nu,n)}{p^{2s\nu}} \quad \text{for }p>2. 
\end{align*}
Then \eqref{eqKLsum-n0weight3-into-prods} and \eqref{eqKLsum-n0weight3-into-prods-oddd} imply
\begin{equation}\label{eqKloostermanzeta-n0-product-asA}
	\sum_{2|c>0} \frac{S(n,0,c,\nu_{\Theta}^{2k})}{(2c)^{2s}}=\prod_{p} A_{2k}(p,n,s)
\end{equation}
and
\begin{equation}\label{eqKloostermanzeta-n0-product-asA-oddd}
	\sum_{2\nmid \widetilde{c}>0} \frac{S(n,0,\widetilde{c},\nu_{\Theta}^{2k})}{\widetilde{c}^{2s}}=\prod_{p>2} A_{2k}(p,n,s).
\end{equation}
In the following two subsections we compute $\alpha_{2k}$ and $A_{2k}$ in the cases $p=2$ or $p>2$, respectively.  

\subsubsection{\texorpdfstring{$p=2$}{p=2}}

It is direct to calculate the results for $\nu=0,1$. For $\nu\geq 2$ we divide into cases for $\nu$ is even or odd. 

If $\nu\geq 2$ is even, then for odd $d$, $1\leq d<2^\nu$, we can write $d=4d'+1$ or $d=4d'+3$ for $d'$ from $1$ to $2^{\nu-2}$. Since $2|\nu$, $(\frac{2^\nu}{d})=1$ and we have
\begin{align}
\begin{split}
	\alpha_{2k}(2^\nu,n) &= \sum_{d'=1}^{2^{\nu-2}}\left\{ \ee\(\frac{n}{2^\nu}\)   \ee\(\frac{nd'}{2^{\nu-2}}\)+i^{2k} \ee\(\frac{3n}{2^\nu} \) \ee\(\frac{nd'}{2^{\nu-2}}\)\right\}\\
	&=\(\ee\(\frac{n}{2^\nu}\)+i^{2k}\ee\(\frac{3n}{2^\nu} \)\)\sum_{d'=1}^{2^{\nu-2}}\ee\(\frac{nd'}{2^{\nu-2}}\)\\
    &=\left\{\begin{array}{ll}
         0,& \text{if }2^{\nu-2}\nmid n \\
         \ee(\frac \ell 4)(1+(-1)^\ell i^{2k}) 2^{\nu-2},&\text{if }n=2^{\nu-2}\ell,\ \ell\in \Z.  
    \end{array}
    \right. 
\end{split}
\end{align}

If $\nu\geq 3$ is odd, then $(\frac {2^\nu}d)=(\frac 2d)$ and we get
\begin{align}
    \begin{split}
        \alpha_{2k}(2^\nu,n) &=\(\ee\(\frac{n}{2^\nu}\)-i^{2k}\ee\(\frac{3n}{2^\nu}\) -\ee\(\frac{5n}{2^\nu}\)+i^{2k}\ee\(\frac{7n}{2^\nu} \)\) \sum_{d'=1}^{2^{\nu-3}} \ee\(\frac{nd'}{2^{\nu-3}}\)\\
        &=\left\{\begin{array}{ll}
             0,& \text{if }2^{\nu-3}\nmid n, \\
             0,& \text{if }n=2^{\nu-3}\ell,\ \ell\in \Z,\ 2|\ell,\\
             4\ee(\frac \ell 8)\mathbf{1}_\Z\(\frac{\ell-2k}4\)2^{\nu-3},&\text{if }n=2^{\nu-3}\ell,\ \ell\in \Z,\ 2\nmid \ell. 
        \end{array}
        \right.
    \end{split}
\end{align}
Here $\mathbf{1}_{\Z}(x)=1$ if $x\in \Z$ and $0$ otherwise.

According to \cite[Chapter~2]{WangPeiMFHalfBook}, or by direct calculation, we have the following lemma. The facts that $\ee(\frac{u-1}8)=(\frac{u}2)$ for $u\equiv 1\Mod 4$ and $\ee(\frac{u+1}8)=(\frac{-u}2)$ for $u\equiv 3\Mod 4$ are helpful. 
\begin{lemma}\label{lemmaA2k2ns}
	For $k=\lambda+\frac 12=\frac 12$ or $\frac 32$, let $h=\nu_2(n)$ and we have
	\begin{align*}
		&A_{2k}(2,n,s)=2^{-4s}(1+i^{2k})\times \\
		&\left\{
		\displaystyle \begin{array}{ll}
			\displaystyle  \frac{1-2^{(h-1)(1-2s)}}{1-2^{2(1-2s)}}-2^{(h-1)(1-2s)},&\ 2\nmid h\geq 1;\vspace{10px} \\
			-1, &\ h=0,\ n\equiv -2k\Mod 4;\vspace{10px}\\
			\displaystyle \frac{1-2^{h(1-2s)}}{1-2^{2(1-2s)}}-2^{h(1-2s)}, &\ 2|h\geq 2,\ \dfrac{n}{2^{h}}\equiv -2k\Mod 4;\vspace{10px}\\
			\displaystyle \frac{1-2^{(h+2)(1-2s)}}{1-2^{2(1-2s)}}+\(\frac{(-1)^{\lambda}n/2^{h}}2\)2^{\frac 12+(h+1)(1-2s)}, &\ 2|h,\ \dfrac{n}{2^{h}}\equiv 2k\Mod 4;\\
		\end{array}
		\right.
	\end{align*}
	
\end{lemma} 
Particularly, when $-n=m^2\geq 1$ is a square, we can write $m_o$ as the odd part of $m$. We get $2|h$, $\frac n{2^h}=-m_o^2\equiv 3\Mod 4$, $(\frac{-n/2^h}2)=(\frac{m_o^{2}}{2})=1$, 
\begin{equation}\label{eqA32n-bound-square}
	\frac {\sqrt 2}{32}\leq \left|\frac{A_3(2,-m^2,s)}{1+2^{-(2s-\frac 12)}}\right| \leq \frac{\sqrt 2}4 \(\frac{\nu_2(n)}2+1\)\quad \text{ for }\re s\in[\tfrac 12,1],
\end{equation}
and
\begin{equation}\label{eqA32n-value-square-s34}
		\left.\frac{A_3(2,-m^2,s)}{1+2^{-(2s-\frac 12)}}\right|_{s=\frac 34}=\frac{2^{-3}(1-i)(2-2^{-h/2}+(\frac{m_o^2}2)2^{-h/2})}{3/2} = \frac{\ee(-\tfrac 18)}{3\sqrt 2}. 
\end{equation}

When $n\neq 0$ is not a negative square, we still have the rough bound
\begin{equation}\label{eqA32n-bound-nonsquare}
	0\leq |A_3(2,n,s)| \leq \frac{\sqrt 2}4 \(\frac{\nu_2(n)}2+1\)\quad \text{ for }\re s\in[\tfrac 12,1]. 
\end{equation}

\subsubsection{\texorpdfstring{$p>2$}{p>2}}
For $\nu=0$ we have $\alpha_{2k}(1,n)=1$. For $\nu\geq 1$, we write $a=a'+pb$ for $a'$ from $1$ to $p-1$ and $b$ from $1$ to $p^{\nu-1}$. Then $(\frac ap)=(\frac{a'}p)$ and we have
\begin{align*}
	\alpha_{2k}(p^\nu,n)=\ep_{p^\nu}^{-2k} \sum_{a'=1}^{p-1} \(\frac{a'}{p^\nu}\) \ee\(\frac{na'}{p^{\nu}}\)\sum_{b=1}^{p^{\nu-1}} \ee\(\frac{nb}{p^{\nu-1}}\). 
\end{align*}
If $\nu-1>\nu_p(n)$, i.e. $\nu\geq \nu_p(n)+2$, the latter sum on $b$ is zero, otherwise the latter sum is $p^{\nu-1}$. The remaining calculation is under the assumption $\nu\leq \nu_p(n)+1$. 

If $\nu$ is odd, then 
\begin{align*}
	\alpha_{2k}(p^\nu,n)&=\ep_{p^\nu}^{-2k}p^{\nu-1} \sum_{a=1}^{p-1} \(\frac ap\) \ee\(\frac{an/p^{\nu-1}}p\)\\
	&=\left\{
	\begin{array}{ll}
		0, &\ p|\frac{n}{p^{\nu-1}}, \text{ i.e. }\nu\leq \nu_p(n);\\
		p^{\nu-\frac 12}\(\frac{(-1)^{k-1/2}\,n/p^{\nu-1}}p\), &\ \nu=\nu_p(n)+1. 
	\end{array}
	\right. 
\end{align*}
Here in the last case we used $\ep_{p^\nu}=\ep_{p}$ for odd $\nu$ and $\ep_p^2=(\frac{-1}p)$. 

If $\nu$ is even, then $\ep_{p^\nu}=1$ and $(\frac{a'}{p^\nu})=1$ for $\gcd(a',p)=1$. We have
\begin{align*}
	\alpha_{2k}(p^\nu,n)=p^{\nu-1} \sum_{a=1}^{p-1} \ee\(\frac{an/p^{\nu-1}}p\)=p^{\nu-1}\cdot\left\{
	\begin{array}{ll}
		p-1, & \ \nu\leq \nu_p(n);\\
		-1, & \ \nu=\nu_p(n)+1. 
	\end{array}
	\right. 
\end{align*}

Finally, we conclude
\begin{align}
	a_{2k}(p^\nu,n)=\left\{
	\begin{array}{ll}
		1,&\ \nu=0;\\
		0, &\ 2\nmid \nu\text{ and }1\leq\nu\leq \nu_p(n)  ;\\
		p^{\nu-1}(p-1), &\ 2|\nu \text{ and }  2\leq \nu\leq \nu_p(n) ;\\
		p^{\nu-\frac 12} \(\frac{(-1)^{k-1/2}\, n/p^{\nu-1}}p\), &\ 2\nmid \nu \text{ and }  \nu=\nu_p(n)+1 ;\\
		-p^{\nu-1}, &\  2|\nu \text{ and }  \nu=\nu_p(n)+1;\\
		0, & \ \nu\geq \nu_p(n)+2. 
	\end{array}
	\right. 
\end{align}
Then we have the following lemma: 
\begin{lemma}
	For prime $p>2$, $k\in \Z+\frac 12$, $n\neq 0$ and $h\defeq \nu_p(n)$, we have
	\begin{align}
		A_{2k}(p,n,s)=1+\left\{
		\begin{array}{ll}
			\displaystyle \frac{p^{2(1-2s)}-p^{(h+1)(1-2s)}}{1-p^{2(1-2s)}}(1-p^{-1})-p^{(h+1)(1-2s)-1}, &\text{if }2\nmid h;\vspace{10px}\\
			\displaystyle  \frac{p^{2(1-2s)}-p^{(h+2)(1-2s)}}{1-p^{2(1-2s)}}(1-p^{-1})&  \vspace{10px}\\
			\displaystyle \qquad \qquad +\(\frac{(-1)^{k-1/2}\, n/p^h}p\)p^{(h+1)(1-2s)-\frac 12}, &\text{if }2| h.
		\end{array}
		\right. 
	\end{align}
\end{lemma}

When $-n=m^2\geq 1$ is a square, by writing $h=\nu_p(n)$, we have $2|h$, $(\frac{-n/p^h}p)=1$, 
\begin{equation}\label{eqA3pn-bound-square}
	1\leq \left|\frac{A_3(p,-m^2,s)}{1+p^{-(2s-\frac 12)}}\right| \leq \frac{\nu_p(n)}2+2\quad \text{ for }\re s\in[\tfrac 12,1],
\end{equation}
and
\begin{equation}\label{eqA3pn-value-square-s34}
	\left.\frac{A_3(p,-m^2,s)}{1+p^{-(2s-\frac 12)}}\right|_{s=\frac 34}=\frac{1+p^{-1}-p^{-\frac{h+2}2}+p^{-\frac{h+1}2-\frac 12}}{1+p^{-1}} = 1. 
\end{equation}

When $n\neq 0$ is not a negative square, we still have the rough bound
\begin{equation}\label{eqA3pn-bound-nonsquare}
	0\leq |A_3(p,n,s)| \leq \frac{\nu_p(n)}2+2\quad \text{ for }\re s\in[\tfrac 12,1]. 
\end{equation}

A special case is that if $p>2$ and $p\nmid n$, i.e. $\nu_p(n)=0$, then 
\begin{equation}\label{eqA3-pnmidn-square}
	A_3(p,n,s)=1+\(\frac{-n} p \) p^{-(2s-\frac 12)}=1+\(\frac{-4n} p \) p^{-(2s-\frac 12)}=\frac{1- p^{-(4s-1)}} {1-(\frac{-4n}p) p^{-(2s-\frac 12)}}. 
\end{equation}

We have the following lemma on $A_3(p,n,s)$ at $s=\frac 34$ for $n\neq 0$. 
\begin{lemma}\label{lemmaA3pn-product-convergent-s34}
	If $n$ is not a negative square nor $0$, then 
	\[\prod_{p\nmid 2n} A_3(p,n,s)\quad \text{	is convergent at }s=\tfrac 34.\]
\end{lemma}
\begin{proof}
	This proof can be traced back to \cite[Lemma~4.1 and (4.05)]{Bateman1951ThreeSquares}. The product
	\[\prod_{p\nmid 2n} A_3(p,n,s)=\sum_{\substack{k\text{ square-free}\\(k,2n)=1}}\(\frac{-n}k\)\frac 1{k^{2s-\frac 12}}=\sum_{k\text{ square-free}}\(\frac{-4n}k\)\frac 1{k^{2s-\frac 12}}\]
	is convergent at $s=\frac 34$ and converge to a positive value
	\[K(-4n)\sum_{(\ell,4n)=1}\frac{\mu(\ell)}{\ell^2},\quad \text{where} \quad K(-4n)=\sum_{m=1}^\infty \(\frac{-4n}m\)\frac 1m. \qedhere\]
\end{proof}

We conclude the following proposition based on the above discussion.  
\begin{proposition}\label{prop_KloostermanSelbergZetaF_m=0_weightk3/2}
	For $n\neq 0$, let $Z_n(s)$ be the Kloosterman-Selberg zeta function 
    \[Z_n(s)=\sum_{2|c>0}\frac{S(n,0,c,\nu_{\Theta}^3)}{c^{2s}}\quad \text{or}\quad \sum_{2\nmid \widetilde{c}>0}\frac{S(n,0,\widetilde{c},\nu_{\Theta}^3)}{{\widetilde{c}}^{2s}}. \]
    Then for either case, there exist functions $C_1(n,s)$ and $C_2(n,s)$ entire in $s$, and positive absolute constants $C$, $C'$ such that for $s\in [\frac 12,1]$,
	\[0<C'\leq |C_1(n,s)|\leq C\log n\quad \text{and}\quad |C_2(n,s)|\leq C\log n, \]
    and we have
	\[Z_n(s)
	=\left\{\begin{array}{ll}
		\displaystyle C_1(n,s) \frac{\zeta(2s-\frac 12)}{\zeta(4s-1)},&\text{if }n=-m^2<0;\vspace{10px}\\
		\displaystyle C_2(n,s) \frac{L(2s-\frac 12, (\frac{-4n}\cdot))}{\zeta(4s-1)}, &\text{other }n\neq 0. 
	\end{array}\right. \]
	Specifically, 
	\[\Res_{s=\frac 34} \sum_{2|c>0}\frac{S(n,0,c,\nu_{\Theta}^3)}{c^{2s}} =\left\{\begin{array}{ll}
		 \ee(-\frac 18)\frac{2}{\pi^2},&\text{ if }n=-m^2<0;\\
		0, &\text{ other }n\neq 0
		\end{array}\right. 
		\]
		and 
	\[\Res_{s=\frac 34} \sum_{2\nmid \widetilde{c}>0}\frac{S(n,0,\widetilde{c},\nu_{\Theta}^3)}{{\widetilde{c}}^{2s}} =\left\{\begin{array}{ll}
		\ee(\frac 38)\frac{2}{\pi^2},&\text{ if }n=-m^2<0;\\
		0, &\text{ other }n\neq 0
	\end{array}\right. 
	\]
\end{proposition}
\begin{proof}
	The proof is direct by combining \eqref{eqKloostermanzeta-n0-product-asA}, \eqref{eqKloostermanzeta-n0-product-asA-oddd}, \eqref{eqA32n-bound-square}, \eqref{eqA32n-value-square-s34}, \eqref{eqA32n-bound-nonsquare}, \eqref{eqA3pn-bound-square}, \eqref{eqA3pn-value-square-s34}, \eqref{eqA3pn-bound-nonsquare}, \eqref{eqA3-pnmidn-square} and Lemma~\ref{lemmaA3pn-product-convergent-s34}. Since for every $n$, there are only finitely many $p|2n$, the factors $\zeta(2s-\frac 12)/\zeta(4s-1)$ and $L(2s-\frac 12,(\frac{-4n}\cdot))/\zeta (4s-1)$ are the result of \eqref{eqA3-pnmidn-square}. The residue in the end is given by \eqref{eqA32n-value-square-s34}, \eqref{eqA3pn-value-square-s34} and 
	\[\Res_{s=\frac 34}\ 2^{2s}\,\frac{\ee(-\frac 18)}{3\sqrt 2}\frac{\zeta(2s-\frac 12)}{\zeta(4s-1)}=\ee(-\tfrac 18)\frac{2}{\pi^2},\quad \Res_{s=\frac 34}\frac{\ee(\frac 38)\zeta(2s-\frac 12)}{(1+2^{-1})\zeta(4s-1)}=\ee(\tfrac 38)\frac{2}{\pi^2}. \qedhere \]
\end{proof}

Using Perron's formula (see e.g. \cite[\S17]{davenport}), we get the following proposition.

\begin{proposition}\label{propSumofKloostermansums-weight32-one-of-mn0}
	For $n\neq 0$, let 
    \[S_n(x)=\sum_{2|c\leq x} \frac{S(n,0,c,\nu_{\Theta}^3)}{c\cdot \ee(-\frac 18)}\quad\text{or}\quad \sum_{2\nmid\widetilde{c}\leq x} \frac{S(n,0,\widetilde{c},\nu_{\Theta}^3)}{\widetilde{c}\cdot \ee(\frac 38)}. \]
    Then we have
	\begin{equation}S_n(x)
	=\left\{\begin{array}{ll}
			\displaystyle \frac{8}{\pi^2}x^{\frac 12}+O_\ep\(x^{\frac 16+\ep} n^\ep\),&\ n=-m^2<0;\vspace{5px}\\
			\displaystyle O_\ep\(x^{\frac 16+\ep} |n|^{\frac 3{16}+\ep}\),&\ \text{other }n\neq 0, \\
		\end{array}
		\right.
	\end{equation}
	The same bound holds if $S(n,0,\cdot,\nu_\Theta^3)$ is changed to $S(0,n,\cdot,\nu_\Theta^3)$ because of \eqref{eqKloosterman-Sums-alter-mn-inTheta}. 
\end{proposition}

\begin{proof}[Proof outline]
	We only prove the case for $2|c$ because the other case follows in the same way. We define the Kloosterman-Selberg zeta function as
	\[Z(s)\defeq \sum_{2|c>0} \frac{S(n,0,c,\nu_{\Theta}^3)}{c^{2s}}. \]
	The Weil bound \eqref{eqWeilBound-Theta3} gives us
	\[\sum_{2|c\leq x}\frac{|S(n,0,c,\nu_{\Theta}^3)|}{c^{\frac 32+\delta}}\ll_{\ep,\delta} |n|^\ep \log x,\quad \text{for any }\delta>0. \]
	Applying Perron's formula, for $\delta>0$ we get
	\[\left|\sum_{2|c\leq x} \frac{S(n,0,c,\nu_{\Theta}^3)}{c} - \frac 1{2\pi i}\int_{\frac 12+\delta -iT}^{\frac 12+\delta+iT} Z(\tfrac{1+s}2)\frac{x^s}s ds\right|\ll_{\ep,\delta} \frac{x^{\frac 12+\delta}}T |n|^\ep \log x.  \]
 We then shift the path of integral to $\frac 12+\delta-iT\rightarrow \delta-iT\rightarrow \delta+iT\rightarrow \frac 12+\delta +iT$. 
	
	Estimates on the integral along this path follow from the uniform bounds of $\zeta(s)$ and $L(s,(\frac{-4n}\cdot))$ in the critical strip. By \cite[Corollary~1.17]{Montgomery_Vaughan_2006_Book}, we get the convexity bound
	\[\zeta(\sigma+it)\ll (1+\tau^{1-\sigma})\min\(\frac{1}{|\sigma-1|},\log t\) \quad \text{for }\delta \leq \sigma\leq 2,\ |t|\geq 1. \]
	By Burgess's bound \cite[Theorem~1]{HeathBrown1978HybridBound} \cite{Burgess1963bound} we have the subconvexity bound
	\[L\(\sigma+it,\(\frac{-4n}\cdot\)\)\ll |n|^{\frac 38(1-\sigma)+\ep} |s|^{\frac{1-\sigma}4+\ep} \quad \text{for }\frac 12\leq \sigma\leq 1,\ |t|\geq 1. \]
	Therefore, by combining above and Proposition~\ref{prop_KloostermanSelbergZetaF_m=0_weightk3/2} we get
	\begin{align*}
		\frac 1{2\pi i}&\(\int_{\frac 12+\delta -iT}^{\delta-iT}+\int_{\delta -iT}^{\delta+iT}+\int_{\delta+iT}^{\frac 12+\delta+iT}\) Z(\tfrac{1+s}2)\frac{x^s}s dx\\
		&\ll_{\ep,\delta}\left\{\begin{array}{ll}
			|xnT|^\ep \(x^{\delta} T^{\frac 12}+x^{\frac 12+\delta}T^{-1}\),&\ n=-m^2<0;\\
			|n|^{\frac 3{16}+\ep}|xnT|^\ep \(x^{\delta} T^{\frac 12}+x^{\frac 12+\delta}T^{-1}\),&\ \text{other }n\neq 0. 
		\end{array}
		\right.
	\end{align*}
	In the first case of the proposition where $n$ is a negative square, $Z(\frac{1+s}2)$ has a pole at $s=\frac 12$ with residue
	\[\Res_{s=\frac 12}Z(\tfrac{1+s}2)=2\Res_{s=\frac 34}Z(s)=\ee(-\tfrac 18)\frac{4}{\pi^2}. \]
	So the residue of $Z(\tfrac{1+s}2)\frac{x^s}s$ at $s=\frac 12$ is $\ee(-\tfrac 18)\frac{8}{\pi^2}x^{\frac 12}$. We finish the proof by taking $T=x^{\frac 13}$, $\delta=\ep$ and combining all of the above calculations. 
\end{proof}

\begin{remark}
	Burgess's bound has been improved by recent research in subconvexity bounds for $L$-functions, but it is enough for our estimates here due to the worse bound on Kloosterman sums with $mn\neq 0$ in Proposition~\ref{propSumofKloostermansums-weight32-mnnot0} in the $mn$-aspect. Comparing with the proposition, the second author wishes to conclude better ``subconvexity bound" for sum of Kloosterman sums in general ($mn\neq 0$). Such estimates will also improve the conclusions in estimates in $n$-aspect in this paper.  
\end{remark}

There is a remaining case $m=n=0$, which we quickly deal with here. 
For $2|c>0$, 
\begin{equation}
	S(0,0,c,\nu_{\Theta}^3)=\left\{\begin{array}{ll}
		\sqrt 2\ee(-\frac 18)n\phi(2n),&\ \frac{\sqrt {2c}}2=n\in \Z_+;\\
		0,&\ \text{otherwise}. 
	\end{array}
	\right.
\end{equation}
Therefore, we get
\begin{align*}
	\sum_{2|c>0}\frac{S(0,0,c,\nu_{\Theta}^3)}{c^{2s}}&= \frac{\sqrt 2\ee(-\frac 18)}{2^{1-2s}} \sum_{n=1}^\infty \frac{\phi(2n)}{(2n)^{4s-1}}\\
	&=\frac{\sqrt 2\ee(-\frac 18)}{2^{1-2s}}\cdot \frac 1{2^{4s-1}-1}\cdot \frac{\zeta(4s-2)}{\zeta(4s-1)},
\end{align*}
\begin{equation}
	\Res_{s=\frac 34}	\sum_{2|c>0}\frac{S(0,0,2c,\nu_{\Theta}^3)}{c^{2s}}= \frac{2\ee(-\frac 18)}{3} \times \frac 14 \times \frac{6}{\pi^2}=\frac{\ee(-\frac 18)}{\pi^2},  
\end{equation}
and by Perron's formula, 
\begin{equation}\label{eqpropSumofKloostermansums-weight32-bothmn0}
	\sum_{2|c\leq x} \frac{S(0,0,c,\nu_{\Theta}^3)}{c}=\ee(-\tfrac 18)\frac{4}{\pi^2} x^{\frac 12}+O_\ep(x^{\frac 16+\ep}). 
\end{equation}

For $2\nmid \widetilde{c}>0$ we also have
\begin{equation}
	S(0,0,\widetilde{c},\nu_{\Theta}^3)=\left\{\begin{array}{ll}
		\ee(\frac 38)n\phi(n),&\ \sqrt {\widetilde{c}}=n\text{ is odd};\\
		0,&\ \text{otherwise}. 
	\end{array}
	\right.
\end{equation}
Then
\begin{align*}
	\sum_{2\nmid \widetilde{c}>0}\frac{S(0,0,\widetilde{c},\nu_{\Theta}^3)}{{\widetilde{c}}^{2s}}&= \ee(\tfrac 38) \sum_{\substack{n=1\\n\text{ odd}}}^\infty \frac{\phi(n)}{n^{4s-1}}\\
	&=\ee(\tfrac 38) \frac{2^{4s-1}-2}{2^{4s-1}-1}\cdot \frac{\zeta(4s-2)}{\zeta(4s-1)}. 
\end{align*}
We have
\begin{equation}
	\Res_{s=\frac 34}\sum_{2\nmid \widetilde{c}>0}\frac{S(0,0,\widetilde{c},\nu_{\Theta}^3)}{{\widetilde{c}}^{2s}}= \frac{2\ee(\frac 38)}{3} \cdot \frac 14\cdot \frac{6}{\pi^2}=\frac{\ee(\frac 38)}{\pi^2}
\end{equation}
and apply Perron's formula to get
\begin{equation}\label{eqpropSumofKloostermansums-weight32-bothmn0-oddd}
	\sum_{2\nmid \widetilde{c}\leq x} \frac{S(0,0,\widetilde{c},\nu_{\Theta}^3)}{\widetilde{c}}=\ee(\tfrac 38)\frac{4}{\pi^2} x^{\frac 12}+O_\ep(x^{\frac 16+\ep}). 
\end{equation}

\section{Real-variable Kloosterman sums in weight \texorpdfstring{$3/2$}{3/2}}
\label{SectionRVKLweight3/2}
Recall the real-variable Kloosterman sums defined in \eqref{eqReal-Val-Kloosterman-Sums}. In this section we assume $r^2\in \R_+$. We would like to analyze the convergence of
    \[\sum_{2|c>0}\frac{K(r^2,n,c,\nu_{\Theta}^3)}c J_{2s-1}\(\frac{2\pi |r|\sqrt n}c\)\quad \text{for }n>0\]
and find properties of 
    \[\sum_{2|c>0}\frac{K(r^2,0,c,\nu_{\Theta}^3)}{c^{2s}}\quad \text{and}\quad \sum_{2|c>0}\frac{K(r^2,n,c,\nu_{\Theta}^3)}c I_{2s-1}\(\frac{2\pi |r^2 n|^{\frac 12}}c\)\quad \text{for }n<0\]
when $s\rightarrow \frac 34^+$, as well as the corresponding cases for $2\nmid \widetilde{c}>0$. 

First we prove a Weil-type bound for $K(r^2,n,c,\nu_{\Theta}^3)$ and $\widetilde{K}(r^2,n,{\widetilde{c}},\nu_{\Theta}^3)$ for $2|c$ and $2\nmid {\widetilde{c}}$. 
\begin{lemma}\label{lemmaWeiltypebound-real-val-k3/2}
	For $r^2\in \R_+$, $n\neq 0$, $2|c$ and $2\nmid {\widetilde{c}}$, we have 
	\[K(r^2,n,c,\nu_{\Theta}^3)\ll_\ep (n,c)^{\frac 12+\ep} c^{\frac 12+\ep}\quad \text{and}\quad  \widetilde{K}(r^2,n,{\widetilde{c}},\nu_{\Theta}^3)\ll_\ep (n,{\widetilde{c}})^{\frac 12+\ep} {\widetilde{c}}^{\frac 12+\ep}. \]
\end{lemma}
\begin{proof}
	The proof is the same as Lemma~\ref{lemmaWeiltypebound-real-val-k2}, taking the Weil bound for half-integral weight Kloosterman sums \eqref{eqWeilBound-Theta3} into account. For the $2\nmid {\widetilde{c}}$ case, we combine \eqref{eqrelation-KLsum-evenodd-k3/2} and \eqref{eqWeilBound-Theta3} to get the desired estimate. 
\end{proof}

\begin{proposition}\label{propEstimate-sumofKlsum-real-val-k3/2-Weil1/2bound}
	For $r^2\in \R_+$, $x\geq 1$, $n\in \Z\setminus \{0\}$, $2|c$ and $2\nmid {\widetilde{c}}$, we have
	\[\left. \begin{array}{r}
		\displaystyle \sum_{2|c\leq x}\frac{|K(r^2,n,c,\nu_{\Theta}^3)|}c\\
		\displaystyle \sum_{2\nmid {\widetilde{c}}\leq x}\frac{|\widetilde{K}(r^2,n,{\widetilde{c}},\nu_{\Theta}^3)|}{\widetilde{c}}\\
	\end{array}\right\}
	\ll_\ep |n|^\ep x^{\frac12+\ep}. \]
\end{proposition}

\begin{proof}
	The proof follows directly from Lemma~\ref{lemmaWeiltypebound-real-val-k3/2} and the Weil bound of Kloosterman sums (when $r^2\in \Z$). 
\end{proof}

\subsection{Case \texorpdfstring{$n>0$}{n>0}}
We begin with an estimate for sums of real-variable Kloosterman sums. 
\begin{proposition}\label{propEstimate-sumofKlsum-real-val-k3/2-QihangAsymptBound<1/4}
	For $r^2\in \R_+$, $x\geq 1$, $n\in \Z_+$, and $X=\max(x,r^2)$, we have
	\[\left. \begin{array}{r}
		\displaystyle \sum_{2|c\leq x}\frac{K(r^2,n,c,\nu_{\Theta}^3)}c\\
		\displaystyle \sum_{2\nmid {\widetilde{c}}\leq x}\frac{\widetilde{K}(r^2,n,{\widetilde{c}},\nu_{\Theta}^3)}{\widetilde{c}}
	\end{array} \right\} \ll_\ep  X^{\frac 14}n^{\frac 3{16}}\(n^{\frac 1{16}-\delta}+u_n^{\frac 18}\)|Xn|^\ep, \]
	where $\delta=\frac 1{147}$ and $u_n=2^{\floor{\nu_2(n)}/2}\ll n^{1/2}$. 
\end{proposition}

\begin{proof}
When $r^2\in \Z_+$, the proposition follows from the better bounds in Proposition~\ref{propSumofKloostermansums-weight32-mnnot0} and Proposition~\ref{propSumofKloostermansums-weight32-mnnot0-4|c}. 
	We prove the bound for $r^2\in \R_+\setminus \Z$ and the sum on $2|c$ and the other case $2\nmid {\widetilde{c}}$ follows similarly. 
	By Lemma~\ref{lemmaExponential-interpolation}, we have
	\begin{align}
		\begin{split}
			\sum_{2|c\leq x}\frac{K(r^2,n,c,\nu_{\Theta}^3)}c
			&=\sum_{2|c\leq x} \sum_{\substack{k\in \Z\\|r^2-k|<c}} \frac{i\sin(\pi r^2)(-1)^k}{\ee(\frac{r^2-k}{2c})-1}\frac{S(k,n,c,\nu_{\Theta}^3)}{c^2}\\
			&=\frac{i\sin(\pi r^2)}{\pi} \sum_{\substack{k\in \Z\\|r^2-k|\leq x}}(-1)^k \sum_{\substack{2|c\\|r^2-k|<c\leq x}} \frac{S(k,n,c,\nu_{\Theta}^3)}{\(\ee(\frac{r^2-k}{2c})-1\)c^2}
		\end{split}
	\end{align}
	For $k\in \Z$ and $n>0$, we apply Proposition~\ref{propSumofKloostermansums-weight32-mnnot0} and Proposition~\ref{propSumofKloostermansums-weight32-one-of-mn0} to get
	\begin{equation}
		S(x)\defeq \sum_{2|c\leq x}\frac{S(k,n,c,\nu_{\Theta}^3)}{c} \ll_\ep \(x^{\frac 16}+A_u(k,n)\)|xkn|^\ep. 
	\end{equation}
	By partial summation and \eqref{eq:crude bound et-1 abs} we get
	\begin{align*}
		\sum_{\substack{2|c\\|r^2-k|<c\leq x}}& \frac{S(k,n,c,\nu_{\Theta}^3)}{\(\ee(\frac{r^2-k}{2c})-1\)c^2}=\int_{|r^2-k|}^x \frac {t^{-1}dS(t)}{\ee(\frac{r^2-k}{2t})-1}\\
		&=\frac{t^{-1}S(t)}{\ee(\frac{r^2-k}{2t})-1}\Big|_{|r^2-k|}^x+\int_{|r^2-k|}^x S(t) \frac{\ee(\frac{r^2-k}{2t})-1+\pi i \ee(\frac{r^2-k}{2t})\frac{r^2-k}{t}}{\(\ee(\frac{r^2-k}{2t})-1\)^2t^2}dt\\
		&\ll \(x^{\frac 16}+A_u(k,n)\)\frac{|xkn|^\ep}{|r^2-k|}. 
	\end{align*}
	By Proposition~\ref{propSumofKloostermansums-weight32-mnnot0}, 
	\[A_u(k,n)\ll |k|^{\frac 14}n^{\frac 3{16}}\(n^{\frac 1{16}-\delta}+u_n^{\frac 18}\), \]
	where $\delta=\frac 1{147}$ and $u_n=2^{\floor{\nu_2(n)/2}}\leq n^{1/2}$. 
	
	If $x<r^2$, summing $k$ in $|r^2-k|\leq x$ with the help of Lemma~\ref{lemmaInequality-estimate-kAlpha-to-r-to-x} and $\frac{\sin (\pi r^2)}{\|r^2\|}\in [2,\pi]$, we get the desired bound. If $x>r^2$, we have an extra contribution from $k=0$. By Proposition~\ref{propSumofKloostermansums-weight32-one-of-mn0}, a similar partial summation process as above gives a better bound: 
	\begin{align*}
		\sum_{\substack{2|c\\r^2<c\leq x}} \frac{S(0,n,c,\nu_{\Theta}^3)}{(\ee(\frac{r^2}{2c})-1)c^2}\ll_\ep x^{\frac 16+\ep} n^{\frac 3{16}+\ep}. 
	\end{align*}
	In either case we get the desired bound and the proposition is proved. 
\end{proof}

Comparing Proposition~\ref{propEstimate-sumofKlsum-real-val-k3/2-Weil1/2bound} and Proposition~\ref{propEstimate-sumofKlsum-real-val-k3/2-QihangAsymptBound<1/4}, we find that Proposition~\ref{propEstimate-sumofKlsum-real-val-k3/2-Weil1/2bound} is better when $x$ is small, while Proposition~\ref{propEstimate-sumofKlsum-real-val-k3/2-QihangAsymptBound<1/4} is better when $x$ is large. In particular,  Proposition~\ref{propEstimate-sumofKlsum-real-val-k3/2-QihangAsymptBound<1/4} allows us to establish the convergence of sums in the following result, since the exponent of $x$ is smaller than $\frac 12$. To estimate the growth rate of the following sums as $n$ becomes large, both of the propositions above should be applied to refine the estimate. 

\begin{proposition}\label{propSum-real-val-Kloosterman-with-J2s-1}
	For $r^2\in \R_+$ and $n\in \Z_+$, the following sums are convergent for $s\geq \frac 34$, and both have the same estimate depending on $r^2n\leq 1$ or $r^2 n\geq 1$: 
	\[\left. \begin{array}{r}
		\displaystyle \sum_{2|c>0}\frac{K(r^2,n,c,\nu_{\Theta}^3)}c J_{2s-1}\(\frac{2\pi |r|\sqrt n}{c}\)\\
		\displaystyle \sum_{2\nmid {\widetilde{c}}>0}\frac{\widetilde{K}(r^2,n,{\widetilde{c}},\nu_{\Theta}^3)}{\widetilde{c}} J_{2s-1}\(\frac{2\pi |r|\sqrt n}{{\widetilde{c}}}\)
	\end{array}\right\} 
	\ll_\ep \left\{\begin{array}{ll}
	      |r|^{\frac 12} \max\(n^{\frac {145}{294}}, n^{\frac {7+\kappa}{16}}\)n^{\ep}& \text{if }r^2\leq \frac 1n, \\
	    (1+|r|^{\frac 12+\ep}) n^{\frac 14+\ep} & \text{if }r^2\geq \frac 1n. 
	\end{array} \right.  \]
    Here $\kappa\in [0,1]$ is determined by $2^{\nu_2(n)}\ll |n|^\kappa$. 
\end{proposition}
\begin{remark}
	For simplicity, the exponent of $n$ in the last bound is $\frac {145}{294}\approx 0.493\ldots$ if $\nu_2(n)$ is absolutely bounded, e.g. if $n$ is odd, and the exponent is $\frac 12=0.5$ if $n$ is a power of $2$. 
\end{remark}

\begin{proof}
	We prove the case for $2|c>0$ and the other case $2\nmid {\widetilde{c}}>0$ can be proved in the same way. The convergence for $\re(s)>1$ is given by the definition in \eqref{eqDefineFks}, so we only need to focus on $s\in [\frac 34,1+\delta]$ for any fixed $\delta>0$. 
    
    By \cite[(10.7.3), (10.7.8), (10.14.1), (10.14.4)]{dlmf}, for $x\in \R$ and $\nu\geq -\frac 12$, 
\begin{equation}\label{eqboundJBesselfor3/2}
	J_{\nu}(x) \ll \min\(x^{-\frac 12}, x^{\nu}\),\quad \text{hence }J_{2s-1}(x)\ll \min\(x^{-\frac 12},x^{\frac 12}\). 
\end{equation}
By \cite[(10.6.2)]{dlmf},
\begin{align}\label{eqbound-derivJ2s-1}
	\begin{split}
		\frac d{dt} J_{2s-1}\(\frac {2\pi |r|\sqrt n}t\) &=\frac{2\pi |r|\sqrt n}{t^2} J_{2s}\(\frac {2\pi |r|\sqrt n}t\)-\frac{(2s-1)}{t}J_{2s-1}\(\frac {2\pi  |r|\sqrt n}t\)\\
		&\ll\frac{|r^2n|^{s- \frac 12}}{t^{2s}}\ll \frac{|r^2 n|^{\frac 14}}{t^{\frac 32}}\quad\text{ if }t\gg |r|n^{\frac 12}. 
		\end{split}
	\end{align}
The convergence in the proposition is by partial summation and  Cauchy's criterion: for $y>x>\max(1,r^2,2\pi |r|\sqrt n)$, Proposition~\ref{propEstimate-sumofKlsum-real-val-k3/2-QihangAsymptBound<1/4} and \eqref{eqbound-derivJ2s-1} gives
\begin{equation}\label{eq:weight3/2SumKLsumJ2s-1Cauchy}
	\sum_{\substack{2|c\\ x<c\leq y}}\frac{K(r^2,n,c,\nu_{\Theta}^3)}c J_{2s-1}\(\frac{2\pi  |r|\sqrt n}{c}\)\ll_{\ep} |r|^{\frac 12} n^{\frac 12+\ep} x^{-\frac 14+\ep}.  
\end{equation}

Let $\beta>\frac 12$, we will choose it later. We have the following two cases: 

(1) when $r^2\geq \frac 1n$, by Lemma~\ref{lemmaWeiltypebound-real-val-k3/2} and \eqref{eqboundJBesselfor3/2}, we have
\begin{align}\label{eqEstimateSumof-real-val-KL-J-weight3/2-prog1}
	\begin{split}
		\sum_{2|c\leq |r^2n|^\beta }&\frac{K(r^2,n,c,\nu_{\Theta}^3)}c J_{2s-1}\(\frac{2\pi  |r|\sqrt n}{c}\)\\
        &\ll_\ep  |r^2 n|^{-\frac 14}\sum_{2|c\leq |r|\sqrt n } (n,c)^{\frac 12+\ep}c^{\ep}+|r^2 n|^{\frac 14}\sum_{\substack{2|c\\|r^2 n|^{1/2}\leq c\leq |r^2 n|^\beta}} (n,c)^{\frac 12+\ep} c^{-1+\ep}\\
        &\ll_\ep|r^2n|^{\frac 14+\beta\ep+\ep}.  
	\end{split}
\end{align}
By partial sum with Proposition~\ref{propEstimate-sumofKlsum-real-val-k3/2-Weil1/2bound} and \eqref{eqbound-derivJ2s-1}, we have
\begin{equation}\label{eqEstimateSumof-real-val-KL-J-weight3/2-prog2}
	\sum_{2|c\geq  |r^2n|^\beta }\frac{K(r^2,n,c,\nu_{\Theta}^3)}c J_{\frac 12}\(\frac{2\pi |r|\sqrt n}{c}\)\ll_\ep \max(|r|^{\beta},|r|)^{\frac 12} n^{\frac{7}{16}-\frac \beta 4+\ep}\(n^{\frac 1{16}-\delta}+u_n^{\frac 18}\). 
\end{equation}
Whenever $\beta\geq 1$, the exponent of $n$ in \eqref{eqEstimateSumof-real-val-KL-J-weight3/2-prog2} is dominated by \eqref{eqEstimateSumof-real-val-KL-J-weight3/2-prog1}. 

(2) when $r^2\leq \frac 1n$, by partial sum with Proposition~\ref{propEstimate-sumofKlsum-real-val-k3/2-Weil1/2bound} and \eqref{eqbound-derivJ2s-1}, we also have
\begin{equation}\label{eqEstimateSumof-real-val-KL-J-weight3/2-prog2-r2n<1}
	\sum_{2|c}\frac{K(r^2,n,c,\nu_{\Theta}^3)}c J_{\frac 12}\(\frac{2\pi |r|\sqrt n}{c}\)\ll_\ep |r|^{\frac 12} n^{\frac{7}{16}+\ep}\(n^{\frac 1{16}-\delta}+u_n^{\frac 18}\). 
\end{equation}

The proof for the case $2|c$ is done by combining \eqref{eqEstimateSumof-real-val-KL-J-weight3/2-prog2-r2n<1} and \eqref{eqEstimateSumof-real-val-KL-J-weight3/2-prog1} with $\beta=1$. 
\end{proof}

\subsection{Case \texorpdfstring{$n=0$}{n=0}}

\begin{proposition}\label{propSum-real-val-Kloosterman-with-constantterm-residue}
	For $r\in \R_+$, we have
	\begin{align*}
		\lim_{s\rightarrow \frac 34^+} \frac 1{\Gamma(s-\frac 34)}\sum_{2|c>0}\frac{K(r^2,0,c,\nu_{\Theta}^3)}{c^{2s}}=\ee(-\tfrac 18)\frac{\sin(\pi r^2)}{\pi^2r\sinh(\pi r)},\\
		\lim_{s\rightarrow \frac 34^+} \frac 1{\Gamma(s-\frac 34)}\sum_{2\nmid {\widetilde{c}}>0}\frac{\widetilde{K}(r^2,0,{\widetilde{c}},\nu_{\Theta}^3)}{{\widetilde{c}}^{2s}}=\ee(\tfrac 38)\frac{\sin(\pi r^2)}{\pi^2r\sinh(\pi r)}. 
	\end{align*} 
\end{proposition}

\begin{proof}
The case $r^2\in \Z_+$ follows from Proposition~\ref{prop_KloostermanSelbergZetaF_m=0_weightk3/2} and that for any $m\in \Z_+$, 
\begin{equation}\label{eqLimitofsinsinh}
    \lim_{r\rightarrow \sqrt m} \frac{\sin(\pi r^2)}{r\sinh(\pi r)}=0. 
\end{equation}
So we only need to prove the case $r^2\in \R_+\setminus \Z$. 

	By Lemma~\ref{lemmaExponential-interpolation} and Lemma~\ref{lemmaBernoulli-numbers-expansion}, we have
	\begin{align*}
		\sum_{2|c>0}&\frac{K(r^2,0,c,\nu_{\Theta}^3)}{c^{2s}}=\sum_{2|c>0}\sum_{\substack{k\in \Z\\|r^2-k|<c}}\frac{i\sin(\pi r^2)(-1)^k}{\ee(\frac{r^2-k}{2c})-1}\cdot\frac{S(k,0,c,\nu_{\Theta}^3)}{c^{2s+1}}\\
		&=\frac{\sin(\pi r^2)}{\pi}\sum_{k\in \Z} \frac{(-1)^k}{r^2-k} \sum_{2|c>|r^2-k|}\frac{S(k,0,c,\nu_{\Theta}^3)}{c^{2s}}\\
		&+\frac{\sin(\pi r^2)}{\pi}\sum_{k\in \Z} \sum_{\ell=1}^\infty B_\ell  \frac{(-1)^k(\pi i)^\ell(r^2-k)^\ell}{(r^2-k)\ell!} \sum_{2|c>|r^2-k|}\frac{S(k,0,c,\nu_{\Theta}^3)}{c^{2s+\ell}} \\
		&=:\Sigma_0+\Sigma_1. 
	\end{align*}
	We are going to prove that the summands for $k=-m^2\leq 0$ in $\Sigma_0$ contributes to the result in the proposition, and prove that $\Sigma_1$ and the remaining sums in $\Sigma_0$ are convergent when $s\rightarrow \frac 34^+ $. 
	
	We begin by estimating the innermost sum for $\ell\geq 0$. Let
	\[A(x)\defeq \sum_{2|c\leq x}\frac{S(k,0,c,\nu_{\Theta}^3)}c \quad \text{and} \quad \mathfrak{R}\defeq \ee(-\tfrac 18)\frac{8}{\pi^2}.  \]
	By Proposition~\ref{propSumofKloostermansums-weight32-one-of-mn0}, when $k=-m^2<0$ and $s>\frac 34$, we have
	\begin{align*}
		\sum_{2|c>r^2+m^2}& \frac{S(-m^2,0,c,\nu_{\Theta}^3)}{c^{2s+\ell}}=\int_{r^2+m^2}^\infty \frac{dA(t)}{t^{2s+\ell-1}}\\
		&=-\frac{\mathfrak{R}(r^2+m^2)^{\frac 12}+O_\ep((r^2+m^2)^{\frac 16+\ep}(m^2)^{\frac 14+\ep})}{(r^2+m^2)^{2s+\ell-1}}\\&-(1-2s-\ell)\int_{r^2+m^2}^\infty \frac{\mathfrak{R} t^{\frac 12}+O(t^{\frac 16+\ep} m^{\frac 14+\ep})}{t^{2s+\ell}} dt\\
		&=\frac 1{4s-3+2\ell} \cdot\frac {\mathfrak{R}}{(r^2+m^2)^{2s+\ell-\frac 32}}+O_\ep\(\frac 1{(r^2+m^2)^{2s+\ell-\frac {17}{12}-\ep}}\),  
	\end{align*}
	It is important to note that the last implied constant does not depend on $\ell\geq 0$, because it is uniformly bounded from above and below for $s\in [\frac 34,1]$ and $\ell\geq 0$. 
	
	Similarly, we also have
	\begin{align*}
		\sum_{2|c>r^2} \frac{S(0,0,c,\nu_{\Theta}^3)}{c^{2s+\ell}}&=\frac 1{4s-3+2\ell} \cdot\frac {\mathfrak{R}/2}{(r^2)^{2s+\ell-\frac 32}}+O_\ep\((r^2)^{-2s-\ell+\frac 76+\ep}+1\). 
	\end{align*}
	Additionally, when $k\neq 0$ nor a negative square,
	\begin{align*}
		\sum_{2|c>|r^2-k|} \frac{S(k,0,c,\nu_{\Theta}^3)}{c^{2s+\ell}}&\ll _\ep |r^2-k|^{-2s-\ell+\frac{17}{12}+\ep}  . 
	\end{align*}
	
	Therefore, we can apply the $\ell=0$ case in $\Sigma_0$ to get
	\begin{align*}
		\Sigma_0&=\frac{\sin(\pi r^2)\mathfrak{R}}{2\pi(4s-3)}\(\frac{1}{(r^2)^{2s-\frac 12}}+\sum_{m=1}^\infty \frac{2(-1)^m}{(r^2+m^2)^{2s-\frac 12}}\)\\
		&+\frac{\sin(\pi r^2)}{\pi}O_\ep\(\sum_{k\in \Z} |r^2-k|^{-2s+\frac 5{12}+\ep}\)\\
		&=\frac{\sin(\pi r^2)\mathfrak{R}}{2\pi(4s-3)}\(\frac{1}{(r^2)^{2s-\frac 12}}+\sum_{m=1}^\infty \frac{2(-1)^m}{(r^2+m^2)^{2s-\frac 12}}\)+O(1). 
	\end{align*}
	The last $O(1)$ part is bounded by an absolute constant because $2s-\frac 5{12}\geq \frac{13}{12}$ and $\frac{\sin(\pi r^2)}{\|r^2\|}\in [2,\pi]$ for all $r\in \R$. Since $r>0$ and
	\begin{equation}
		\frac{1}{r^2}+\sum_{m=1}^\infty \frac{2(-1)^m}{r^2+m^2}=\frac{\pi}{r\sinh(\pi r)}, 
	\end{equation}
	we get
	\[\lim_{s\rightarrow \frac 34^+} \frac{\Sigma_0}{\Gamma(s-\frac 34)}=\frac{\sin(\pi r^2)\mathfrak{R}}{8\pi}\cdot \frac{\pi}{r \sinh(\pi r)}=\ee(-\tfrac 18)\frac{\sin(\pi r^2)}{\pi^2 r \sinh(\pi r)}. \]
	
	For $\Sigma_1$, we have
	\begin{align*}
		\Sigma_1&=\frac{\sin(\pi r^2)\mathfrak{R}}{2\pi}\(\frac{1}{(r^2)^{2s-\frac 12}}+\sum_{m=1}^\infty \frac{2(-1)^m}{(r^2+m^2)^{2s-\frac 12}}\)\sum_{\ell=1}^\infty \frac{B_\ell (\pi i )^\ell}{\ell!(4s-3+2\ell)} \\
		&+\frac{\sin(\pi r^2)}{\pi}O_\ep\(\sum_{k\in \Z} |r^2-k|^{-2s+\frac {5}{12}+\ep}\sum_{\ell=1}^\infty \frac{B_\ell (\pi i )^\ell}{\ell!(4s-3+2\ell)}\). 
	\end{align*}
	It is then direct to conclude $\Sigma_1$ converges to a finite value for $s\rightarrow \frac 34^+$ because 
	\[\sum_{\ell=1}^\infty \frac{B_\ell (\pi i )^\ell}{\ell!(4s-3+2\ell)}\quad \text{converges for }s\in [\tfrac 34,1]. \]
	This is the reason why we separate the $\ell=0$ and $\ell\geq 1$ terms in Lemma~\ref{lemmaBernoulli-numbers-expansion}. We have finished the proof for the case $2|c$. 
	
	The other case $2\nmid {\widetilde{c}}$ follows by the same process applying Proposition~\ref{propSumofKloostermansums-weight32-one-of-mn0} and \eqref{eqpropSumofKloostermansums-weight32-bothmn0-oddd}. 
\end{proof}

\subsection{Case \texorpdfstring{$n<0$}{n<0}}

Similarly to the proof of Proposition~\ref{propEstimate-Sum-of-real-variable-Kloosterman-sums-weightk2-with-IJBessel}, the following proposition follows from Lemma~\ref{lemmaWeiltypebound-real-val-k3/2}, Proposition~\ref{propEstimate-sumofKlsum-real-val-k3/2-Weil1/2bound}, partial summation, Proposition~\ref{propEstimate-sumofKlsum-real-val-k3/2-QihangAsymptBound<1/4}, and by \cite[(10.29.1)]{dlmf}
	\begin{align}\label{eqEstimateIBessel2s-1}
    \begin{split}
	    \frac{d}{dt}I_{2s-1}\(\frac{2\pi |r^2n|^{\frac 12}}{t}\)&=-\frac{\pi |r^2n|^{\frac 12}}{t^2}\(I_{2s}\(\frac{2\pi |r^2n|^{\frac 12}}{t}\)+I_{2s-2}\(\frac{2\pi |r^2n|^{\frac 12}}{t}\)\)\\
&\ll\left\{\begin{array}{ll}
	    \displaystyle \frac{|r^2n|^{\frac 12}}{t^2} & \text{for }t\geq 2\pi |r^2n|^{\frac 12}, \vspace{3px} \\
		     \displaystyle |r^2n|^{\frac 12} e^{2\pi |r| |n|^{1/2}} & \text{for } 1\leq t\leq 2\pi |r^2n|^{\frac 12}.  
		\end{array} \right.
	\end{split}
	\end{align}. 
    We omit the proof here.
    
\begin{proposition}\label{propSum-real-val-Kloosterman-with-I2s-1-Anti-residue-EstimateNnot-m2}
    Let $r^2\in \R_+$, $n\in \Z$, $n<0$ and $s\in [\frac 34,1.001]$. 
    
    (1) If $n$ is not a negative square, the following sums are convergent and we have
    \begin{align*}
    \left.
        \begin{array}{r}
             \displaystyle \sum_{2|c>0}\frac{K(r^2,n,c,\nu_\Theta^3)}{c}I_{2s-1}\(\frac{2\pi |r^2n|^{\frac 12}}{c}\) \\
            \displaystyle \sum_{2\nmid {\widetilde{c}}>0}\frac{\widetilde{K}(r^2,n,{\widetilde{c}},\nu_\Theta^3)}{{\widetilde{c}}}I_{2s-1}\(\frac{2\pi |r^2n|^{\frac 12}}{{\widetilde{c}}}\) 
        \end{array}
        \right\} \ll (1+|r|^{2})|n|e^{2\pi |r| |n|^{1/2}}. 
    \end{align*}
    
    (2) If $r^2=q\in \Z_+$, $m\in \Z_+$ and $n=-m^2$, the following sums are convergent and we have
    \begin{align*}
        \begin{array}{r}
             \displaystyle \sum_{2|c>0}\frac{S(q,-m^2,c,\nu_\Theta^3)}{c}I_{2s-1}\(\frac{2\pi m\sqrt q}{c}\)  \\
            \displaystyle \sum_{2\nmid {\widetilde{c}}>0}\frac{S(q,-m^2,{\widetilde{c}},\nu_\Theta^3)}{{\widetilde{c}}}I_{2s-1}\(\frac{2\pi m\sqrt q}{{\widetilde{c}}}\) 
        \end{array}\ll |qn| e^{2\pi |qn|^{1/2}}. 
    \end{align*}
\end{proposition}
\begin{remark}
    In case (1), we need the estimates for 
    \[\sum_{2|c\leq x} \frac{S(k,n,c,\nu_\Theta^3)}{c}\quad\text{and}\quad \sum_{2\nmid {\widetilde{c}}\leq x}\frac{S(k,n,{\widetilde{c}},\nu_\Theta^3)}{{\widetilde{c}}}\]
    for all $k\in \Z$; in case (2), we keep $q> 0$. In both cases, we avoid to have both $k,n\in \{-t^2: t\in \Z\}$, which corresponds to the ``other $mn\neq 0$" cases in Proposition~\ref{propSumofKloostermansums-weight32-mnnot0} and Proposition~\ref{propSumofKloostermansums-weight32-mnnot0-oddd}. 
\end{remark}
\begin{proposition}\label{propSum-real-val-Kloosterman-with-I2s-1-residue}
	For $r\in \R_+$ and $m,n\in \Z$, $n<0$, let
    \[\mathscr{F}(r,n)=\left\{\begin{array}{ll}
        \displaystyle\ee(-\tfrac 18)\frac{4(rm)^{\frac 12}\sin(\pi r^2)}{\pi^2 r\sinh(\pi r)},  & \text{if }n=-m^2<0, \\
        0, & \text{if }n\text{ is not a negative square,}
    \end{array}\right. \]
    and $\widetilde{\mathscr{F}}(r,n)=-\mathscr{F}(r,n)$.
    Then we have the following limits: 
	\begin{align*}
		\lim_{s\rightarrow \frac 34^+} 
		\frac 1{\Gamma(s-\frac 34)} \sum_{2|c>0} \frac{K(r^2,n,c,\nu_{\Theta}^3)}{c} I_{2s-1}\(\frac{2\pi r|n|^{\frac 12}}c\)
		&= \mathscr{F}(r,n),\\
        \lim_{s\rightarrow \frac 34^+} 
		\frac 1{\Gamma(s-\frac 34)} \sum_{2\nmid {\widetilde{c}}>0} \frac{\widetilde{K}(r^2,n,{\widetilde{c}},\nu_{\Theta}^3)}{{\widetilde{c}}} I_{2s-1}\(\frac{2\pi r|n|^{\frac 12}}{\widetilde{c}}\)&=\widetilde{\mathscr{F}}(r,n). 
	\end{align*}
\end{proposition}

\begin{remark}
Heuristically, this proposition is similar to Proposition~\ref{propSum-real-val-Kloosterman-with-constantterm-residue}, by noticing that the coefficient of $x^{\frac 12}$ in Proposition~\ref{propSumofKloostermansums-weight32-mnnot0} is twice the value in Proposition~\ref{propSumofKloostermansums-weight32-one-of-mn0} and four times the value in \eqref{eqpropSumofKloostermansums-weight32-bothmn0}, and by \cite[(10.30.1)]{dlmf}:
\[I_{2s-1}\(\frac{2\pi r|n|^{\frac 12}}c\)\sim \frac 1{\Gamma(2s)}\(\frac{\pi r|n|^{\frac 12}}c\)^{2s-1}\quad \text{for }c\rightarrow \infty. \]
We give a detailed argument below.
\end{remark}

\begin{proof}
Recall \eqref{eqLimitofsinsinh}, hence the case for $n<0$ not a negative square or for $r^2\in \Z_+$ has been covered by Proposition~\ref{propSum-real-val-Kloosterman-with-I2s-1-Anti-residue-EstimateNnot-m2}. Here we only need to consider the case $n=-m^2<0$ and $r^2\in \R_+\setminus \Z$. We write $m>0$ for simplicity. 

By applying Lemma~\ref{lemmaExponential-interpolation} and Lemma~\ref{lemmaBernoulli-numbers-expansion}, we get
	\begin{align*}
		&\sum_{2|c>0}\frac{K(r^2,-m^2,c,\nu_{\Theta}^3)}{c}I_{2s-1}\( \frac{2\pi rm}c \)\\
		&=\frac{\sin(\pi r^2)}{\pi}\sum_{k\in \Z} \frac{(-1)^k}{r^2-k} \sum_{2|c>|r^2-k|}\frac{S(k,-m^2,c,\nu_{\Theta}^3)}{c}I_{2s-1}\(\frac{2\pi rm}c\)\\
		&+\frac{\sin(\pi r^2)}{\pi}\sum_{k\in \Z} \sum_{\ell=1}^\infty B_\ell  \frac{(-1)^k(\pi i)^\ell(r^2-k)^\ell}{(r^2-k)\ell!} \sum_{2|c>|r^2-k|}\frac{S(k,-m^2,c,\nu_{\Theta}^3)}{c^{1+\ell}} I_{2s-1}\(\frac{2\pi rm}c\)\\
		&=:\Sigma_0+\Sigma_1,  
	\end{align*}
    where $\Sigma_0$ is given by
    \begin{equation}\label{eq:weight3/2Sigma0part0I2s-1n<0n=-m2estimate}
        \Sigma_0=\frac{\sin(\pi r^2)}{\pi}\sum_{k\in \Z} \frac{(-1)^k}{r^2-k} \sum_{2|c>|r^2-k|}\frac{S(k,-m^2,c,\nu_{\Theta}^3)}{c}I_{2s-1}\(\frac{2\pi rm}c\). 
    \end{equation}
    We estimate the latter sum with the help of results in \S\ref{SectionKLweight3/2}.  
	Let
	\[A(x)\defeq \sum_{2|c\leq x}\frac{S(k,-m^2,c,\nu_{\Theta}^3)}c \quad \text{and} \quad \mathfrak{R}\defeq \ee(-\tfrac 18)\frac{16}{\pi^2}.  \]
	By \cite[(10.29.2)]{dlmf}, for $\alpha=2\pi rm$,
	\begin{equation}
		\frac {d}{dt} I_{2s-1}\(\frac{\alpha}t\)=-\frac {\alpha}{t^2}I_{2s}\(\frac{\alpha}t\)-\frac{2s-1}{t} I_{2s-1}\(\frac{\alpha}t\). 
	\end{equation}
	By \cite[(10.25.2)]{dlmf}, 
	\begin{equation}\label{eqBesselI-asympt-zsmall}
		I_{\nu}(x)=\frac{(x/2)^\nu}{\Gamma(\nu+1)}+O(x^{\nu+2}) \quad \text{for }\nu\in [0,3]\text{ and }x\in [0,1]. 
	\end{equation}
    By \cite[(10.30.4)]{dlmf} and \eqref{eqBesselI-asympt-zsmall},
    \begin{equation}\label{eqBesselI-asympt-zlarge}
        I_{\nu}(x)\ll e^x \quad \text{for }\nu\in [0,3] \text{ and }x\geq 0. 
    \end{equation}
	When $k=-q^2<0$ with $q>0$, note that $r^2+q^2\geq 2rq$ and $q\geq \pi m$ implies $r^2+q^2\geq 2\pi rm$. By applying Proposition~\ref{propSumofKloostermansums-weight32-mnnot0}, \eqref{eqBesselI-asympt-zsmall}, and \eqref{eqBesselI-asympt-zlarge}, we have that for $s\in(\frac 34,1.001]$, $2s-1>\frac 12$ and
	\begin{align*}
		&\sum_{2|c>r^2+q^2}\frac{S(-q^2,-m^2,c,\nu_{\Theta}^3)}{c}I_{2s-1}\(\frac{2\pi rm}c\)\\
        &=A(t)I_{2s-1}\(\frac{2\pi rm}t\)\Bigg|_{t=r^2+q^2}^\infty-\int_{r^2+q^2}^\infty A(t)\frac{d}{dt} I_{2s-1}\(\frac{2\pi rm}t\) dt\\
		&=-\frac{\mathfrak{R} (r^2+q^2)^{\frac 12}}{\Gamma(2s)} \(\frac{\pi rm}{r^2+q^2}\)^{2s-1}  +\frac{\mathfrak{R}(2s-1)}{\Gamma(2s)}\int_{r^2+q^2}^\infty {t^{-\frac 12}} \(\frac{\pi rm}t\)^{2s-1} dt\\
		&\ +\left\{\begin{array}{ll}
		     \displaystyle O_\ep\((1+r^3)m^3(r^2+q^2)^{\frac 54-2s+\ep}\),  &\ q\geq \pi m \\
		     \displaystyle O_\ep \((1+r^3)m^3e^{2\pi rm}\),&\ q\leq \pi m 
		\end{array}\right. \\
		&=\frac{\mathfrak{R}(\pi rm)^{2s-1}}{(4s-3)\Gamma(2s)}(r^2+q^2)^{\frac 32-2s}+\left\{\begin{array}{ll}
		     \displaystyle O_\ep\((1+r^3)m^3(r^2+q^2)^{\frac 54-2s+\ep}\),  &\ q\geq \pi m \\
		     \displaystyle O \((1+r^3)m^3e^{2\pi rm}\),&\ q\leq \pi m 
		\end{array}\right. 
	\end{align*}
    with crude bounds in the big-$O$ term. 
	Similarly, with Proposition~\ref{propSumofKloostermansums-weight32-one-of-mn0}, we have
	\begin{align*}
		\sum_{2|c>r^2}&\frac{S(0,-m^2,c,\nu_{\Theta}^3)}{c}I_{2s-1}\(\frac{2\pi rm}c\)\\
		&=\frac{(\pi rm)^{2s-1}\mathfrak{R}/2}{(4s-3)\Gamma(2s)}(r^2)^{\frac 32-2s}+ O \((1+r^3)m^3e^{2\pi rm}\). 
	\end{align*}
	For $k\neq 0$ nor negative square, we also have the crude bound
	\begin{align*}
		\sum_{2|c>|r^2-k|}\frac{S(k,-m^2,c,\nu_{\Theta}^3)}{c}I_{2s-1}\(\frac{2\pi rm}c\)=O_{\ep}\(\frac {(1+r^3)m^3e^{2\pi rm}}{|r^2-k|^{2s-\frac 54-\ep}}\). 
	\end{align*}
	Recall \eqref{eq:weight3/2Sigma0part0I2s-1n<0n=-m2estimate}. Combining the three equations above and noting that there are at most $O(m)$ of $q$'s such that $q\leq \pi m$, we conclude for $\Sigma_0$ that
	\begin{align*}
		\lim_{s\rightarrow \frac 34^+}\frac {\Sigma_0}{\Gamma(s-\frac 34)}&=\lim_{s\rightarrow \frac 34^+}\frac {(\pi rm)^{2s-1}\mathfrak{R}/2}{\Gamma(s-\frac 34)(4s-3)}\cdot  \frac{\sin(\pi r^2)}{\pi\Gamma(2s)}\(\frac 1{r^2}+\sum_{q=1}^\infty \frac{2(-1)^q}{r^2+q^2}\)\\
		&\ +\lim_{s\rightarrow \frac 34^+} \frac {\sin(\pi r^2)}{\Gamma(s-\frac 34)}O_\ep\(\sum_{k\in \Z} \frac{(1+r^3)m^4e^{2\pi rm}}{|r^2-k|^{2s-\frac 94-\ep}}\) \\
        &=\ee(-\tfrac 18)(rm)^{\frac 12} \frac{4\sin(\pi r^2)}{\pi^2 r \sinh(\pi r)}. 
	\end{align*}
	
	The similar method on $\Sigma_1$ in the proof of Proposition~\ref{propSum-real-val-Kloosterman-with-constantterm-residue} allows us to conclude that
	\begin{equation*}
		\Sigma_1=O((1+r^3)m^4e^{2\pi rm})\quad \text{and}\quad \lim_{s\rightarrow \frac 34^+} \frac {\Sigma_1}{\Gamma(s-\frac 34)}=0. 
	\end{equation*}
	The case for $2\nmid {\widetilde{c}}$ follows similarly. We have finished the proof. 
\end{proof}

Combining Proposition~\ref{propSum-real-val-Kloosterman-with-I2s-1-Anti-residue-EstimateNnot-m2} and the proof of Proposition~\ref{propSum-real-val-Kloosterman-with-I2s-1-residue}, we have the following crude estimate. 
\begin{proposition}\label{propSum-real-val-Kloosterman-with-I2s-1-estimatebound-withs}
    For $r^2\in \R_+$ and $m,n\in\Z$, $n<0$, let $\mathscr{F}(r,n)$ and $\widetilde{\mathscr{F}}(r,n)$ be the same as in Proposition~\ref{propSum-real-val-Kloosterman-with-I2s-1-residue}.  Then
    \begin{align*}
    &\left.
        \begin{array}{r}
             \displaystyle \frac 1{\Gamma(s-\frac 34)}\sum_{2|c>0}\frac{K(r^2,n,c,\nu_\Theta^3)}{c}I_{2s-1}\(\frac{2\pi |r^2n|^{\frac 12}}{c}\)-\mathscr{F}(r,n) \\
            \displaystyle  \frac 1{\Gamma(s-\frac 34)}\sum_{2\nmid {\widetilde{c}}>0}\frac{\widetilde{K}(r^2,n,{\widetilde{c}},\nu_\Theta^3)}{{\widetilde{c}}}I_{2s-1}\(\frac{2\pi |r^2n|^{\frac 12}}{{\widetilde{c}}}\)-\widetilde{\mathscr{F}}(r,n) 
        \end{array}
        \right\}\ll (1+|r|^{3})|n|^2 e^{2\pi |r| |n|^{1/2}}. 
    \end{align*}
\end{proposition}

\section{Proof of Theorem~\ref{theoremFourierExpansionForWeight3/2Ats=3/4} and Theorem~\ref{theoremMainDim3}}
\label{sectionProof-of-theorem-FourierExpansion-weight3/2-s3/4}

Based on the discussion of weight $3/2$ real-variable Kloosterman sums from the previous section, we are now ready to prove Theorem~\ref{theoremFourierExpansionForWeight3/2Ats=3/4}. 

\begin{proof}[Proof of Theorem~\ref{theoremFourierExpansionForWeight3/2Ats=3/4}]
Recall \eqref{eqWWhittaker-special-form} and the bound \eqref{eq:WhittakerboundE-piny} of the Whittaker function. 
	By combining Lemma~\ref{lemmaFourierExpansionAts}, Proposition~\ref{propSum-real-val-Kloosterman-with-J2s-1}, Proposition~\ref{propSum-real-val-Kloosterman-with-constantterm-residue}, and Proposition~\ref{propSum-real-val-Kloosterman-with-I2s-1-estimatebound-withs}, we find that $B_{3,n}(r;s,y)$ and $\widetilde{B}_{3,n}(r;s,y)$ have the claimed bounds for $n\neq 0$ and the limits of $B_{3,0}(r;s,y)$ and $\widetilde{B}_{3,0}(r;s,y)$ for $s\rightarrow \frac 34^+$ exist. 
    
    By dominated convergence theorem, the exponential decay $e^{-\pi |n| y}$ ensures that the Fourier expansions of $F_{\frac 32}(\tau;r,s)$ \eqref{eqFourierExpansionFAts-on-cusp-Infinity} and of $\widetilde{F}_{\frac 32}(\tau;r,s)$ \eqref{eqFourierExpansion-tildeFAts-on-cusp-infty} are absolutely convergent for $s\in [\frac 34,1.001]$. By \eqref{eqWWhittaker-special-form}, Proposition~\ref{propSum-real-val-Kloosterman-with-constantterm-residue}, Proposition~\ref{propSum-real-val-Kloosterman-with-I2s-1-residue} and by analytic continuation, we conclude 
	\[F_{\frac 32}(\tau;r,\tfrac 34)=\lim_{s\rightarrow \frac 34^+} F_{\frac 32}(\tau;r,s)\quad \text{and}\quad \widetilde{F}_{\frac 32}(\tau;r,\tfrac 34)=\lim_{s\rightarrow \frac 34^+} \widetilde{F}_{\frac 32}(\tau;r,s)\]
	and Theorem~\ref{theoremFourierExpansionForWeight3/2Ats=3/4} follows.
\end{proof}

The rest of this section is devoted to the proof of Theorem~\ref{theoremMainDim3}. We use Zagier's famous weight~$\frac 32$ mock modular form on $\Gamma_0(4)$ in \cite[Th\'eor\`eme~2]{Zagier1975weight3/2MF}. Here $q=e^{2\pi i \tau}$ and $\tau=x+iy\in \HH$. 
	\begin{theorem}[Zagier]
		Let $H(n)$ be the Hurwitz class number. Then the function 
		\begin{equation}\label{eq:Zagier'sMockModularForm}
		    \mathcal{H}(\tau)\defeq -\frac 1{12}+\sum_{n=1}^\infty H(n) q^n +\frac 1{8\pi \sqrt y} +\frac 1{4\sqrt\pi}\sum_{m=1}^\infty m\Gamma(-\tfrac 12,4\pi m^2 y)q^{-m^2}
		\end{equation}
		transform like a weight $3/2$ modular form on $\Gamma_0(4)$. 
	\end{theorem}
    More precisely, we have the following transformation formula for $\mathcal{H}$:
	\begin{equation}
		\mathcal{H}(\gamma \tau)=\nu_{\theta}(\gamma)^{3} (c\tau+d)^{\frac 32} \mathcal{H}(\tau) \quad \text{for }\gamma=\begin{psmallmatrix}
	a&b\\c&d
	\end{psmallmatrix}\in \Gamma_0(4). 
	\end{equation}
	By \cite[\S2.2]{HirzebruchZagier1976}, we write
	\begin{equation}\label{eqZagier3/2Holomorphic_part_and_nonholo_integral}
		\mathscr{H}_1(\tau)=-\frac 1{12}+\sum_{n=1}^{\infty} H(n)q^n\quad \text{and}\quad \mathcal{H}(\tau)-\mathscr{H}_1(\tau)=\frac{1+i}{16\pi} \psi(\tau)
	\end{equation}
	where
	\begin{equation}\label{eqZagier3/2_nonholo_integral_with_theta}
		\psi(\tau)=\int_{-\overline{\tau}}^{i\infty} \frac{\theta(v)}{(\tau+v)^{\frac 32}} dv
	\end{equation}
	and the integral is taken along the vertical path $v=2iuy-\tau$, $u\in [-1,\infty)$. Then the discussion after \cite[Theorem~2, Corollary]{HirzebruchZagier1976}, together with $\theta(-\frac1{4\tau})=\sqrt{-2i\tau}\theta(\tau)$, shows that
	\begin{align}
		\begin{split}
			(-2i\tau)^{-\frac 32}\mathscr{H}_1(-\tfrac 1{4\tau} ) +\mathscr{H}_1(\tau)&= -\frac 1{24}\theta(\tau)^3-\sqrt{\frac{\tau}{8i}} \int_{\R} \ee(\xi^2 \tau)\frac{1+\ee(2\xi\tau)}{1-\ee(2\xi\tau)} \xi d\xi,\\
			(-2i\tau)^{-\frac 32}\psi(-\tfrac 1{4\tau} ) +\psi(\tau)&=\int_{0}^{i\infty} \frac{\theta(u)}{(\tau+u)^{\frac 32}} du. 
		\end{split}
	\end{align}
	
\begin{lemma}\label{lemmaZagier3/2TwoIntegrals_Same}
	For $\tau\in \HH$, we have
	\[\frac{1+i}{16\pi }\int_{0}^{i\infty} \frac{\theta(u)}{(\tau+u)^{\frac 32}} du =\sqrt{\frac{\tau}{8i}}\int_\R \ee(\xi^2 \tau) \frac{1+\ee(2\xi \tau)}{1-\ee(2\xi \tau)} \xi d\xi. \]
\end{lemma}
\begin{proof}
	We start with the left hand side. Since $\theta(u)=1+2\sum_{n=1}^\infty \ee(n^2 u)$, we have
	\[\int_0^{i\infty} (\tau+u)^{-\frac 32} du=2\tau^{-\frac 12},\quad \int_0^{i\infty} \frac{2e^{2\pi i n^2 u}}{(\tau+u)^{\frac 32}} du=2n\sqrt{2\pi} \ee(-\tfrac 18)\ee(-n^2 \tau)\Gamma(-\tfrac 12,-2\pi i n^2 \tau). \]
	For the right side, note that the integrand is an even function of $\xi$, so we only need to deal with the part $\xi>0$. Since $\tau\in \HH$, in this range we have $|\ee(2\xi \tau)|<1$ and 
	\[\frac{1+\ee(2\xi\tau)}{1-\ee(2\xi\tau)}=1+2\sum_{n=1}^\infty \ee(2n\xi \tau). \]
	Hence we get
	\[\int_{\R} \ee(\xi^2 \tau) \xi d\xi=2\int_0^\infty \ee(\xi^2\tau) \xi d\xi=\frac{i}{2\pi \tau}\]
	and
	\[4\int_{0}^\infty e^{2\pi i \xi^2 \tau + 4\pi i n \xi \tau} \xi d\xi=ne^{-2\pi i n^2 \tau}(-2\pi i \tau)^{-\frac 12}\Gamma(-\tfrac 12,-2\pi i n^2 \tau). \]
	The lemma is proved by comparing the two expressions term by term. 
\end{proof}
	
Combining \eqref{eqZagier3/2Holomorphic_part_and_nonholo_integral}, \eqref{eqZagier3/2_nonholo_integral_with_theta} and Lemma~\ref{lemmaZagier3/2TwoIntegrals_Same}, we conclude that for $\tau\in \HH$, 
    \[
	(-2i\tau)^{-\frac 32}\mathcal{H}(-\tfrac 1{4\tau})+\mathcal{H}(\tau)=-\frac 1{24}\theta(\tau)^3.   
    \]
If we define
    \begin{align}
	\begin{split}
		\mathcal{H}^*(\tau)\defeq \mathcal{H}(\tfrac{\tau}2)+\frac 1{48} \Theta(\tau)^3=&-\frac 1{16}+\sum_{n=1}^\infty \(H(n)+\frac{r_3(n)}{48}\) e^{\pi i n \tau}\\
		&+\frac 1{8\pi}\(\sqrt{\frac 2y}+2\sqrt \pi \sum_{m=1}^\infty m\Gamma(-\tfrac 12,2\pi m^2 y)e^{-\pi i m^2 \tau}\),
	\end{split}
    \end{align}
then $\mathcal{H}^*$ satisfies
    \begin{equation}
	(-i\tau)^{-\frac 32}\mathcal{H}^*(-\tfrac 1\tau)+\mathcal{H^*}(\tau)=0. 
    \end{equation}

Therefore, we define
	\begin{align}\label{eq:G3/2inConstruction}
    \begin{split}
    G_{\frac 32}(\tau;r)&\defeq F_{\frac 32}(\tau;r,\tfrac 34)-\frac{8\sin(\pi r^2)}{r\sinh(\pi r)} \mathcal{H}^*(\tau)\\
    &=\frac{\sin(\pi r^2)}{2r\sinh(\pi r)}-\frac{\sin(\pi r^2)}{r\sinh(\pi r)}\sum_{n=1}^\infty e^{\pi i n \tau}\(8H(n)+\frac {r_3(n)}{6}\)\\
    &+\pi \ee(\tfrac 18)\sum_{n=1}^\infty e^{\pi i n \tau} \(\frac{n}{r^2}\)^{\frac 14} \sum_{2|c>0} \frac{K(r^2,n,c,\nu_{\Theta}^3)}{c} J_{\frac 12}\(\frac{2\pi |r^2n|^{\frac 12}}{c}\)
    \end{split}
	\end{align}
and
	\begin{align}\label{eq:tildeG3/2inConstruction}
    \begin{split}
    \widetilde{G}_{\frac 32}(\tau;r)&\defeq \widetilde{F}_{\frac 32}(\tau;r,\tfrac 34)-\frac{8\sin(\pi r^2)}{r\sinh(\pi r)} \mathcal{H}^*(\tau)\\
    &=\frac{\sin(\pi r^2)}{2r\sinh(\pi r)}-\frac{\sin(\pi r^2)}{r\sinh(\pi r)}\sum_{n=1}^\infty e^{\pi i n \tau}\(8H(n)+\frac {r_3(n)}{6}\)\\
    &+\pi \ee(-\tfrac 38)\sum_{n=1}^\infty e^{\pi i n \tau} \(\frac{n}{r^2}\)^{\frac 14} \sum_{2\nmid d>0} \frac{\widetilde{K}(r^2,n,d,\nu_{\Theta}^3)}{d} J_{\frac 12}\(\frac{2\pi |r^2n|^{\frac 12}}{d}\). 
    \end{split}
	\end{align}
These functions satisfy the functional equation
	\begin{equation}
    G_{\frac 32}(\tau;r)+(-i\tau)^{-\frac 32}\widetilde{G}_{\frac 32}(-\tfrac 1\tau;r) =e^{\pi i r^2 \tau}. 
	\end{equation}
Here we write the linear combination $8H(n)+\frac{r_3(n)}6$ because it is always an integer. 

\begin{proof}[Proof of Theorem~\ref{theoremMainDim3}]\label{ProofofTheoremMainDim3}
The functions $b_{3,n}(r)$ and $\widetilde{b}_{3,n}(r)$ can be read from \eqref{eq:G3/2inConstruction} and \eqref{eq:tildeG3/2inConstruction}, respectively, with the help of \cite[(10.16.1)]{dlmf}:
    \[J_{\frac 12}(z)=\(\frac 2{\pi z}\)^{\frac 12} \sin(z). \] 
It remains to compute the values 
    \begin{equation}
         B_{3,n}(0)=\lim_{r\rightarrow 0^+} B_{3,n}(r)\quad\text{and}\quad \widetilde{B}_{3,n}(0)=\lim_{r\rightarrow 0^+} \widetilde{B}_{3,n}(r)\quad \text{for }n\geq 1
    \end{equation}
because by Proposition~\ref{propSum-real-val-Kloosterman-with-J2s-1} and the dominated convergence theorem, 
\begin{align*}
    G_{\frac 32}(\tau;0)&=\lim_{r\rightarrow0}G_{\frac 32}(\tau;r)=\sum_{n=0}^\infty e^{\pi in \tau}\lim_{r\rightarrow 0} b_{3,n}(r),\\\widetilde{G}_{\frac 32}(\tau;0)&=\lim_{r\rightarrow0}{G}_{\frac 32}(\tau;r)=\sum_{n=0}^\infty e^{\pi in \tau}\lim_{r\rightarrow 0} \widetilde b_{3,n}(r)
\end{align*}
Similarly to \eqref{eq:b4nANDtildeb4nAS_KLsumr=0}, by $\sin(z)=z-z^3/3!+\cdots$, we claim that
    \begin{align}\label{eqB3n0B3n0tildeinS(0,n)}
    \begin{split}
    B_{3,n}(0)&=2\pi \ee(\tfrac 18)\sqrt{n} \sum_{2|c>0}\frac{S(0,n,c,\nu_\Theta^3)}{c^{3/2}}\\
    \widetilde{B}_{3,n}(0)&=2\pi \ee(-\tfrac 38)\sqrt{n} \sum_{2\nmid \widetilde{c}>0}\frac{S(0,n,\widetilde{c},\nu_\Theta^3)}{\widetilde{c}^{3/2}}. \\
    \end{split}
    \end{align}
This can be directly proved by applying Proposition~\ref{propEstimate-sumofKlsum-real-val-k3/2-QihangAsymptBound<1/4} and Proposition~\ref{propSum-real-val-Kloosterman-with-J2s-1}, as well as \eqref{eq:weight3/2SumKLsumJ2s-1Cauchy} and Proposition~\ref{propSumofKloostermansums-weight32-one-of-mn0} in order to show
\[\left. \begin{array}{r}
     \displaystyle\sum_{2|c>x}\frac{K(r^2,n,c,\nu_\Theta^3)}{c^{3/2}} \\
     \displaystyle\sum_{2|c>x} \frac{K(0,n,c,\nu_\Theta^3)}{c^{3/2}}
\end{array}\right\} \ll x^{-\frac 14+\ep}n^{\frac 14+\ep}\quad \text{for }x>\max(1,r^2, 2\pi |r|\sqrt{n})\]
which implies
\[\lim_{r\rightarrow 0}\sum_{2|c} \frac{K(r^2,n,c,\nu_\Theta^3)}{c^{3/2}}=\sum_{2|c} \frac{S(0,n,c,\nu_\Theta^3)}{c^{3/2}}\]
and similarly for the case $2\nmid \widetilde{c}$. 

We can evaluate the sums of Kloosterman sums in \eqref{eqB3n0B3n0tildeinS(0,n)} by recalling the calculations in \S\ref{subsectionWeight3/2mn=0alphaApns}. Comparing with the quantity $B_k(n)$ defined by Bateman (taking $s=3$) \cite[(1.02)]{Bateman1951ThreeSquares}, one can find that
    \[B_{2^\nu}(n)=2^{-\frac 32\nu}\ee(\tfrac 38)\overline{\alpha_3(2^{\nu+1},n)},\quad B_{p^\nu}(n)=p^{-\frac 32 \nu}\overline{\alpha_3(p^\nu,n)},\ p>2. \]
Moreover, for $n\geq 1$ and $4^\alpha\|n$ for $\alpha=\floor{\nu_2(n)/2}$, the quantity $\chi_2(n)$ defined and calculated at \cite[under (4.06), also Theorem~B]{Bateman1951ThreeSquares} gives
    \begin{equation}
    1-\chi_2(n)=2^{\frac 32} \ee(\tfrac 18)A_3(2,n,\tfrac 34)=\left\{\begin{array}{ll}
         1,& \text{if } n/4^\alpha\equiv 7\Mod 8,\\
         1-2^{-\alpha},& \text{if }n/4^{\alpha}\equiv 3\Mod 8,  \\
         1-3\times 2^{-\alpha-1}, & \text{if }n/4^\alpha\equiv 1,2,5,6 \Mod 8.
    \end{array}\right. 
    \end{equation}
Additionally, \cite[Theorem~B]{Bateman1951ThreeSquares} gives that
    \begin{align}
    r_3(n)&=2\pi\sqrt n\chi_2(n)\cdot \frac{8K(-4n)}{\pi^2}\prod_{p^2|n}
    \(\sum_{j=1}^{\floor{\frac{\nu_p(n)}2}-1} \frac 1{p^j}+\frac{p^{-\floor{\frac{\nu_p(n)}2}}}{1-\(\frac{-n/p^{2\floor{\nu_p(n)/2}}}p\)\frac 1p}\)\\
    &=2\pi \sqrt n\chi_2(n)\prod_{p>2}A_3(p,n,\tfrac 34), 
    \end{align}
where $K(-4n)=\sum_{m=1}^\infty (\frac{-4n}m)\frac 1m $ is convergent.

Therefore, by the discussion above, as well as \eqref{eqKLsum-n0weight3-into-prods}-\eqref{eqKloostermanzeta-n0-product-asA-oddd}, for $n\geq 1$, if $n\neq 4^\alpha(8\beta+7)$ for any $\alpha,\beta\in \Z$, we have
    \begin{equation}
    B_{3,n}(0)=r_3(n)\cdot \frac{1-\chi_2(n)}{\chi_2(n)}\quad\text{and}\quad 
    \widetilde{B}_{3,n}(0)=\frac{r_3(n)}{\chi_2(n)}=B_{3,n}(0)+r_3(n); 
    \end{equation}
if $n=4^\alpha(8\beta+7)$ for some $\alpha,\beta\in \Z$, we have
    \begin{align}\label{eq:B3n0tildeB3n0WhenNNotr3Reprsted}
    \begin{split}
    B_{3,n}(0)&=\frac{16}{\pi}\sqrt n\(1-\chi_2(n)\)K(-4n)\prod_{p^2|n}
        \(\sum_{j=1}^{\floor{\frac{\nu_p(n)}2}-1} \frac 1{p^j}+\frac{p^{-\floor{\frac{\nu_p(n)}2}}}{1-\(\frac{-n/p^{2\floor{\nu_p(n)/2}}}p\)\frac 1p}\)\\
    \widetilde{B}_{3,n}(0)&=B_{3,n}(0)/\(1-\chi_2(n)\)=B_{3,n}(0). 
    \end{split}
    \end{align}
The formulas in the theorem about $B_{3,n}(0)$, $\widetilde{B}_{3,n}(0)$, $b_{3,n}(0)$, and $\widetilde{b}_{3,n}(0)$ can then be obtained from the relations between $r_3(n)$ and the Hurwitz class number $H(n)$ (see e.g. \cite[\S4.8]{Grosswald1985}).
    
This finishes the proof of Theorem~\ref{theoremMainDim3}. 
\end{proof}

\section{Concluding remarks}
\label{SectionDiscussion}

One of the corollaries of Theorem~\ref{theoremMainDim4} and Theorem~\ref{theoremMainDim3} is that the functions $B_{d,n}(r)$ satisfy the following:
    \begin{align*}
    \begin{array}{l}
          |B_{4,n}(r)|, |\widetilde{B}_{4,n}(r)| \ll_{\ep,c}   n^{3/4+\ep}  \\
         |B_{3,n}(r)|, |\widetilde{B}_{3,n}(r)| \ll_{\ep,c}    n^{1/2+\ep}
    \end{array}
   \quad \text{if }r^2\geq c>0,\quad \ep>0. 
    \end{align*}
In particular, these bounds give improved control over the coefficients $a_{d,n}(r)$, $\widetilde{a}_{d,n}(r)$ of the radial Fourier interpolation formulas~\eqref{eq:radialmain} for $d=3,4$. Optimistically, we conjecture the following stronger bounds to hold. 
\begin{conjecture} The functions $B_{d,n}(r)$ from Theorem~\ref{theoremMainDim3} and Theorem~\ref{theoremMainDim4} satisfy the following bounds:
    \begin{align*}
    \begin{array}{l}
          |B_{4,n}(r)|, |\widetilde{B}_{4,n}(r)| \ll_{\ep,c}   n^{1/2+\ep}  \\
         |B_{3,n}(r)|, |\widetilde{B}_{3,n}(r)| \ll_{\ep,c}    n^{1/4+\ep}
    \end{array}
   \quad \text{if }r^2\geq c>0,\quad \ep>0.
    \end{align*}
\end{conjecture}
The (hopeful) expectation is that, at least for fixed $r>0$, the coefficients $B_{d,n}(r)$ and $\widetilde{B}_{d,n}(r)$ should grow like coefficients of a weight $d/2$ cusp form, so the above conjecture corresponds roughly to the bound from the Ramanujan-Petersson conjecture. They are consistent with the numerics as far as we can check, and they can also be deduced assuming the Linnik-Selberg conjecture about sums of Kloosterman sums (see \cite[(7)]{SarnakTsimerman09} for $d=4$, $k=2$ and \cite[Conjecture~1.4]{QihangSecondAsympt} for $d=3$, $k=\frac 32$). 

We plan to generalize Theorem~\ref{theoremMainDim4} and Theorem~\ref{theoremMainDim3} to dimensions $d=1,2$ in a future work. In particular, we expect to recover the interpolation formula for even Schwartz functions in dimension 1 \cite[Theorem~1]{RadchenkoViazovska2019}, obtaining an explicit formula for the interpolation basis functions $a_n(x)$ in terms of real-variable Kloosterman sums.

\bibliographystyle{alpha}
\bibliography{allrefs}
\end{document}